\documentclass[letterpaper,11 pt]{article}
\usepackage{amsfonts,amsmath,graphicx,fullpage,bbm,amsthm}
\usepackage{amssymb,multirow,verbatim}
\usepackage{acronym,wrapfig,plain,mathrsfs,enumerate,relsize,color}
\newtheorem{theorem}{Theorem}

\newtheorem{assumption}{Assumption}

\newtheorem{corollary}{Corollary}

\newtheorem{definition}{Definition}

\newtheorem{lemma}{Lemma}

\newtheorem{proposition}{Proposition}
\newtheorem{remark}{Remark}

\usepackage{algorithm}
\usepackage{subfig}
\usepackage{algorithmic}
\usepackage{boxedminipage}
\usepackage[justification=centering]{caption}
\def\bko{{\rm 1\kern-.17em l}}
\usepackage{wrapfig}
\usepackage{lipsum}

\newcommand{\an}[1]{{\color{black}#1}}
\newcommand{\fy}[1]{{\color{black}#1}}
\newcommand{\us}[1]{{\color{black}#1}}
\newcommand{\uvs}[1]{{\color{black}#1}}
\newcommand{\fytwo}[1]{{\color{black}#1}}

\newcommand{\fyRev}[1]{{\color{black}#1}}

\def\be{\begin{enumerate}}
\def\ee{\end{enumerate}}

\def\re{\mathbb{R}}

 \newcommand{\remove}[1]{}
\newcommand{\EXP}[1]{\mathsf{E}\!\left[#1\right] }

\def\sF{\mathcal{F}}

\def\Real{\mathbb{R}}
\def\g{\gamma}

\def\e{\epsilon}
\def\a{\alpha}

\begin{document}
\title{On Smoothing, Regularization and Averaging in Stochastic Approximation Methods for Stochastic Variational Inequalities
}

% Block of authors and their affiliations starts here:
% NOTE: Authors with same affiliation, if the order of authors allows,
%   should be entered in ONE field, separated by a comma.
%   \EMAIL field can be repeated if more than one author
%\ARTICLEAUTHORS{%
\author{Farzad Yousefian\thanks{Industrial \& Manufacturing Engineering,
	Pennsylvania State University,  University Park, State College, PA
		16802, USA,  \texttt{szy5@psu.edu}}, {Angelia
			Nedi{\'c}}\thanks{Industrial \& Enterprise Systems
				Engineering Department, University of Illinois at
					Urbana-Champaign, Urbana, IL 61801, USA,
				\texttt{angelia@illinois.edu}}, and 
Uday V. Shanbhag\thanks{Industrial \& Manufacturing Engineering,
	Pennsylvania State University,  University Park, State College, PA
		16802, USA, \texttt{udaybag@psu.edu};
Nedi\'{c} and
Shanbhag gratefully acknowledge the support of the NSF through awards
CMMI 0948905 ARRA (Nedi\'{c} and Shanbhag), CMMI-0742538 (Nedi\'{c}) and 
CMMI-1246887 (Shanbhag). A part of this paper has appeared in~\cite{Farzad-WSC13}.}
} % end of the block

\setlength{\textheight}{9in} \setlength{\topmargin}{0in}\setlength{\headheight}{0in}\setlength{\headsep}{0in}
\setlength{\textwidth}{6.5in} \setlength{\oddsidemargin}{0in}\setlength{\marginparsep}{0in}

\maketitle
\thispagestyle{empty}

\begin{abstract} 
\uvs{Traditionally, stochastic approximation schemes for SVIs have relied on
strong monotonicity and Lipschitzian properties of the underlying map. In
contrast, we consider monotone stochastic variational inequality (SVI) problems
where the strong monotonicity and Lipschitzian assumptions on the mappings are
weakened.  In the first part of the paper, to address such
shortcomings,  a regularized smoothed SA (RSSA) scheme is developed wherein
the stepsize, smoothing, and regularization parameters are diminishing
sequences} updated after every iteration.
Under suitable assumptions on the sequences, we show that the algorithm
generates iterates that converge to a solution in an almost sure
sense, \uvs{extending the \an{results} in~\cite{koshal12regularized} to
the non-Lipschitzian regime}. Additionally, we provide rate estimates that
relate iterates to their counterparts derived from a smoothed Tikhonov
trajectory associated with a deterministic problem. \uvs{Motivated by the need to
develop non-asymptotic rate statements, in the second part of the paper, we develop a variant of the RSSA
scheme, denoted by aRSSA$_r$, in which we employ a  weighted iterate-averaging,
{\em parametrized} by a scalar $r$ where $r = 1$ provides us with the standard
averaging scheme. We make several contributions in this context:} First,
we show that the gap function associated with the sequences by the \fyRev{aRSSA$_r$} scheme
tends to zero \an{in both an almost sure and an expected-value sense,} 
\uvs{when \an{the} parameter sequences are chosen appropriately}. Second, we show
that the gap function associated with the averaged sequence diminishes to zero
at the optimal rate ${\cal O}({1}\slash{\sqrt{K}})$ after $K$ steps when smoothing and
regularization are suppressed \uvs{and $r < 1$, \an{thus} {\em improving} the rate
statement for the standard averaging which admits a rate of ${\cal
O}(\ln(K)/\sqrt{K})$}. 
\uvs{Third, we develop a window-based variant of this
scheme that also displays the optimal rate for $r < 1$. Notably, we prove the
superiority of the scheme with $r < 1$ with its counterpart with $r=1$ in terms
of the constant factor of the error bound when the size of the averaging window
is sufficiently large. Our numerical study on  a stochastic Nash-Cournot game
provides two empirically driven insights: (i) Iterates generated by the RSSA
scheme under almost sure convergence guarantees rather than
convergence in the mean display far smaller variances in terms of gap function;
and (ii) In the window-based variants of the RSSA scheme, choosing $r<1$ leads to lower average gap
function values compared to $r = 1$, particularly when the window sizes are
large.}   \end{abstract}

\section{Introduction}\label{sec:intro}
%\bigfirstletter 
{Given a set $X$ and a mapping $F:X \rightarrow \mathbb{R}^n$, a
variational inequality (VI) problem, denoted by VI$(X,F)$, requires a
vector $x^* \in X$ such that $F(x^*)^T(x-x^*)\geq 0$ for all $x \in
X$.} Over the last several decades, variational inequality problems have
been applied in capturing a wide range of optimization and equilibrium
problems in engineering, economics, game theory, and finance
(cf.~\cite{facchinei02finite,Rockafellar98}).  In this paper, we
consider a stochastic generalization of this problem
{where} the components of the {mapping $F$} are expectation-valued.
{More precisely, we} are interested in solving VI$(X,F)$ where
mapping $F:X
\rightarrow \mathbb{R}^n$ represents the expected value of a stochastic
mapping ${\Phi}:X\times\fyRev{\Real^d} \rightarrow \mathbb{R}^n$, i.e.,
		$F(x)\triangleq  {\EXP{\Phi(x,\xi\fytwo{(\omega)})}}$ where $\xi
		: \Omega \to \Real^d$ is a
$d-$dimensional random
variable and  $(\Omega, {\cal F}, \mathbb{P})$ represents the
probability space. \us{Then $x^*
\in X$ solves VI$(X,F)$} if
\begin{align}\label{def:SVI}
{\EXP{\Phi(x^*,\xi\fytwo{(\omega)})}^T(x-x^*)} \geq 0, \qquad \hbox{for any } x \in X.
\end{align}
For brevity, throughout this paper, $\xi$ is used to denote $\xi(\omega)$. 

The stochastic variational inequality (SVI) problem~\eqref{def:SVI} assumes relevance
in a range of settings. Such models have immediate
utility as they represent the (sufficient) optimality conditions of
stochastic convex optimization
problems~\cite{birge97introduction,SPbook} as well as the equilibrium
conditions of stochastic convex Nash
games~\cite{ravat11characterization,Nem11,kannan13addressing}.  \uvs{These}
models  find further applicability when the evaluation of the map is
corrupted by errors.
	%(\ref{def:SVI}) is applied in situations where evaluations of the
%	mapping is associated with errors and therefore, the uncertain value
%	of the mapping $\Phi$ is replaced by an unbiased estimate of the
%	mapping $F$. Problems arising from stochastic convex optimization
%	(\cite{nemirovski_robust_2009}, \cite{Sundhar08}), stochastic Nash
%	equilibrium problems (\cite{ravat11characterization}, \cite{Nem11}),
%				and large-scale machine learning (\cite{Nocedal12}) are
%					among the instances modeled using this setting.}
While SVIs represent a natural extension of their
	deterministic counterparts, generally deterministic schemes cannot
	be applied directly, particularly when the expectation cannot be
	evaluated efficiently or the underlying distribution
			$\mathbb{P}$ is unavailable.
%Recently, there has been significant interest
%in the analysis~\cite{ravat09characterization} and
%solution~\cite{Ruz03,Gurkan99,Houyuan08,koshal10single} of such
%problems.  
%. In this work, we refer to such problems as
%stochastic variational inequalities (SVI). SVIs have been used in the
%literature to capture different problems arising from stochastic
%optimization and stochastic Nash games over continuous strategy sets.
%The characterization of solution sets of stochastic Nash games has been
%analyzed for smooth and nonsmooth settings
%\cite{ravat09characterization}. 
Our interest {lies in developing schemes that produce asymptotically
exact solutions}. A \uvs{popular} avenue for solving
SVI problems is \uvs{through} Monte-Carlo sampling methods. Of these, sample
average approximation methods (SAA) and stochastic approximation methods
(SA) are \uvs{amongst} the well-known approaches. In the context of SAA methods
for solving stochastic {optimization problems}, asymptotic
convergence of estimators and exponential rate analysis have been
studied comprehensively by Shapiro~\cite{shap03sampling}. Extensions to
SVI problems have been provided by Xu in~\cite{Xu10}, where the
exponential convergence rate of the estimators was established under
more general assumptions on sampling while confidence statements for such
problems are examined in~\cite{lu14symmetric,lu13confidence}. 
%{|m{2cm}|m{3.2cm}|m{3.4cm}|m{3.4cm}|m{2.8cm}|  } 
\begin{table}[htb] 
%\vspace{-0.05in} 
\tiny 
\centering 
\begin{tabular}{|c|c|c|c|c|c|c|c|} 
\hline 
\textbf{Ref.} &     \textbf{Algorithm}& \textbf{Monotonicity}&
\textbf{Lipscitz} &
{\bf Metric} & {\bf Convergence} & \textbf{Rate}
\\ 
\hline 
\hline
  \cite{Houyuan08} & Standard &S & Y & Soln & a.s. & -- \\
    \hline
\cite{nemirovski_robust_2009} & Mirror-descent (averaging, \uvs{opt. problems}) & Y & N &  Gap & mean &${\cal O}\left(\frac{1}{\sqrt{k}}\right)$\\
    \hline
\cite{Jud11} &  Mirror-Prox  (averaging) &  Y &  \uvs{N} & Gap* & mean &${\cal O}\left(\frac{1}{\sqrt{k}}\right)$\\
    \hline
      \cite{koshal12regularized}& Regularized SA & Y & Y& Soln & a.s. &-- \\
    \hline
\cite{Farzad03} & Self-tuned smoothing SA & S & N & Soln* & a.s. &${\cal O}\left(\frac{1}{{k}}\right)$ \\
    \hline
   \cite{wang14}  &Incremental constraint projection & S & {Y} & Soln & a.s. & ${\cal O}\left(\frac{1}{{k}}\right)$\\
    \hline\hline
     
                \bf & Regularized smoothing SA (RSSA) & Y & N & Soln &
				a.s. & -- \\
	{this paper}	      &Averaging ( {aSA$_r$}) & Y & N & Gap & mean &${\cal O}\left(\frac{1}{\sqrt{k}}\right)$\\
		      &Window-based  Averaging (aRSSA$_{\ell,r}$) & Y & N & Gap & mean &${\cal O}\left(\frac{1}{\sqrt{k}}\right)$\\
    \hline
\end{tabular} 
\caption{A comparison of stochastic approximation methods for solving
	SVIs (S: Strongly Monotone, *: Approximate solution, a.s.: almost
			sure, Y: Yes, N: No) }
\label{tab:SVI_algorithms} 
\end{table}

A different tack is adopted by stochastic approximation (SA) schemes
which  were  first introduced by Robbins and Monro~\cite{robbins51sa} for
stochastic root-finding problems; \uvs{this class of problems} require an $x^* \in \Real^n$ such
that $\EXP{g(x,\xi)}=0$, where $\xi:\Omega \rightarrow \Real^d$ is a
random variable \uvs{and} $g(\cdot,\xi): \Real^n \rightarrow \Real^n$ is a
continuous map for any realization of $\xi$.
The
standard SA scheme is based on the \uvs{iteration} {$ x_{k+1}
:=x_k-\g_k g(x_k,\xi_k)$ for all $k\ge0$, } where $\g_k>0$ \us{denotes}
the stepsize while $\xi_k$ represents a realization of  a random
variable $\xi$ at the $k$-th iteration.  
 SA schemes have been applied
extensively in solving convex stochastic optimization
problems~\cite{Ermoliev83,Kush03,Farzad1,Zeevi11,Ghad12}. There has been
a surge of interest in the solution of SVIs via stochastic approximation
schemes. Amongst the earliest work was by Jiang and Xu~\cite{Houyuan08},
who considered SVIs with strongly monotone and Lipschitz
continuous maps over a closed and convex set and proved that the
sequence of solution iterates converge to the unique solution in an
almost sure sense.  In an extension of {that} work, motivated
by Tikhonov regularization scheme, a regularized SA method was developed
for solving SVIs with a merely monotone but continuous
mapping~\cite{koshal12regularized}.  \us{A comprehensive summary
	of the various schemes for solving SVIs via SA schemes is provided in
Table~\ref{tab:SVI_algorithms}.} \uvs{We emphasize that our focus is on
developing asymptotic and rate statements for stochastic variational inequality
problems in which the map is not necessarily Lipschitz continuous, a
problem that has been considered only in~\cite{Jud11}. Next, we provide a brief summary of the paper.}

\subsection{Motivation and summary}
  
 The first part of this paper is motivated by the need to weaken
 Lipschitzian and strong monotonicity requirements. Employing a
 smoothing technique \uvs{introduced by Steklov~\cite{steklov1} and utilized for the solution of stochastic
optimization problems
\cite{Bertsekas73,norkin93optimization,DeFarias08,Duchi12}}, our first goal is to weaken the typical
 conditions for almost sure convergence of SA methods by allowing for a
 \us{non-Lipschitzian} mapping.  
%{Contending with nonsmothness in deterministic settings has been through
%   introducing a sequence of smooth and approximate problems
%		   \cite{facchinei96smoothing} or using conjugate and
%		   proximal functions \cite{Nesterov05}.} 
%However, in stochastic regimes, a challenge in applying such schemes is that a closed form of the stochastic function is required while such information may not be available. 
%by Steklov  , employed for the solution of  
\uvs{Recall that given} a convex
function $\uvs{f}:\Real^n \rightarrow \Real$ and a random variable $\omega$
with probability distribution $P(\omega)$, the function $\hat f$ \an{given by}  
$\hat f(x) \triangleq \int_{\Real^n}f(x+\omega)P(\omega)d\omega=\EXP{f(x+\omega)}$
\an{is {\it differentiable}.}
\uvs{Incorporating this
technique within SA schemes} \uvs{allows for addressing} nonsmoothness in 
stochastic convex optimization problems~\cite{Farzad1}
and \uvs{the absence of Lipschitzian properties on the maps in monotone}
SVIs~\cite{Farzad03}. The main limitation of the proposed smoothing SA
scheme in~\cite{Farzad1,Farzad03} is that the smoothing parameter is assumed to
be fixed through implementing the SA method. As a consequence, the generated
sequence by the smoothing SA scheme is shown to converge almost surely to the
optimal solution of the approximate problem, i.e.,  $\min_{x\in X} \hat f(x)$,
where $\hat f$ is the smoothed objective function. 
\uvs{A logical extension to achieve convergence to an optimal solution lies in ``adaptive smoothing''
schemes} where the smoothing parameter is reduced adaptively to zero. \uvs{For
instance in~\cite{Polak84}, a deterministic nondifferentiable function is
minimized by solving a sequence of smoothed problems. Unfortunately, such an
approach in the current setting requires getting solutions to a sequence of
{\em smoothed} stochastic variational inequality problems, each of which
requires resolution via simulation-based schemes. An implementable
``single-loop'' scheme requires updating the smoothing parameter after every
step. In fact,  Xu~\cite{Xu01} utilizes an analogous approach in the context of
Newton schemes for the solution of nonsmooth equations where a smoothing scheme
is utilized for computing an element of the Clarke generalized Jacobian and the
smoothing parameter is updated after {\em every} Newton step. Our goal lies in
introducing a single-loop SA scheme where the steplength and smoothing
parameters are updated at each iteration.}

\uvs{Our second goal lies in weakening the strong monotonicity requirement on the mapping. We achieve this by utilizing a
		  regularization term, inspired by Tikhonov
		  regularization~\cite{facchinei02finite} and their iterative
		  counterparts~\cite{golshtein89modified,kannan10online,koshal12regularized}.
In the resulting scheme, referred to as  a {\em regularized smoothed
	stochastic approximation} (RSSA) scheme, {\em both} the smoothing parameter and the regularization
parameter are updated after every iteration and are driven to zero in
the limit, in sharp contrast with our prior work~\cite{Farzad1,Farzad03}, where the smoothing parameter is assumed to be fixed. This allows for proving asymptotic convergence to the solution set of the
 original problem, rather than an approximate problem.
%  {The key distinction with prior schemes is that the RSSA scheme can cope with merely monotone SVIs without a Lipschitzian assumption on the map and are equipped with {almost sure} asymptotic convergence guarantees.} 
Unfortunately, we cannot derive non-asymptotic rate statements without
reverting to averaging, \an{which is} the focus of the second part of the
paper.}  

In the second part of this paper,  motivated by the need to derive rate
statements, we consider averaged counterparts of the RSSA scheme. Averaging approaches have proved useful in
developing rate statements (cf.~\cite{Polyak92,Jud11})
%the limitations in the
%Prox-type methods. Such methods were first introduced by Nemirovski (\cite{Nemir04}) for solving VIs with monotone and Lipschitz
%continuous maps. Since then, they have been applied to address
%different problems in convex optimization and variational inequalities
%(\cite{nemirovski_robust_2009, Jud11, lan12}). 
%Recently, Juditsky et. al.~(\cite{Jud11}) introduced the stochastic
%Mirror-Prox (SMP) algorithm for solving stochastic VIs with
%non-Lipschitzian but bounded mappings where an extra-gradient type
%method is employed and the optimal rate of convergence is attained for a
for an averaged sequence $\bar x_k$ defined as $\bar x_k =
\sum_{t=0}^{k-1}
\bar \gamma_t x_t, $
where  
$$ \bar \gamma_{t} \triangleq \frac{\g_{t}
	}{\sum_{i=0}^{k-1}\g_i}$$
and $x_k$ is generated via a standard SA scheme. \uvs{However, when the stepsize parameter $\g_k$ is decreasing, the averaging
	weights $\bar \gamma_k$ are decreasing as well, implying that recent iterates $x_k$ are assigned less
	weight than the \uvs{older} iterates.} This suggests that it may be sensible to
	consider an increasing set of weights. In fact, Nedi\'{c} and Lee~\cite{Nedic14} showed that
	by using weights of the form
  $$ \tilde \gamma_{t} \triangleq \frac{\g^{-1}_{t}
	}{\sum_{i=0}^{{k}-1}\g^{-1}_i},$$
 the subgradient mirror-descent algorithm {attains} the optimal rate of
 convergence {\em without} requiring window-based averaging. 
 \an{In this paper, we show} optimality of the rates for an averaged sequence  
	$\bar x_{k}\uvs{(r)}\triangleq
	\sum_{t=0}^{k-1}\bar
	\gamma_{t,r} x_t$ when $r <1$ and 
	 \begin{align} \bar \gamma_{t,r} \triangleq
	\frac{\g_{t}^r}{\sum_{i=0}^{k-1}\g_i^r}.
	\label{defbargamma}\end{align}
\subsection{Outline of contributions}
 We now outline our main contributions.

(a) {\em Almost sure convergence  under monotone non-Lipscitzian
regimes:} We consider  SVIs where the mapping is monotone and
not necessarily Lipschitz continuous. A regularized smoothing SA scheme, referred
to as the RSSA scheme, is developed wherein the regularization
parameter, smoothing parameter, and the steplength are updated
after each iteration. \uvs{Under suitable
assumptions on the smoothing, regularization, and steplength sequences,
the sequence of iterates is shown to converge to the solution set of the
	SVI in an almost sure sense, which is in contrast with almost all available
	almost sure convergence results (that typically require Lipschitz continuity). We
		also derive a  bound on the mean-squared distance of
			any iterate produced by the RSSA scheme from the regularized
			smoothed trajectory (which is known to converge to a
			solution of the original SVI).}

%First, we establish the convergence of the sequence of unique solutions of the smoothed regularized problems, denoted by $s_k$, to the solution set of the SVI. Next, invoking this result, we are able to recover the almost sure convergence of any subsequence of $x_k$ generated by the algorithm to the solution set of the original SVI. 

%\\ (b) {\em Convergence in a mean-squared sense:} Next, we show that the
%RSSA scheme produces iterates Our second goal
%lies in showing convergence in the mean-squared sense for the proposed
%SA method. Suppose $\{x_k\}$ is the sequence generated by our proposed
%SA method and $s_k$ is the solution to the $k$-th regularized and
%smoothed SVI problem, we first derive a bound for the error
%$\EXP{\|x_{k+1}-s_k\|^2}$ in terms of problem parameters and the
%stepsize, regularization and smoothing sequences. Using this bound, we
%are able to show that $\EXP{\hbox{dist}(x_{k},\hbox{SOL}(X,F))^2}$ goes
%to zero as $k\rightarrow \infty$, where $\hbox{SOL}(X,F)$ is the
%solution set of VI$(X,F)$.\\

(b) {\em Optimal averaging schemes:} \uvs{Motivated by the need to develop
rate statements,  we {first} consider an averaging-based
extension of the RSSA scheme, referred to as aRSSA$_r$, in which $\bar x_k$ is
defined as a weighted average: $\bar x_{k}(r)\triangleq \sum_{t=0}^{k-1}\bar
\gamma_{t,r} x_t$ where {$\bar \gamma_{t,r}$ is defined
by \eqref{defbargamma}}. {We derive the underlying conditions under which the
mean gap function of the averaged sequence $\bar x_k(r)$ converges to zero.}
Additionally, \an{under suitable conditions, 
we show that the {aRSSA$_r$} scheme with $r < 1$ produces a sequence of
iterates that converges to the solution set in an almost sure sense 
and the mean gap function diminishes to zero at a rate given by 
${\cal O}(1/K^{(1/6)-\delta})$ where $\delta \in (0,1/6)$.} 
When both regularization and smoothing are suppressed
and $\g_k={1}/{\sqrt{k}}$, we further show that the mean gap function
diminishes to zero at the optimal rate of ${\cal O}(1/\sqrt{K})$ when $r < 1$.
This represents an improvement over the rate ${\cal O}(\ln(K)/\sqrt{K})$
obtained for standard averaging schemes based on $r = 1$. When a window-based
averaging sequence is employed, we show that the scheme recovers the
optimal rate and provide some insights on how the constant factor in the error
bound is improved when $r < 1$ compared to the setting when $r = 1$,
 when the window size is large.}
%(IV) \textbf{almost sure convergence for {$\boldsymbol{\bar x_k}$}:}
%Employing the smoothing and regularization in the averaging SA
%algorithm, we show almost sure convergence of any subsequence of
%$\bar x_k$ under specified conditions. 

(c) {\em Numerics:} Preliminary numerics on a set of stochastic
Nash-Cournot games support the theoretical findings and several observations can be made. First,  the \uvs{performance of the RSSA scheme} is relatively
robust to choices of the parameter sequences. \uvs{Second, iterates generated by the RSSA scheme under the almost sure convergence guarantee, rather than convergence in the mean, are shown to have lower variance in the gap function. Third, in the window-based regimes, choices of $r < 1$ tend to outperform standard averaging schemes with $r=1$, particularly when the window sizes are large.} 
%Importantly, choosing $r< 1$ has significant benefits in terms of finite-time behavior. 

The rest of the paper is organized as follows. Section \ref{sec:IRLSA}
presents the RSSA scheme, its averaged variants, and  our  main
assumptions. In  Section \ref{sec:a.s.}, we prove the almost sure
convergence of the RSSA scheme while in Section \ref{sec:ave}, we
analyze the convergence and derive the rate for the averaged variants of
the RSSA scheme. In Section \ref{sec:num}, the performance of the
proposed methods \uvs{is} tested on a stochastic Nash-Cournot game. The paper
ends with some concluding remarks in Section~\ref{sec:concl}.\\

\textbf{Notation:} A vector $x$ is assumed to be
a column vector, $x^T$ denotes the transpose of a vector $x$, and
$\|x\|$ denotes the Euclidean vector norm, i.e., $\|x\|=\sqrt{x^Tx}$.
We use $\Pi_X(x)$ to denote the Euclidean projection of a vector $x$ on
a set $X$, i.e., $\|x-\Pi_X(x)\|=\min_{y \in X}\|x-y\|$. We  abbreviate ``almost
surely'' as {\em a.s.,} \an{while} $\EXP{z}$ is used to denote the expectation of a random
variable~$z$. We let dist$(s,S)$ denote the Euclidean distance of a
vector $s \in \Real^n$ from a set $S \subset \Real^n$.   	
We use $B_n(y,\rho)$ to denote the ball
	centered at a point $y$ with a radius $\rho$, i.e., $B_n(y,\rho)=\{x \in
	\mathbb{R}^n\mid \|x-y\|\leq \rho\}$.
	\an{We use 
	$X^*$ to denote the solution set of the variational inequality problem in~\eqref{def:SVI}.}
	%VI$(X,F)$ for a set $X$ and a mapping $F$

	%The \texttt{Matlab} notation $(u_1;u_2;u_3)$ refers to a column
	%vector with components $u_1$, $u_2$ and $u_3$, respectively.
	
\section{Algorithm and assumptions}\label{sec:IRLSA}
We present  our 
	schemes of interest and our assumptions in Sections~\ref{sec:21} and ~\ref{sec:22}, respectively.

\subsection{Algorithm}\label{sec:21}
In this section, we present the regularized smoothing stochastic
	 approximation (RSSA) scheme for solving~\eqref{def:SVI}.
We motivate our scheme by first defining the traditional stochastic
approximation scheme for SVIs. Given an $x_0 \in X$, the standard SA
scheme generates a  sequence $\{x_{k}\}$:
\begin{align}
	 \tag{SA}
	 x_{k+1}:= \Pi_X(x_k - \gamma_k \Phi(x_k,\xi_k)), \quad k \geq 0,
\end{align}
where $\{\gamma_k\}$ defines a steplength sequence, while $x_0\in X$ is an
	initial random vector independent
of the random variables $\xi_k$ and such that $\EXP{\|x_0\|^2}<\infty$. This
SA scheme for SVIs appears to have been first studied by
Jiang and Xu~\cite{Houyuan08} where a.s.\ convergence statements were
provided under Lipschitz continuity and strong monotonicity of the map.
In deterministic variational inequality problems, Tikhonov
regularization techniques have proved useful for solving merely monotone
problems through the generation of increasingly accurate solutions of a
sequence of a {\em regularized} VIs (cf.~\cite{facchinei02finite}).
Unfortunately, in stochastic regimes, such an approach is not practical
since it requires running a sequence of simulations of increasingly
longer lengths. Inspired by prior work in deterministic
VIs~\cite{golshtein89modified,kannan10online}, the stochastic iterative
Tikhonov regularization scheme was developed
subsequently~\cite{koshal12regularized}. The {\em regularized}
stochastic approximation scheme (RSA) is defined as follows:
\begin{align}
	 \tag{RSA}
	 x_{k+1}:= \Pi_X(x_k - \gamma_k (\Phi(x_k,\xi_k)+\eta_k x_k)), \quad k \geq 0,
\end{align}
where $\{\eta_k\}$ denotes a regularization sequence
that is driven to zero at specified rates to ensure a.s.\ convergence of
the sequence of iterates to the least norm solution of the monotone
stochastic variational inequality problem. However, the RSA scheme requires Lipschitz
continuity of the map. In prior work, in the context of nonsmooth
stochastic optimization~\cite{Farzad1}, we have employed {\em local
	smoothing} to construct an approximate problem with a prescribed
	Lipschitz constant. Such a problem can then be solved via standard
	SA schemes. However, this avenue provides only approximate
	solutions.  In this paper, we {resolve this shortcoming by
		presenting} a smoothed variant of the RSA
	scheme, referred to as the {\em regularized smoothed} SA (or RSSA) scheme
	under which we can recover solutions to the original problem without
	requiring Lipschitz continuity of the map:
\begin{align}\label{algorithm:IRLSA-impl}
	 \tag{RSSA}
	 x_{k+1}:= \Pi_X(x_k - \gamma_k (\Phi(x_k+z_k,\xi_k)+\eta_k x_k)), \quad k \geq 0,
\end{align}
where $z_k \in \mathbb{R}^n$ is a uniform
	random variable over {an} $n$-dimensional ball centered at the
	origin with radius $\e_k$ for any $k \geq 0$. To have a well defined
	$\Phi$ in the \ref{algorithm:IRLSA-impl} scheme, we define $X^\e$ as $\e$-enlargement of the set $X$, i.e.,
\begin{align}\label{eq:xe}
X^\e\triangleq X+B_n(0,\e),\end{align} 
where $\e$ is an upper bound of the
	sequence $\{\e_k\}$ (which will be finite under our assumptions). Note that {by introducing
		stochastic errors $w_k$,  the \ref{algorithm:IRLSA-impl} scheme is equivalent to the following method:
\begin{align}\label{algorithm:IRLSA}
\tag{RSSA$_w$}
\begin{aligned}
x_{k+1}&=\Pi_{X}\left(x_k-\g_k(F(x_k+z_k)+\eta_kx_k+w_k)\right), \qquad
	\quad &  \ k \geq 0,\cr
 w_k & \triangleq {\Phi(x_k+z_k,\xi_k)}- F(x_k+z_k), \qquad &k \geq 0.
\end{aligned}
\end{align}
In this representation of the RSSA scheme, $w_k$ is the deviation between the
	sample $\Phi(x,\xi_k)$ observed at the $k$-th iteration and the
	expected-value mapping $F(x)$, at $x=x_k+z_k$. An implicit assumption in
	our work is that we have access to a stochastic oracle which is able
	to generate random  samples $\Phi(\cdot,\xi_k)$ at a given point.
	Such an oracle is assumed to be an unbiased estimator, meaning that
	$F(x)= \EXP{\Phi(x,\xi)}$ for any $x \in X^\e$. The results in this
	paper can be extended to the case where $\Phi$ is a biased
	estimator of the mapping $F$, i.e.,  $F(x)= \EXP{\Phi(x,\xi)}+b$ for
	some $b>0$ and all $x \in X^\e$. 

The RSA and RSSA schemes in their presented forms do not easily
allow for determination of non-asymptotic rates of convergence. This
may be provided by constructing averaging-based counterparts which have been
developed both in the context of stochastic optimization
problems~\cite{Polyak92} as well as their variational inequality
counterparts~\cite{Jud11}. Strictly speaking, averaging schemes are not
distinct algorithmically but merely average the generated sequence of
iterates. In contrast with the traditional
averaging approach, we consider {\em weighted averaging} akin to recent work in stochastic
optimization~\cite{Nedic14} by
defining the sequence $\bar x_k(r)$  for $k \geq 0$ and $r\in\Real$. 
The {\em averaged} variant of the RSSA scheme using the
parameter $r$ (referred to as aRSSA$_r$) is defined as follows:
\begin{align}\label{algorithm:IRLSA-averaging}\tag{aRSSA$_{r}$}
\begin{aligned}
x_{k+1}&:=\Pi_{X}\left(x_k-\g_k(\Phi(x_k+z_k,\xi_k)+\eta_kx_k)\right), \cr %\qquad \hbox{for all } k \geq 0 \cr
 \bar x_{k+1}(r) &\triangleq \frac{\sum_{t=0}^k \gamma_t^r x_t}{\sum_{t=0}^k \gamma_t^r}.
 \end{aligned}
\end{align}
Throughout the paper, when we \uvs{neither regularize nor smooth, the
	aRSSA$_r$ algorithm is referred to as aSA$_r$.
Finally,  variants of aSA$_r$ in which averaging is carried out over a {\em
	window} are denoted by
	{aSA$_{\ell,r}$}. If {$0< \ell \leq k$ and $k \geq 1$},
	aSA$_{\ell,r}$ is defined as follows}: 
\begin{align}\label{alg:window}\tag{aSA$_{\ell,r}$}
\begin{aligned}
x_{k+1}&:=\Pi_{X}\left(x_k-\g_k\Phi(x_k,\xi_k)\right), \cr %\qquad \hbox{for all } k \geq 0 \cr
\bar x_{k+1}^{\ell}(r) &\triangleq \frac{\sum_{t=\ell}^{k} \gamma_t^r x_t}{\sum_{t=\ell}^{k} \gamma_t^r}.
 \end{aligned}
\end{align}}

%-------------------------------------------------------
\subsection{Assumptions}\label{sec:22}
%-------------------------------------------------------
We now outline the key assumptions employed in the remainder of this
	paper. Let 
$\sF_k$ denote the history of the method up to time $k$, i.e., 
$\sF_k=\{x_0,\xi_0,\xi_1,\ldots,\xi_{k-1}, z_1,\ldots,z_{k-1}\}$ for
$k\ge 1$ and $\sF_0=\{x_0\}$. Our first set of assumptions is on the
	properties of the set $X$, the mapping $F$, and the non-emptiness of
	the solution set $X^*$ \an{for the problem in~\eqref{def:SVI}}.

\begin{assumption}[{\bf Problem properties}]\label{assum:step_error_sub_1} 
Let the following hold:\\ 
(a) \ The set $X \subset \Real^n$ is closed, bounded, and convex; \\
(b) \  The mapping $F(x)=\EXP{\Phi(x,\xi)}$ is 
monotone and continuous over the set $X^\e$ given in~\eqref{eq:xe};\\
%(c) \  \fy{${\Phi(x,\xi)}$ is a \fy{bounded mapping} over the set {$X^\e$} with respect to $x$ almost surely};\\
(c) \  There exists a scalar $C>0$ such that $\EXP{\|\Phi(x,\xi)\|^2}\leq C^2$ for any $x \in X^\e$;\\
(d) \ $X^* \neq \emptyset$, i.e., there exists an $x^* \in X$ such that $(x-x^*)^T\EXP{\Phi(x^*,\xi)} \geq 0$ for all $x \in X$.
\end{assumption}
\begin{remark}\label{rem:C-M} 
Note that by using Jensen's inequality and Assumption \ref{assum:step_error_sub_1}(c) we can write
\[\|F(x)\|\leq \EXP{\|\Phi(x,\xi)\|}= \sqrt{\EXP{\|\Phi(x,\xi)\|}^2} \leq \sqrt{\EXP{\|\Phi(x,\xi)\|^2}}\leq  \sqrt{C^2} =C.\]
 In our analysis, we make use of the preceding inequality, i.e. \begin{align}\label{ineq:F-C}\|F(x)\| \leq \EXP{\|\Phi(x,\xi)\|}\leq C, \quad \hbox{for all }x \in X^\e.\end{align}
 We also utilize the boundedness of $X$ by which there exists \uvs{a positive
 scalar $M$} such that 
 \[\|x\| \leq M\qquad \hbox{for all $x \in X$}.\] 
 \end{remark}
In the implementation of the RSSA scheme, two distinct random variables
	require discussion. First, the random vector $\xi$ is
	\uvs{idiosyncratic to problem
		(\ref{def:SVI})}, while the random vector $z$ is artificially
		introduced.  Next, we provide some assumptions on
		these two random variables. 
		
\begin{assumption}[{\bf Random variables $\xi$ and $z$}]\label{assum:step_error_sub_2} Let the following hold:\\
(a) \  The random variables \us{$\xi_j \in \mathbb{R}^d$} 
are \an{independent and identically distributed}   for any $j \geq 0$.\\
(b) \ \uvs{The} random variables $z_i \in \mathbb{R}^n$ are 
independent and uniformly distributed in an 
$n$-dimensional ball with radius $\e_i$ centered at the origin for any $i \geq 0$.\\
(c) \  \an{The random variables $z_i$ and $\xi_j$ are independent} for any $i,j \geq 0$.
\end{assumption}
Based on this assumption, we may derive the following regarding the
conditional first and second moments of $w_k$.
\begin{lemma}[{\bf Conditional first and second moments of $w_k$}]\label{lemma:bound-on-errors} 
	Consider the 
		(\ref{algorithm:IRLSA}) scheme and suppose Assumptions
		\ref{assum:step_error_sub_1}(c) and \ref{assum:step_error_sub_2}
	hold. Then, the stochastic error $w_k$ satisfies the following
		relations for any $k \geq 0$: \[\EXP{w_k \mid \sF_k\cup
			\{z_k\}}=0 \quad  \hbox{ and } \quad\EXP{\|w_k\|^2\mid \sF_k\cup
				\{z_k\}} \leq C^2.\]
				{Furthermore, for any $k \geq 0$,
				\[\EXP{w_k \mid \sF_k}=0 \quad  \hbox{ and } \quad \EXP{\|w_k\|^2\mid \sF_k} \leq C^2.\]
				}
\end{lemma}
\begin{proof} Let $k$ be a \uvs{a fixed nonnegative
	integer}. The definition of $w_k$ in (\ref{algorithm:IRLSA}) implies that
\begin{align*}
\EXP{w_{k} \mid \sF_k \cup \{z_k\}} =\EXP{\Phi(x_{k}+z_k,\xi_{k})\mid \sF_k\cup \{z_k\}}-F(x_k+z_k)= F(x_k+z_k)-F(x_k+z_k)= 0,
\end{align*} where we used the independence of $z_k$ and $\xi_k$.
{By taking expectations with respect to $z_k$, we obtain
$\EXP{w_k \mid \sF_k}=0.$}
{For the term $\EXP{\|w_k\|^2\mid \sF_k\cup
				\{z_k\}}$ using Assumption \ref{assum:step_error_sub_1}(c)}, we may write
\begin{align}\label{equ:-w_k-1}
& \ \quad \EXP{\|w_k\|^2\mid \sF_k\cup \{z_k\}}=\EXP{\|\Phi(x_k+z_k,\xi_k)- F(x_k+z_k)\|^2\mid \sF_k\cup \{z_k\}}\\
& = \underbrace{\EXP{\|\Phi(x_k+z_k,\xi_k)\|^2\mid \sF_k\cup \{z_k\}}}_{\hbox{Term } 1}
+{\|F(x_k+z_k)\|^2} \notag
-2{\underbrace{\EXP{\Phi(x_k+z_k,\xi_k)\mid \sF_k\cup \{z_k\}}^TF(x_k+z_k)}_{\hbox{Term } 2}}.
\end{align}
Using Assumption \ref{assum:step_error_sub_1}(c), we observe that Term $1\leq C^2$. {Furthermore, we have
\begin{align}\label{equ:-w_k-3}
%\EXP{\Phi(x_k+z_k,\xi_k)\mid \sF_k\cup \{z_k\}}^TF(x_k+z_k)
\mbox{Term 2} =F(x_k+z_k)^TF(x_k+z_k)
=\|F(x_k+z_k)\|^2.
\end{align}
}
Therefore, from relations (\ref{equ:-w_k-1}) and (\ref{equ:-w_k-3}) we obtain 
\[{\EXP{\|w_k\|^2\mid \sF_k\cup \{z_k\}} \leq C^2 - \|F(x_k+z_k)\|^2 \leq C^2}.\]
{Taking expectation with respect to $z_k$ implies that $\EXP{\|w_k\|^2\mid \sF_k} \leq C^2.$}
\end{proof}

%-----------------------------------------------------------------------------------
\section{Convergence analysis of \us{RSSA scheme}}\label{sec:a.s.}
%-----------------------------------------------------------------------------------
\an{In this section, we provide some properties of the smoothed map (Section~\ref{sec:smoothed}). Then, we prove a.s. 
(Section~\ref{sec:31}) and mean-squared convergence
(Section~\ref{sec:mean}) of the RSSA scheme.}

%-----------------------------------------------------------------------------------
\subsection{Smoothed map and its properties}\label{sec:smoothed}
%-----------------------------------------------------------------------------------

\an{Our RSSA scheme uses} a family of approximate smoothed mappings defined as follows.  
\begin{definition}[{\bf Smoothed mapping}]\label{def:F_k}
Consider mapping $F:X^\e \rightarrow \mathbb{R}^n$. Let $z_k \in
\mathbb{R}^n$ be a uniform random vector in $B_n(0,\e_k)$ for all $k\geq
0$. The smoothed (approximate) mapping $F_k:X \rightarrow \mathbb{R}^n$ is defined by
\[F_k(x)=\EXP{F(x+z_k)}, \quad \hbox{for any } x \in X.\]
\end{definition}

\an{Note that $F_k$ is characterized
		by the random variable $z_k$  which is assumed to be uniformly
			distributed over the $\e_k$-ball centered at the origin. 
Thus, when $\e_k$ is small,  the mapping $F_k$ can be viewed as an approximation of
		the original mapping $F$.} 
			More precisely, the random variable $z_k$ has the following probability distribution function:
\begin{equation}\label{eqn:zuniform}
p{_u}(z_k) = \left\{\begin{array}{ll} \frac{1}{c_n \varepsilon_k^n}
&\hbox{for } z_k \in B_{n}(0,\e_k),\cr \hbox{} &\hbox{}\cr 0
&\hbox{otherwise,}\end{array}\right.
\end{equation}
where \an{$c_n$ is the volume of the unit ball in $\re^n$, i.e.,}
$c_{n} = \int_{B_n(0,1)}dy= \dfrac{\pi^\frac{n}{2}}{\Gamma(\frac{n}{2} + 1)}$,
%$!!$ denotes the double factorial operator 
	and $\Gamma$ is the gamma function defined by 
\begin{eqnarray}\label{def:gamma}
\Gamma\left(\frac{n}{2}+1\right)= \left\{ \begin{array}{ll}
\left(\frac{n}{2}\right)!, &\hbox{if $n$ is even,}\cr
\hbox{}&\hbox{}\cr \sqrt{\pi}\,\frac{n!!}{2^{(n+1)/2}}. &\hbox{if $n$
is odd.}
\end{array}\right.
\end{eqnarray}
\an{The advantage of smoothing is that, when the mapping $F$ is bounded over the set $X^\e$, 
then each approximate mapping $F_k$ is Lipschitz continuous, with Lipschitz constant depending on $\e_k$. 
When $F$ is not Lipschitz continuous, one may consider approximating $F$ with $F_k$, and thus take advantage of 
efficiently solving the variational inequality VI($X,F_k$) in order to find an approximate solution to 
VI($X,F$), and obtain a solution to VI($X,F_k$)  as $\e_k$ decreases to 0.
This can be done using either {\it two-loop or single-loop} schemes.

A two-loop scheme consists of an outer loop for iteratively changing  $\e_k$, 
and an inner loop for solving 
VI($X,F_k$) for a fixed value of $\e_k$. In \uvs{instances where} $F$ has a
	\uvs{closed-form expression} and the evaluation of the integral $F_k$ (possibly
			multi-dimensional) can be done efficiently, then VI$(X,F)$
	can be solved by employing an adaptive smoothing
	scheme~\cite{Polak84}.  
In our prior work~\cite{Farzad1,Farzad03}, we have used the adaptive smoothening to address nonsmoothness in 
stochastic convex optimization problems~\cite{Farzad1} and the absence of Lipschitzian properties on the maps 
in monotone SVIs~\cite{Farzad03}. Specifically, regarding the variational inequality problems, we developed
					   an adaptive smoothing SA scheme for the case  where
					   in the inner loop, the problem VI$(X,F_k)$ is
					   solved using a smoothing variant of SA method and
					   the asymptotic convergence of the scheme is shown
					   in an almost sure sense.  A challenge associated with the
	 adaptive smoothing approach is that in many settings, $F_k$ cannot be evaluated efficiently; this could be a
	consequence of the stochasticity of $F_k$. As a result, the inner
	loop of such a scheme requires conducting a simulation. In fact, 
	%shows in Lemma~\ref{lemma:F_k_monotone}
	even if the subproblem was characterized by a tractable $F_k$, 
	the Lipschitz constant of $F_k$
is proportional to the inverse of $\e_k$, as we will show soon. Therefore, when $\e_k \to 0$, the
Lipschitz constant tends to $+\infty$, thus making the problems VI$(X,F_k)$ increasingly challenging to solve.

As a possible remedy for this situation, in this paper, we develop a single-loop approach where we change $\e_k$ at each iteration, thus changing the smoothed map $F_k$ at each step of the algorithm.
Furthermore, we also update the regularization parameter
			at every step. The parameters are  driven to
			zero at prescribed rates to ensure a.s. convergence of the
			produced iterate sequence (Theorem~\ref{prop:almost-sure}).}

Next, we prove the monotonicity and Lipschitz continuity of the smoothed
mapping. % $F_k$.
\begin{comment}
Our work is motivated by a class of averaged functions first introduced
by Steklov  \cite{steklov1}, employed for the solution of stochastic
optimization problems
\cite{Bertsekas73,norkin93optimization,DeFarias08,Duchi12}. 
It is well-known that given a convex
function $f(x):\Real^n \rightarrow \Real$ and a random variable $\omega$
with probability distribution $P(\omega)$, the function $\hat f$ defined
by $\hat f(x) \triangleq \int_{\Real^n}f(x+\omega)P(\omega)d\omega=\EXP{f(x+\omega)} $  is a differentiable function.  
Employing this
technique allowed us to address non-smoothness in developing adaptive
stepsize SA schemes for stochastic convex optimization problems
and Cartesian SVIs in absence or unavailability of a Lipschitz
constant~\cite{Farzad1,Farzad03}. \fyRev{The main limitation of the employed smoothing scheme in our previous work was that in the proposed smoothing SA scheme~\cite{Farzad1,Farzad03}, the smoothing parameter is assumed to be fixed through implementing the SA method. As a consequence, the convergence of the smoothing SA scheme is shown to the optimal solution of the approximate problem, i.e.,  $\min_{x\in X} \hat f(x)$, where $\hat f$ is the smoothed objective function. A natural approach to address this limitation and ascertain convergence to the optimal solution of the original problem has been managed through employing ``adaptive smoothing'' schemes where the smoothing parameter is reduced adaptively to zero (see~\cite{Polak84,Xu01}).
\end{comment}

\begin{lemma}[{\bf Properties of {smoothed mapping}}]\label{lemma:F_k_monotone} 
Consider the smoothed mapping $F_k$ as given in Definition \ref{def:F_k}.
Then, the following hold:
\begin{itemize}
\item[(a)] {Let $\{x_t\}\subset X$ be a convergent sequence in $X$ i.e., such that 
$\lim_{t\to\infty} x_t=\hat x$ with $\hat x\in X$. Also, let $F$ be continuous on the set $X^\e$.
If Assumption~\ref{assum:step_error_sub_1}(c) holds
and $\e_t\to0$, then
\[ \lim_{t \to\infty} F_t(x_t)=F(\hat x).\] }
\item [(b)] 
Let Assumptions~\ref{assum:step_error_sub_1}(a) and~\ref{assum:step_error_sub_1}(c) hold. For any $k \geq 0$, 
the mapping $F_k$ is Lipschitz continuous over the set $X$ 
with the parameter $\kappa\frac{n!!}{(n-1)!!}\frac{C}{\epsilon_k}$, 
where $\kappa=1$ if $n$ is odd and $\kappa=\frac{2}{\pi}$ otherwise.
\item [(c)] 
If the mapping $F:{X^\e} \rightarrow \mathbb{R}^n$ is monotone over the set {$X^\e$}, then the mapping $F_k$ is monotone over the set $X$.
	\end{itemize}
\end{lemma}
 \begin{proof}
 {Using the definition of $F_t$ and letting $c_n$ be the volume of the $n$-dimensional unit ball, i.e., $c_n=\int_{\|y\| \leq1}dy$, we have}
\begin{align*}
\lim_{t\to\infty} \EXP{F(x_t+z_{t})}
= \lim_{t \to \infty}\int_{\|z\| \leq \e_{t}} F(x_{t}+z) \frac{1}{c_n\e_{t}^n}dz.
\end{align*}
By change of variables, \fyRev{$y=\frac{z}{\e_t}$}, it follows that 
\begin{align*}%\label{equ:VI-sk-3}
\lim_{t\to\infty} \EXP{F(x_t+z_{t})}= \frac{1}{c_n}\lim_{t\to \infty}\int_{\|y\| \leq1} F(x_{t}+\e_t y) dy.
\end{align*}
{By Assumption~\ref{assum:step_error_sub_1}(c) we have that $\|F(x+z)\| \leq C$ (see Remark~\ref{rem:C-M}), 
implying that $F(x+z)$ is integrably bounded with respect to the distribution defining the random
variable $z$. Thus, by appealing to Lebesgue's dominated
convergence theorem, we interchange the limit and the integral leading to the following relations:
\[\lim_{t\to\infty} \EXP{F(x_t+z_{t})} = \frac{1}{c_n}\int_{\|y\| \leq1} \lim_{t\to \infty}F(x_{t}+\e_t y) dy
=\frac{1}{c_n}\int_{\|y\| \leq1} F(\hat x)dy,\]
where the last equality follows by the
continuity of the mapping $F$, and $x_t\to\hat  x$, $\e_t\to0$.
Finally, we may conclude that the above integral reduces to $F(\hat s)$
by invoking the definition of $c_n$ as the volume of $B_n(0,1)$.}

\noindent
(b)\  Let $p_u$ denote the probability density function of the random vector $z$ and suppose  $k \geq 0$ is fixed. From the definition of $F_k$, for any $x,y\in X$, 
\[ \| F_k(x) -F_k(y)\| =\left\|\int_{\Real^n}F(x+z_k)p{_u}(z_k)dz_k-\int_{\Real^n}F(y+z_k)p{_u}(z_k)dz_k\right\|. \]
By changing the integral variable in the preceding relation, we obtain
\begin{align}\label{ineq:lips-bound-1}
\|F_k(x) - F_k(y)\| &  =\left\|\int_{\Real^n}(p{_u}(v -x)-p_u(v -y))F(v)dv\right\|\cr
& \leq  \int_{\Real^n}|p_u(v -x)-p_u(v -y)|\|F(v)\| dv\leq  C \int_{\Real^n}|p_u(v -x)-p_u(v -y)|dv,
\end{align}
where the \fyRev{first} inequality follows from Jensen's inequality and
the \uvs{second} inequality is a consequence of boundedness of the
mapping $F$ over $X^\e$. The remainder of the proof is similar to
\uvs{that of}\an{\footnote{In \cite{Farzad1}, we have considered a convex nondifferentiable optimization 
problem, where the smoothening was effectively applied to the subdifferental set of the objective function. A part of that proof applies here.} } 
Lemma~$8$ of~\cite{Farzad1}. 

\noindent(c)\ 
 Since $F$ is monotone over {$X^\e$} we have that 
 \[(a-b)^T(F(a)-F(b))\geq 0,\quad \hbox{for all } a,b \in {X^\e}.\]
 Therefore, for choice of \uvs{$ a = x+z_k$ and $b= y+z_k$} in $X^\e$, we have that 
 \[((x+z_k)-(y+z_k))^T(F(x+z_k)-F(y+z_k))\geq 0,\quad \hbox{for all } x,y \in X.\]
It follows that 
 \[(x-y)^T(F(x+z_k)-F(y+z_k))\geq 0,\quad \hbox{for all } x,y \in X.\]
 Taking expectations on both sides of the preceding relation, the monotonicity of $F_k$ follows from
 \[(x-y)^T( F_k(x)- F_k(y))\geq 0,\quad \hbox{for all } x,y \in X.\]
 \end{proof}
 
\begin{remark} 
%Note that the notation ``!!'' in Lemma \ref{lemma:F_k_monotone}(a) denotes double factorial, and 
Note that $\frac{n!!}{(n-1)!!}$ in Lemma~\ref{lemma:F_k_monotone}(b)
 is of the order $\sqrt{n}$. Lemma~\ref{lemma:F_k_monotone}(c) implies that the mapping $F_k+\eta_k\bf{I}$
is strongly monotone for any $\eta_k>0$. 
In view of Lemma~\ref{lemma:F_k_monotone}(c), when $X$ is closed and
convex, Theorem $2.3.3$(b) of~\cite{facchinei02finite}, page 156,  ensures that VI$(X,F_k+\eta_k\bf{I})$ has a
unique solution.  \end{remark}

%-------------------------------------------------------------------
\subsection{Almost sure convergence}\label{sec:31}
%--------------------------------------------------------------------
\an{Our main result in this section shows that the sequence produced by the RSSA
scheme converges to $X^*$ in an almost sure sense for suitably chosen stepsize values and other parameters.
The proof idea is to consider the (Tikhonov) sequence $\{s_k\}$ where $s_k$ is the solution
to the regularized smoothed approximation VI$(X,F_k+\eta_k\bf{I})$ of the
original SVI. Then, knowing the behavior of the sequence $\{s_k\}$ and a relation for the consecutive differences 
$\|s_{k+1}-s_k\|$,
we can show that $\|x_{k+1}-s_k\| \to 0$ under suitable conditions on the stepsizes, the regularization and the smoothening parameters.}

\begin{definition}[{\bf Solution of the smoothed regularized problem}]\label{def:s_k}
For a given iteration index $k$, let $s_k$ be the unique solution of VI$(X,F_k+\eta_k\bf{I})$, where 
$F_k:X \rightarrow \mathbb{R}^n$ is given by Definition \ref{def:F_k} and 
$\eta_k>0$ is the regularization parameter. 
\end{definition}

%We let $t^*$ denote the least norm solution of VI$(X,F)$, i.e.,
%$t^*=\argmin_{x\in X^*} \|x\|$, and note that $t^*$ is unique.
%%since it is the projection of the origin on the convex and bounded set \fyRev{$X^*$}. 
%Specifically,
%{we first show that the sequence $\{s_k\}$ has its accumulation points in $X^*$
%	(Proposition~\ref{prop:sk_estimate}) and identify  conditions that
%		ensure that this sequence converges to $t^*$, the smallest-norm solution
%	of VI$(X,F)$}. Then, we proceed to
%	derive a bound on the difference between $x_k$ and $s_k$ (Lemma
%			~\ref{lemma:err_ineq}). By utilizing this bound, we show that
%	$\|x_{k+1}-s_k\| \to 0$ as $k \to \infty$ in an almost sure sense
%	(Theorem~\ref{prop:almost-sure}). 
%
%
%Recall that the VI$(X,F_k + \eta_k {\bf I})$ is a {\em strongly
%	monotone} deterministic
%	variational inequality since $F_k$ was shown to be a monotone map. 
	We now present a bound on $\|s_k-s_{k-1}\|$
		and proceed to prove that the sequence $\{s_k\}$ of approximate
		solutions has accumulation points in the set $X^*$.
\begin{proposition}[{\bf Convergence of $\{s_k\}$}]\label{prop:sk_estimate} 
Suppose Assumptions \ref{assum:step_error_sub_1} and \ref{assum:step_error_sub_2} hold. 
Consider the sequence $\{s_k\}$ of solutions $s_k$ given in Definition~\ref{def:s_k}. Then, \\
(a) \ For any $k\geq1$, $$\|s_k -s_{k-1}\| \leq \frac{2nC}{\eta_{k-1}}\left(1-\frac{\min\{\e_k,\e_{k-1}\}}{\max\{\e_k,\e_{k-1}\}} \right)+M\left
|1-\frac{\eta_k}{\eta_{k-1}} \right|,$$ where $M$ and $C$ are the
	bounds on $X$ and $F$  (see Remark~\ref{rem:C-M} and  Assumption~\ref{assum:step_error_sub_1}(c), respectively).\\
(b) \  
Suppose that the sequences $\{\eta_k\}$ and $\{\e_k\}$ tend to zero,
i.e.,  $\lim_{k\to\infty}\e_k=0$ and $\lim_{k\to0}\eta_k=0.$
Then, we have the following:\\
(1) \
$\{s_k\}$ has an accumulation point and 
every accumulation point of $\{s_k\}$ is a solution to VI$(X,F)$;\\
(2) \an{Let $t^*$ be the smallest norm solution to VI$(X,F)$.}
If $\lim_{k\to\infty}\frac{\e_k} {\eta_k}=0$ and $F$ is
differentiable at $t^*$ with a bounded Jacobian in a neighborhood of
$t^*$, then $\{s_k\}$ converges to \an{$t^*$}.
%the smallest norm solution of VI$(X,F)$.}
 \end{proposition}
\begin{proof}
(a) \ Suppose $k \geq 1$ is fixed. Since \an{$s_k$ and $s_{k-1}$
are the solutions to VI$(X, F_k+\eta_k\textbf{I})$ and VI$(X, F_{k-1}+\eta_{k-1}\textbf{I})$, respectively,
it follows that}
\[(s_{k-1}-s_k)^T (F_k(s_k)+\eta_ks_k) \geq 0 \hbox{ and }(s_k-s_{k-1})^T (F_{k-1}(s_{k-1})+\eta_{k-1}s_{k-1}) \geq 0.\]
Adding the preceding relations, yields $(s_{k-1}-s_k)^T( F_k(s_k)- F_{k-1}(s_{{k-1}})+\eta_ks_k-\eta_{k-1}s_{k-1}) \geq 0.$ 
By adding  and subtracting ${(s_{k-1}-s_k)^T \left(F_{k-1}(s_k)+\eta_{k-1}s_k\right)}$, we obtain that 
 \begin{align*}
 &(s_{k-1}- s_k)^T ( F_k(s_k) -  F_{k-1}(s_{k}))  +  (s_{k-1}- s_k)^T(F_{k-1}(s_{k}) 
 -  F_{k-1}(s_{k-1}))\cr&+(\eta_k-\eta_{k-1})(s_{k-1}- s_k)^Ts_k-\eta_{k-1}\|s_k-s_{k-1}\|^2 \geq 0.
 \end{align*}
By the monotonicity of $F_{k-1}$, it follows that $(s_{k-1}- s_k)^T(F_{k-1}(s_{k}) 
 -  F_{k-1}(s_{k-1}))\le0$, thus implying
\[(s_{k-1}- s_k)^T( F_k(s_k) -  F_{k-1}(s_{k}))+(\eta_k-\eta_{k-1})(s_{k-1}- s_k)^Ts_k
\ge \eta_{k-1}\|s_k-s_{k-1}\|^2. \]
By the Cauchy-Schwartz inequality and by recalling that $\|s_k\| \leq M$ {(see Remark~\ref{rem:C-M})}, 
we obtain
\begin{align}\label{smooth:sub-bound_1}
 \eta_{k-1}\|s_k-s_{k-1}\| \leq \| F_k(s_{k}) - F_{k-1}(s_{k})\|+M|\eta_{k-1}-\eta_k|.
 \end{align}

Let $p_u$ denote the density function of the random vector
	 \fyRev{$z_k$ uniformly distributed over the ball $B_n(0,\e_k)$ given by equation~\eqref{eqn:zuniform}.}
					 To estimate the term $\|
			 F_k(s_{k}) - F_{k-1}(s_{k})\|$, we consider two cases
				 based on whether $\epsilon_k$ is less than
					 $\epsilon_{k-1}$ or not.\\
\noindent (i) \ 
($\epsilon_k \leq \epsilon_{k-1}$): We begin by showing that $\|F_k(s_{k}) - F_{k-1}(s_{k})\|$ can be expressed as follows:
\begin{align*}
  & \quad \ \|F_k(s_{k}) -  F_{k-1}(s_{k})\|=\left\| \int_{\mathbb{R}^n}F(s_{k}+z_k)p_u(z_k)dz_k - \int_{\mathbb{R}^n}F(s_{k}+z_{k-1})p_u(z_{k-1})dz_{k-1}\right\| \cr
 &=\left\| \int_{\|z\|< \e_k}\frac{F(s_{k}+z)}{c_n\e_k^n}dz - \int_{\|z\|< \e_{k-1}}\frac{F(s_{k}+z)}{c_n\e_{k-1}^n}dz\right\| \cr
 &=\left\| \int_{\|z\|< \e_k}\frac{F(s_{k}+z)}{c_n\e_k^n}dz -
 \left(\int_{\|z\|<
		 \e_{k}}\frac{F(s_{k}+z)}{c_n\e_{k-1}^n}dz+\int_{\e_k
		 \leq\|z\|<
		 \e_{k-1}}\frac{F(s_{k}+z)}{c_n\e_{k-1}^n}dz\right)\right\|. 
\end{align*}
where in the third equality, we note that $\{z \in \mathbb{R}^n
\mid \|z\|< \epsilon_{k-1}\}=\{z \in \mathbb{R}^n \mid \|z\|<
\epsilon_{k}\}\cup \{z \in \mathbb{R}^n \mid \epsilon_{k} \leq \|z\|<
\epsilon_{k-1}\}$ when $\e_k\leq \e_{k-1}$.
The right hand side may be further bounded as follows:
\begin{align*} & \quad \left\| \int_{\|z\|< \e_k}\frac{F(s_{k}+z)}{c_n\e_k^n}dz - \left(\int_{\|z\|< \e_{k}}\frac{F(s_{k}+z)}{c_n\e_{k-1}^n}dz+\int_{\e_k \leq\|z\|< \e_{k-1}}\frac{F(s_{k}+z)}{c_n\e_{k-1}^n}dz\right)\right\| \cr
 &\leq  \left\| \int_{\|z\|< \e_k}F(s_{k}+z)\left(\frac{1}{c_n\e_k^n}-\frac{1}{c_n\e_{k-1}^n}\right)dz\right\| +\left\|\int_{\e_k \leq \|z\|< \e_{k-1}}\frac{F(s_{k}+z)}{c_n\e_{k-1}^n}dz\right\| \cr
  &\leq  \int_{\|z\|< \e_k}\|F(s_{k}+z)\|\left|\frac{1}{c_n\e_k^n}-\frac{1}{c_n\e_{k-1}^n}\right|dz +\int_{\e_k \leq \|z\|< \e_{k-1}}\frac{\|F(s_{k}+z)\|}{c_n\e_{k-1}^n}dz,
 \end{align*}
where in  the last two inequalities, we use the
triangle inequality and Jensen's inequality respectively.
Invoking relation~\eqref{ineq:F-C}, we obtain
  \begin{align*}
 \| F_k(s_{k}) -  F_{k-1}(s_{k})\| & \leq  C\int_{\|z\|< \e_k}\left|\frac{1}{c_n\e_k^n}-\frac{1}{c_n\e_{k-1}^n}\right|dz +C\int_{\e_k \leq \|z\|< \e_{k-1}}\frac{1}{c_n\e_{k-1}^n}dz \cr
 &=
 C(c_n\e_k^n)\left(\frac{1}{c_n\e_k^n}-\frac{1}{c_n\e_{k-1}^n}\right)+C(c_n\e_{k-1}^n-c_n\e_k^n)\frac{1}{c_n\e_{k-1}^n}
 \cr & =2C\left(1-\left(\frac{\e_k}{\e_{k-1}}\right)^n\right).
 %&=C\left(1-\left(\frac{\e_k}{\e_{k-1}}\right)^n\right)+C\left(1-\left(\frac{\e_k}{\e_{k-1}}\right)^n\right)\cr 
 \end{align*}
 Now, using relation (\ref{smooth:sub-bound_1}), we obtain
 \begin{align}\label{ineq:smooth_bound_2}\|s_k -s_{k-1}\| \leq \frac{2C}{\eta_{k-1}}\left(1-\left(\frac{\e_k}{\e_{k-1}}\right)^n \right)+M\left|1-\frac{\eta_k}{\eta_{k-1}} \right|.\end{align}
Since we assumed that $\e_{k} \leq \e_{k-1}$, we may write
\begin{align}\label{ineq:smooth_e_k}
1-\left(\frac{\e_k}{\e_{k-1}}\right)^n  = \left(1-\frac{\e_k}{\e_{k-1}}\right)\left(1+ \left(\frac{\e_k}{\e_{k-1}}\right)+\ldots+\left(\frac{\e_k}{\e_{k-1}}\right)^{n-1}\right)\leq n\left(1-\frac{\e_k}{\e_{k-1}}\right).
\end{align}
Therefore when $\e_{k} \leq \e_{k-1}$, from (\ref{ineq:smooth_e_k}) and (\ref{ineq:smooth_bound_2}), \fy{the desired inequality} holds for all $k\geq 1$. \\
\noindent (ii)\ ($\epsilon_k \geq \epsilon_{k-1}$):
Now, suppose $\e_{k} \geq \e_{k-1}$. Following the similar steps
above, one \us{may note that} that if $\e_{k} \geq \e_{k-1}$, then $\|
F_k(s_{k}) -  F_{k-1}(s_{k})\| \leq 2C\left(1-({\e_{k-1}}\slash{\e_k})^n\right)$. 

Therefore, \us{the desired equality follows by combining cases (i) and
	(ii) to obtain the following bound: $$\| F_k(s_{k}) -  F_{k-1}(s_{k})\| \leq
2nC\left(1-\frac{\min(\e_{k-1},\e_k)}{\max(\e_{k-1},\e_k)}\right).$$}  \\
(b) We begin by considering \an{part} (1).  
{By Definition~\ref{def:s_k}, the vector $s_k$ is the solution of
VI$(X,F_k+\eta_k\mathbf{I})$ indicating that  for all $k\ge0$,
\begin{align}\label{equ:VI-sk}
(x-s_k)^T\left(F_k(s_k)+\eta_ks_k\right) \geq 0, \quad \hbox{for all } x \in X,
\end{align}
with $s_k \in X$. Furthermore, by
Assumption \ref{assum:step_error_sub_1}(a), the set $X$ is bounded and,
therefore,} $\{s_k\}$ is bounded and has at least one accumulation point. 
Let $\hat s$ denote an arbitrary accumulation point of the
sequence $\{s_k\}$, i.e. $\lim_{i \rightarrow \infty} s_{k_i}=\hat s$. Observe that by Lemma~\ref{lemma:F_k_monotone}(a), it follows that 
the limit $\lim_{i \rightarrow
		\infty}F_{k_i}(s_{k_i})$ exists, namely
\[\lim_{i \rightarrow
		\infty}F_{k_i}(s_{k_i})=F(\hat s).\]	
Thus, by taking the limit along the subsequence $\{k_i\}$ in relation \eqref{equ:VI-sk} and using $\e_k\to0$, for
any $x \in X $ we obtain
\begin{align*}
 \qquad (x-\lim_{i \rightarrow \infty}s_{k_i})^T\left(\lim_{i \rightarrow
		\infty}F_{k_i}(s_{k_i})+\lim_{i \rightarrow
		\infty}\eta_{k_i}\lim_{i \rightarrow \infty}s_{k_i}\right) \geq
0 \qquad 
 \implies (x-\hat s)^T F(\hat s) &\geq 0,
\end{align*}
showing that $\hat s$ is a solution to VI$(X,F)$. 
Thus, all accumulation points of $\{s_k\}$ are solutions to VI$(X,F)$, which proves the statement in part (1).

We now consider (2) of (b) where we show that \fyRev{$\lim_{k \to
	\infty} s_k=t^*$ \uvs{where} $t^* \in X^*$}.  Therefore, \begin{align}\label{equ:VI-tk}
(x-t^*)^TF(t^*)\geq 0, \quad \hbox{for any } x \in X.
\end{align}
Furthermore, we have
\begin{align}\label{equ:VI-sk2}
(x-s_k)^T(F_k(s_k)+\eta_ks_k)\geq 0, \quad \hbox{for any } x \in X.
\end{align}
Replacing $x$ by $s_k$ in (\ref{equ:VI-tk}) and replacing $x$ by $t^*$ in (\ref{equ:VI-sk2}) and then summing the resulting inequalities we obtain
\begin{align*} &(s_k-t^*)^T(F(t^*)-F_k(s_k)-\eta_ks_k) \geq 0\cr
\implies \quad  & \eta_k\|s_k\|^2 \leq \eta_ks_k^Tt^*+(s_k-t^*)^T(F(t^*)-F_k(s_k)) \cr 
\implies \quad  & \eta_k\|s_k\|^2 \leq \eta_ks_k^Tt^*+(s_k-t^*)^T(F(t^*)-F_k(t^*))+(s_k-t^*)^T(F_k(t^*)-F_k(s_k)).
\end{align*}
\us{By observing the nonpositivity of the third term on the right from
	the monotonicity of $F_k$ and by  using the Cauchy-Schwartz
		inequality,} we have the following:
\begin{align*}
&\eta_k\|s_k\|^2 \leq \eta_k\|s_k\|\|t^*\|+\|s_k-t^*\|\|F(t^*)-F_k(t^*)\| \cr 
\implies \quad & \|s_k\| \leq \|t^*\|+\eta_k^{-1}\left(\frac{\|s_k\|+\|t^*\|}{\|s_k\|}\right)\|F(t^*)-F_k(t^*)\|.
 \end{align*}
\us{\fyRev{Without loss of generality, we may assume  that  $0 \notin X$ (In the case that $0 \in X$, since $X$ is bounded, we can reformulate the problem VI$(X,F)$ by changing the variable $x$ and defining $v = x +2\left(\max_{x\in X} \|x\|\right) \mathbf{1_n}$ where $\mathbf{1}_n$ is the unit vector. Note that $\|v\| \neq 0$ for all $x \in X$). Therefore, since we assume $0 \notin X$, there exists $d>0$ such that $\|s_k\|
		\geq d$.} Since $X$ is bounded, then $\|t^*\| \leq M$ where
			$M$ is the bound on set $X$.} 
This implies \us{the following:} 
\begin{align}\label{ineq:s_k-least-norm}
  \|s_k\|& \leq \|t^*\| +
  \eta_k^{-1}\left(1+\frac{M}{d}\right)\left\|\int(F(t^*+z_k)-F(t^*))p_u(z_k)dz_k\right\|
  \cr &\leq \|t^*\| +\eta_k^{-1}
  \left(1+\frac{M}{d}\right)\int_{\|\fyRev{z_k}\|\leq \e_k} \|F(t^*+z_k)-F(t^*)\|
  p_u(z_k)\fyRev{dz_k} \cr &\leq  \|t^*\| +\eta_k^{-1} \left(1+\frac{M}{d}\right)\sup_{\|\an{z_k}\|\leq \e_k}\|F(t^*+z_k)-F(t^*)\|.
 \end{align}
Let $J_F$ denote the Jacobian of $F$. \us{By} assumption, there exists a
$\rho>0$ where $\|J_F(x)\| \leq J_{ub}$ for any $x \in B(t^*,\rho)$.
Using the \us{mean value} theorem,
\begin{align}\label{mean-valued}
F(x+\delta)-F(x)=\left(\int_{0}^1 J_F(x+\tau\delta)d\tau \right)\delta, \quad \hbox{for any } \|\delta\| \leq \rho.
\end{align}
Assume that $K$ is a large number such that for any $k>K$, $\e_k<\rho$. Using the boundedness of $J_F$ and the triangle inequality, from (\ref{mean-valued}) we obtain 
\[\|F(t^*+z_k)-F(t^*)\| \leq J_{ub} \|z_k\| \leq J_{ub}\e_k.\]
\fyRev{Note that it is assumed that $\lim_{k\to \infty} \e_k/\eta_k =0.$}
  \us{By the preceding
	inequality, relation (\ref{ineq:s_k-least-norm}) and
		\fyRev{$\lim_{k\to \infty} \e_k/\eta_k =0$, we conclude that
			\uvs{every}
		subsequence of $s_k$}, denoted by $\{s_{k_i}\}$, has a limit
		point $\bar s$ such that $\|\bar s \| \leq \|t^*\|$. But from
		Prop. \ref{prop:sk_estimate} (b), every accumulation point of
		$\{s_k\}$ lies in $X^*$. But, since every limit point is bounded
		in norm by $\|t^*\|$ \uvs{and $t^*$ is the least-norm point in
			$X^*$}, it follows that every limit point of
		$\{s_k\}$ is
		$t^*$, the unique least-norm solution. It follows that $\{s_k\}$
		is a convergent sequence that tends to the least-norm solution
		$t^*$.}
\end{proof}
\begin{remark}
We note that part (2) of (b) in the above proposition requires a local
differentiability and boundedness property. This can be seen to be
weaker than a global Lipschitzian requirement. We also note that without
such an assumption, we may still claim that $\{s_k\}$ converges to a
point in $X^*$ but cannot provide a characterization of its limit point. 
\end{remark}

\an{For the convergence of RSSA to $X^*$ in an almost sure sense we have to suitably choose stepsize and the parameters values.
Specifically, there are three user-defined sequences in the RSSA scheme: the stepsize
sequence $\{\g_k\}$, the regularization sequence 
$\{\eta_k\}$, and the smoothing sequence $\{\e_k\}$. At each
iteration, all three parameters are updated.  
%For example, if the stepsize sequence decays to zero  faster than 
	%the other two sequences, the solution iterate $x_k$ may not converge to {$X^*$}. 
We next present the {underlying conditions} on these
	sequences.}
\begin{assumption}\label{assum:IRLSA} Let the following hold: \\
(a) \ $\{\g_k\}$, $\{\eta_k\}$, and $\{\e_k\}$ are positive sequences {for $k \geq 0$} converging to zero;\\
(b) \  There exists $K_1\geq 0$ such that $\frac{\g_k}{\eta_k\epsilon_k^2} \leq 0.5\left(\frac{(n-1)!!}{n!! \kappa C}\right)^2$ for any $k \geq K_1$, {where $n$ is the dimension of the space and $\kappa=1$ if $n$ is odd and $\kappa=\frac{2}{\pi}$ otherwise; }\\
(c) \  {For any $k \geq 0$, $\e_k \leq \e$, where $\e$ is the parameter of the set $X^\e$;}\\
(d) \  $\sum_{k=0}^{\infty}\g_k\eta_k =\infty$; \ \ (e) \
	$\sum_{k=0}^{\infty}\g_k^2 <\infty$; \ \ (f) \  $\sum_{k=0}^{\infty}
	\frac{1}{\eta_{k-1}^2\eta_k\g_k}\left(1-\frac{\min\{\e_k,\e_{k-1}\}}{\max\{\e_k,\e_{k-1}\}}
			\right)^2<\infty$;  \ \ \\ (g) \  $\sum_{k=0}^{\infty} \frac{1}{\eta_k\g_k}\left(1-\frac{\eta_k}{\eta_{k-1}}\right)^2<\infty$;\\
(h) \  $\lim_{k \rightarrow \infty} \frac{\g_k}{\eta_k}=0$; \ \ (i) \  $\lim_{k \rightarrow \infty} \frac{1}{\eta_k^2\g_k}\left(1-\frac{\min\{\e_k,\e_{k-1}\}}{\max\{\e_k,\e_{k-1}\}} \right)=0$; \ \  (j) \  $\lim_{k \rightarrow \infty} \frac{1}{\eta_k\g_k}\left|1-\frac{\eta_k}{\eta_{k-1}}\right|=0$.
\end{assumption}

The notation ``!!'' in the condition in~(b) denotes the double factorial.
Later on, in Lemma
	\ref{lemma:stepsize_a.s.}, we provide \an{some choices for the
	sequences $\{\g_k\}$, $\{\eta_k\}$, and $\{\e_k\}$ that satisfy
	 Assumption \ref{assum:IRLSA}.}

\begin{remark}\label{rem:strong}
If we neither regularize nor smooth, then our scheme reduces
to the SA scheme and the necessary conditions for the almost-sure
convergence would be $\sum_{k=0}^{\infty}\g_k = \infty$,
$\sum_{k=0}^{\infty}\g_k^2 <\infty$, as well as the Lipschitzian property
	and the strong monotonicity of the mapping $F$ (cf.~\cite{Houyuan08}). 
\end{remark}

Next, we establish a recursive relation that relates a bound on the
	difference between $x_{k+1}$ and $s_k$ with that from the prior
		step. Such a relation essentially captures the distance of the
		sequence generated by the RSSA scheme  from the regularized
		smoothed trajectory $\{s_k\}$, \an{which} is {important in
		our proof of the a.s.\ convergence.}
\begin{lemma}[{\bf A recursive relation for $\|x_{k+1}-s_k\|$}]\label{lemma:err_ineq}
\an{Consider the RSSA scheme and 
%\an{where the sequences $\{\g_k\}$, $\{\eta_k\}$,and $\{\e_k\}$ sequences {of  positive scalars}. 
let Assumptions
\ref{assum:step_error_sub_1}, \ref{assum:step_error_sub_2},
\ref{assum:IRLSA}(b), and \ref{assum:IRLSA}(c) hold. 
Also, assume that there exists $K_2\geq 0$ such that  $\eta_k\g_k
<1$ for all $k \geq K_2$.} {Then, with $K_1$ given by Assumption~\ref{assum:IRLSA}(b),} the
following relation holds a.s.\  for any $k \geq \max\{K_1,K_2\}$:
\begin{align}
\EXP{\|x_{k+1}-s_k\|^2\mid\sF_k} &\leq
	\left(1-\frac{1}{2}\eta_k\g_k\right)\|x_{k}-s_{k-1}\|^2+2C^2\g_k^2+4M^2\eta_k^2\g_k^2\notag
	\\
&+16n^2C^2\left(1-\frac{\min\{\e_k,\e_{k-1}\}}{\max\{\e_k,\e_{k-1}\}}
		\right)^2\frac{1}{\eta_{k-1}^2\eta_k
	\g_k}+4M^2\left(1-\frac{\eta_k}{\eta_{k-1}} \right)^2\frac{1}{\eta_k
		\g_k}.\label{aslemma}
\end{align}
\end{lemma}
\begin{proof} 
Using the fixed point property of the projection operator
	at the solution $s_k$ of the  VI($X,F_k+\eta_k\bf{I})$, we
	\us{may} write $s_k=\Pi_{X}(s_k-\g_k(F_k(s_k)+\eta_ks_k))$.
	Employing the non-expansivity property of the projection operator, the preceding relation, and 
	the~\ref{algorithm:IRLSA} algorithm, we obtain
\begin{align}\label{equ:vv}
\|x_{k+1}-s_k\|^2 & 
\leq \|x_k-\g_k(F(x_k+z_k)+\eta_kx_k+w_k)-s_k+\g_k(F_k(s_k)+\eta_ks_k)\|^2\cr
&=  \|(1-\eta_k\g_k)(x_k-s_k)-\g_k(F(x_k+z_k)-F_k(s_k))-\g_kw_k)\|^2\cr
&=  (1-\eta_k\g_k)^2\|x_{k}-s_k\|^2+\g_k^2\|F(x_k+z_k)-F_k(s_k)\|^2+\g_k^2\|w_k\|^2\cr
&-2\g_k(1-\eta_k\g_k)(x_k-s_k)^T(F(x_k+z_k)-F_k(s_k)) \cr
& -2\fyRev{\g_k}\Bigl((1-\eta_k\g_k)(x_k-s_k)-\g_k(F(x_k+z_k)-F_k(s_k))\Bigr)^Tw_k.
\end{align}
{Using the iterated expectation rule and Lemma~\ref{lemma:bound-on-errors},
we find that
\[\EXP{\Bigl((1-\eta_k\g_k)(x_k-s_k)-\g_k(F(x_k+z_k)-F_k(s_k))\Bigr)^Tw_k\mid\sF_k}=0.\]
Thus, by taking the conditional expectations conditioned on $\sF_k$ in~\eqref{equ:vv} and using 
$\EXP{\|w_k\|^2 \mid\sF_k}\le C^2$ (cf.\ Lemma~\ref{lemma:bound-on-errors}),
we obtain
\begin{align*}
\EXP{ \|x_{k+1}-s_k\|^2 \mid\sF_k}
&\le  (1-\eta_k\g_k)^2\|x_{k}-s_k\|^2+\g_k^2 \EXP{ \|F(x_k+z_k)-F_k(s_k)\|^2\mid\sF_k} +\g_k^2 C^2\cr
&-2\g_k(1-\eta_k\g_k)(x_k-s_k)^T(\EXP{F(x_k+z_k)\mid\sF_k}-F_k(s_k)).
\end{align*}
Noting that
		$\EXP{F(x_k+z_k)\mid \sF_k}=F_k(x_k)$ and using the monotonicity of $F_k$ 
		%with constant $\kappa\frac{n!!}{(n-1)!!}\frac{C}{\e_k}$ 
		{(see Lemma~\ref{lemma:F_k_monotone}(c))},
		we further obtain $(x_k-s_k)^T(\EXP{F(x_k+z_k)\mid\sF_k}-F_k(s_k))\ge0$.
		Since $\eta_k\g_k <1$ for any $k \geq K_2$, the term
		$(1-\eta_k\g_k)$ is
	positive
		 implying that, for all $k\ge K_2$ we have almost surely,
\begin{align}\label{equ:oo}
\EXP{ \|x_{k+1}-s_k\|^2 \mid\sF_k}
\le  (1-\eta_k\g_k)^2\|x_{k}-s_k\|^2+\g_k^2 \EXP{ \|F(x_k+z_k)-F_k(s_k)\|^2\mid\sF_k} +\g_k^2 C^2.
%&-2\g_k(1-\eta_k\g_k)\kappa\frac{n!!}{(n-1)!!}\frac{C}{\e_k} \|x_k-s_k\|^2.
\end{align}
To estimate $\EXP{ \|F(x_k+z_k)-F_k(s_k)\|^2\mid\sF_k} $ we add and subtract $F_k(x_k)$,
which yields
\begin{align*}
\|F(x_k+z_k)-F_k(s_k)\|^2
& =
\|F(x_k+z_k)-F_k(x_k)\|^2+ \|F_k(x_k)-F_k(s_k)\|^2 \cr
& +2(F(x_k+z_k)   -F_k(x_k))^T(F_k(x_k)-F_k(s_k))\cr
&\le \|F(x_k+z_k)-F_k(x_k)\|^2+ \fyRev{\left(\kappa\frac{n!!}{(n-1)!!}\frac{C}{\e_k}\right)^2}\|x_k-s_k\|^2\cr
& +2(F(x_k+z_k)   -F_k(x_k))^T(F_k(x_k)-F_k(s_k)),
\end{align*}
where the second inequality follows by the Lipschitz continuity of $F_k$ with
		constant $\kappa\frac{n!!}{(n-1)!!}\frac{C}{\e_k}$ {(see Lemma~\ref{lemma:F_k_monotone}(b))}.
Taking expectations conditioned on $\sF_k$ and using $F_k(x_k)=\EXP{F(x_k+z_k)\mid \sF_k}$,
we find that almost surely for all $k\ge K_2$,
\begin{align*}%\label{equ:one}
\EXP{ \|F(x_k+z_k)-F_k(s_k)\|^2\mid\sF_k} 
& \le 
\EXP{ \|F(x_k+z_k)-F_k(x_k)\|^2\mid\sF_k} + \fyRev{\left(\kappa\frac{n!!}{(n-1)!!}\frac{C}{\e_k}\right)^2}\|x_k-s_k\|^2\cr
&=
\EXP{ \|F(x_k+z_k)\|^2 \mid\sF_k} -\|F_k(x_k)\|^2 + \fyRev{\left(\kappa\frac{n!!}{(n-1)!!}\frac{C}{\e_k}\right)^2}\|x_k-s_k\|^2\cr
&\le C^2 +\fyRev{\left(\kappa\frac{n!!}{(n-1)!!}\frac{C}{\e_k}\right)^2}\|x_k-s_k\|^2,\quad
\end{align*}
where the last inequality is obtained by using $\|F(x_k+z_k)\|\le C$ (see Remark~\ref{rem:C-M})
and by ignoring the negative term.
Substituting the preceding estimate  in the relation~\eqref{equ:oo}, we obtain a.s.\ for $k\ge K_2$,
\begin{align}\label{cond-e3}
\EXP{ \|x_{k+1}-s_k\|^2 \mid\sF_k}
&\le  \left( (1-\eta_k\g_k)^2 + \g_k^2\left(\kappa\frac{n!!}{(n-1)!!}\frac{C}{\e_k}\right)^2\right)\|x_{k}-s_k\|^2 + 2\g_k^2 C^2\cr
& =
	\left(1-2\eta_k\g_k+\eta_k^2\g_k^2+\g_k^2\left(\kappa\frac{n!!}{(n-1)!!}\frac{C}{\e_k}\right)^2\right)\|x_{k}-s_k\|^2
	+ 2\g_k^2C^2.
\end{align}
Using the definition of $M$ in {Remark \ref{rem:C-M}} and the triangle
	inequality, we may write $\|x_k-s_k\|\leq \|x_k\|+\|s_k\| \leq 2M$, 
	which}
leads to the following bound on  $\eta_k^2 \g_k^2
			\|x_k-s_k\|^2,$
$$ 	\eta_k^2 \g_k^2 \|x_k-s_k\|^2 \le 4 \eta_k^2 \g_k^2 M^2. $$
{This inequality and relation~\eqref{cond-e3} yield a.s.\ for all $k\ge K_2$,}
\begin{align}\label{ineq:sk_bound1}
\EXP{\|x_{k+1}-s_k\|^2\mid\sF_k} & \leq
	\left(1-2\eta_k\g_k+\g_k^2\left(\kappa\frac{n!!}{(n-1)!!}\frac{C}{\e_k}\right)^2\right)\|x_{k}-s_k\|^2\cr
	&
	 +2\g_k^2C^2+4\eta_k^2\g_k^2M^2.\end{align}

To obtain a recursion,  we need to
		estimate the term $\|x_k-s_k\|$ in terms of $\|x_k-s_{k-1}\|$.
		Using the triangle inequality, we may write
$\|x_k-s_k\|  \leq \|x_k-s_{k-1}\|+\|s_k-s_{k-1}\|$. Therefore, we obtain
\begin{align}\label{ineq:sk_bound2}
\|x_k-s_k\|^2 & \leq \|x_k-s_{k-1}\|^2+\|s_k-s_{k-1}\|^2+2\|s_k-s_{k-1}\|\|x_k-s_{k-1}\|.
\end{align}
%\cr
%& \leq \|x_k-s_{k-1}\|+\frac{2nC}{\eta_k}\left(1-\frac{\e_k}{\e_{k-1}} \right)+M_y\left(1-\frac{\eta_k}{\eta_{k-1}} \right)
Using {the} relation $2ab \leq a^2+b^2$, for $a,b \in  R$, we have
\begin{align*}
2\|s_k-s_{k-1}\|\|x_k-s_{k-1}\|&= 2\left(\sqrt{\eta_k \g_k}\|x_k-s_{k-1}\|\right)\left(\frac{\|s_k-s_{k-1}\|}{\sqrt{\eta_k \g_k}}\right)\cr & \leq \eta_k \g_k\|x_k-s_{k-1}\|^2+ \frac{\|s_k-s_{k-1}\|^2}{\eta_k \g_k}.
\end{align*}
Combining this result, Proposition \ref{prop:sk_estimate}(a), and (\ref{ineq:sk_bound2}), we obtain for all $k\ge K_2$,
\begin{align}\label{ineq:sk_bound3}
\|x_k-s_k\|^2  & \leq  (1+\eta_k
		\g_k)\|x_k-s_{k-1}\|^2\notag \\ & +2\left(\frac{2nC}{\eta_{k-1}}\left(1-\frac{\min\{\e_k,\e_{k-1}\}}{\max\{\e_k,\e_{k-1}\}}
				\right) +M\left
|1-\frac{\eta_k}{\eta_{k-1}} \right|\right)^2\frac{1}{\eta_k \g_k},
\end{align}
where in the last inequality we used
	$1+{1}\slash({\eta_k \g_k}) < {2}\slash{\eta_k \g_k}$ as a
	consequence of $\g_k\eta_k <1$ for {$k\ge K_2$}. 
	If $q_k$ is defined as 
	$$q_k \triangleq 1- 2\eta_k\g_k
	+\g_k^2\left(\kappa\frac{n!!}{(n-1)!!}\frac{C}{\e_k}\right)^2,$$
	then inequalities (\ref{ineq:sk_bound1}) and (\ref{ineq:sk_bound3})
	imply that for $k \geq K_2$, the following relation holds:
\begin{align}\label{ineq:sk_bound4}
\EXP{\|x_{k+1}-s_k\|^2\mid\sF_k} &\leq q_k(1+\eta_k\g_k)\|x_{k}-s_{k-1}\|^2+2C^2\g_k^2+4M^2\eta_k^2\g_k^2\cr
&+2q_k\left(\underbrace{\frac{2nC}{\eta_{k-1}}\left(1-\frac{\min\{\e_k,\e_{k-1}\}}{\max\{\e_k,\e_{k-1}\}} \right)}_{a}+\underbrace{M\left
|1-\frac{\eta_k}{\eta_{k-1}} \right|}_{b}\right)^2\frac{1}{\eta_k \g_k}. 
\end{align}	
By Assumption \ref{assum:IRLSA}(b), we can write for $k\geq K_1$,
\begin{align*}
\frac{\g_k}{\eta_k\epsilon_k^2} \leq 0.5\left(\frac{(n-1)!!}{n!! \kappa
		C }\right)^2 & \implies
	\g_k^2\left(\kappa\frac{n!!}{(n-1)!!}\frac{C}{\e_k}\right)^2 \leq
	\frac{\eta_k \g_k}{2} \\
	& \implies  - 2\eta_k\g_k +\g_k^2\left(\kappa\frac{n!!}{(n-1)!!}\frac{C}{\e_k}\right)^2 \leq - \frac{3}{2}\eta_k \g_k. 
\end{align*}
Therefore, $q_k \leq 1 - \frac{3}{2}\eta_k \g_k$. Consequently, we
	may provide an upper bound on $q_k(1+\eta_k\g_k)$ using the preceding relation: 
\begin{align*}
q_k(1+\eta_k\g_k) \leq  (1 - \frac{3}{2}\eta \g_k)(1+\eta_k\g_k)= 1-\frac{1}{2}\eta_k\g_k-\frac{3}{2}\eta_k^2\g_k^2 \leq  1-\frac{1}{2}\eta_k\g_k.
\end{align*}
Using relation (\ref{ineq:sk_bound4}) and $q_k \leq 1$ (which follows by $q_k \leq 1 - \frac{3}{2}\eta_k \g_k$), {and $(a+b)^2 \leq 2a^2+2b^2$, we conclude that the desired relation holds.}
\end{proof}

The following supermartingale convergence theorem is a key in our
	analysis in establishing the almost sure convergence of the RSSA
		scheme and
may be found in~\cite{Polyak87} (cf.~Lemma 10, page 49).  
\begin{lemma}[Robbins and Siegmund Lemma]\label{lemma:probabilistic_bound_polyak}
Let $\{v_k\}$ be a sequence of nonnegative random variables, where 
$\EXP{v_0} < \infty$, and let $\{\a_k\}$ and $\{\mu_k\}$
be deterministic scalar sequences such that \fy{$0 \leq \alpha_k \leq 1$, and $\mu_k \geq 0$ for all $k\ge0$, $\sum_{k=0}^\infty \alpha_k =\infty$, $\sum_{k=0}^\infty \mu_k < \infty$, and $\lim_{k\to\infty}\,\frac{\mu_k}{\alpha_k} = 0$, and }
$\EXP{v_{k+1}|v_0,\ldots, v_k} \leq (1-\alpha_k)v_k+\mu_k$ a.s. for all $k\ge0$.
Then, $v_k \rightarrow 0$ almost surely as $k \rightarrow \infty$.
\end{lemma}

We are now ready to present the main convergence result showing that
	the sequence generated by the RSSA scheme has its accumulation
		points in the solution set $X^*$ of the original VI$(F,X)$
		almost surely. Under the assumption that $\e_k/\eta_k\to0$ and 
		suitable local requirements, we may further 
		claim that the sequence converges to the smallest norm solution
		in $X^*$ almost surely.
	
\begin{theorem}[{\bf Almost sure convergence of RSSA scheme}]\label{prop:almost-sure}
Let Assumptions~\ref{assum:step_error_sub_1},
	\ref{assum:step_error_sub_2} and~\ref{assum:IRLSA} hold, and let
	$\{x_k\}$ be given by the RSSA scheme. Then, the following statements hold almost surely:\\
	(a)\ $\lim_{k\rightarrow
		\infty}\|x_{k+1}-s_k\|=0$ and \uvs{every} accumulation point of $\{x_k\}$ is a solution of VI$(X,F)$.\\
	(b)\ \an{Let $t^*$ be the smallest norm solution of VI$(X,F)$.
	If $\lim_{k\to\infty}\frac{\e_k} {\eta_k}=0$ and $F$ is
differentiable at $t^*$ with a bounded Jacobian in a neighborhood of
$t^*$, then $\{x_k\}$ converges to $t^*$.}
\end{theorem}

\begin{proof} (a) \ 
From Assumption \ref{assum:IRLSA}(a), $\g_k$ and $\eta_k$ go to zero. Thus, there exists a constant $K_2 \geq 0$ such that $\g_k\eta_k < 1$ for any $k \geq K_2$. \fyRev{Let us define sequences }$\{v_k\}$, $\{\alpha_k\}$, and $\{\mu_k\}$ for $k \geq \fytwo{\max\{K_1,K_2\}}$ \fy{given by $v_k \triangleq \|x_{k}-s_{k-1}\|, \ \alpha_k \triangleq \frac{1}{2}\g_k\eta_k$ and} 
\begin{align*}
  \fy{\mu_k \triangleq 2C^2\g_k^2+4M^2\eta_k^2\g_k^2+16n^2C^2\overbrace{\left(1-\frac{\min\{\e_k,\e_{k-1}\}}{\max\{\e_k,\e_{k-1}\}} \right)^2\frac{1}{\eta_{k-1}^2\eta_k \g_k}}^{\mbox{Term 1}}+4M^2\overbrace{\left(1-\frac{\eta_k}{\eta_{k-1}} \right)^2\frac{1}{\eta_k \g_k}}^{\mbox{Term 2}}.}
\end{align*}
Therefore, Lemma \ref{lemma:err_ineq} implies that $\EXP{v_{k+1}\mid { \sF_k}} 
\le (1-\alpha_k)v_k  + \mu_k,\hbox{ for } k \geq
\fytwo{\max\{K_1,K_2\}}.$ To claim convergence of the sequence
$\{x_k\}$, we show that conditions of Lemma
\ref{lemma:probabilistic_bound_polyak} hold. The nonnegativity of $v_k$,
	$\alpha_k$, and $\mu_k$ for $k \geq \fytwo{\max\{K_1,K_2\}}$ is
	trivial. Assumption \ref{assum:IRLSA}(d) indicates that the
	condition $\sum_{k}\alpha_k= \infty$ is satisfied. On the other
	hand, positivity of $\g_k$ and $\eta_k$ indicates that $\alpha_k
	\leq 1$ holds for $k \geq  \fytwo{\max\{K_1,K_2\}}$. Since $\eta_k$
	goes to zero, there exists a bound $\bar \eta$ such that $\eta_k
	\leq \bar \eta$ \uvs{for $k \geq \max\{K_1,K_2\}$}. Therefore, \fy{$\mu_k \leq  (2C^2+4M^2\bar
			\eta^2)\g_k^2+ 16n^2C^2(\mbox{Term 1}) +4M^2(\mbox{Term
				2})$.} Assumptions \ref{assum:IRLSA}(e), (f), and (g)
	show that \fy{$\g_k^2$, $\mbox{Terms 1 and 2}$} are summable.
	Therefore, we conclude that $\mu_k$ is summable too. It remains to
	show that $\lim_{k\rightarrow \infty}\frac{\mu_k}{\alpha_k}=0$. It
	suffices to show that $$\lim_{k\rightarrow
		\infty}\frac{\g_k^2}{\alpha_k}=0, \lim_{k\rightarrow
			\infty}\frac{\mbox{Term 1}}{\alpha_k}=0, \mbox{ and }
			\lim_{k\rightarrow \infty}\frac{\mbox{Term 2}}{\alpha_k}=0.$$ These three conditions hold due to Assumption \ref{assum:IRLSA} (h), (i), and (j) respectively. In conclusion, all of the conditions of Lemma \ref{lemma:probabilistic_bound_polyak} hold and thus $\|x_{k}-s_k\|$ goes to zero almost surely. 
			
Since $F$ is continuous and $\eta_k$ and $\e_k$ go to zero, Proposition \ref{prop:sk_estimate}(b1) 
implies that any limit point of the sequence $\{s_k\}$ converges to a solution of VI$(X,F)$. Hence, from the result of part (a), we conclude that any accumulation point of the sequence $\{x_k\}$ generated by {the \ref{algorithm:IRLSA-impl} algorithm} converges to a solution of VI$(X,F)$ almost surely.\\
{(b) \ 
The statement in part (b) follows by part (a) and Proposition~\ref{prop:sk_estimate}(b2).
}
\end{proof}

A reader might question whether Assumption \ref{assum:IRLSA} is
	vacuous in that there are no set of sequences satisfying the
		required assumptions. We \an{show} that this is not the case by
		showing that there is a set of stepsize, regularization, and
		smoothing sequences that satisfy the given conditions. 
\begin{lemma}\label{lemma:stepsize_a.s.}
Suppose sequences $\{\g_k\}$, $\{\eta_k\}$, and $\{\e_k\}$ are given by $\g_k=\g_0(k+1)^{-a}$, $\eta_k=\eta_0(k+1)^{-b}$, and $\e_k=\e_0(k+1)^{-c}$ where $a$, $b$, and $c$ satisfy the following conditions:
\begin{align*}
\fytwo{a,b,c>0, \quad a+3b<1, \quad a > b+2c, \quad a > 0.5},
\end{align*}
and $\g_0$, $\eta_0$, $\e_0$ are \fytwo{strictly} positive scalars and  \fytwo{$\e_0 = \e$}.
Then, sequences $\{\g_k\}$, $\{\eta_k\}$, and $\{\e_k\}$ satisfy Assumption \ref{assum:IRLSA}. 
\end{lemma}

\begin{proof} We show that each part of Assumption \ref{assum:IRLSA} holds as follows:\\
%\begin{itemize}
(a) \ Assumption \ref{assum:IRLSA}(a) holds since $a$, $b$, $c$, $\g_0$, $\eta_0$, and $\e_0$ are \fytwo{strictly} positive. \\
(b) \ To show that \an{condition 3(b)} holds, we note that
$$\frac{\g_k}{\eta_k\epsilon_k^2} =
\frac{\g_0(k+1)^{-a}}{\eta_0(k+1)^{-b}\epsilon_0^2(k+1)^{-2c}}=
(k+1)^{-(a-b-2c)}\frac{\g_0}{\eta_0\epsilon_0^2}.$$
Since {$a > b+2c$, then $(k+1)^{-(a-b-2c)} \rightarrow 0$. Therefore, $\frac{\g_k}{\eta_k\epsilon_k^2} \rightarrow 0$ implying that
there exists $K_1\geq 0$ such that $$\frac{\g_k}{\eta_k\epsilon_k^2}
\leq 0.5\left(\frac{(n-1)!!}{n!! \kappa C}\right)^2$$ for any $k\geq K_1$. This indicates that Assumption \ref{assum:IRLSA}(b) holds.}\\
(c) \ \an{Condition 3(c) is satisfied since} $\e_k \leq \e_0$ for any $k \geq 0$ and \fytwo{$\e_0 = \e$}.\\ 
(d) \ We have
$\sum_{k=0}^\infty\eta_k\g_k=\eta_0\g_0\sum_{k=0}^\infty{1}\slash{(k+1)^{a+b}}.$
Since $a, b >0$ and \fytwo{$a+3b < 1$, then $a+b < 1$. Thus,}
$\sum_{k=0}^\infty{1}\slash{(k+1)^{a+b}}=\infty$. Therefore, Assumption \ref{assum:IRLSA}(d) is met. \\
(e) \ To \an{verify that condition 3(e)} holds, \us{it suffices to show} that $\g_k^2$ is
summable. We have $\g_k^2=\g_0^2(k+1)^{-2a}$ and $2a>1$ since $a >0.5$.
Therefore, $\g_k^2$ is summable.\\
(f) \ Note that sequences $\{\eta_k\}$ and $\{\e_k\}$ are both decreasing. Therefore, 
\[ \frac{1}{\eta_{k-1}^2\eta_k\g_k}\left(1-\frac{\min\{\e_k,\e_{k-1}\}}{\max\{\e_k,\e_{k-1}\}} \right)^2=\frac{1}{\eta_{k-1}^2\eta_k\g_k}\left(1-\frac{\e_k}{\e_{k-1}} \right)^2 <  \frac{1}{\eta_k^3\g_k}\left(1-\frac{\e_k}{\e_{k-1}} \right)^2\triangleq \mbox{Term 1}.\]
\fy{It suffices to show that} $\mbox{Term 1}$ is summable. First, we
estimate $1-{\e_k}\slash{\e_{k-1}}$. We have $$1-\frac{\e_k}{\e_{k-1}}=
1-\frac{\e_0(k+1)^{-c}}{\e_0k^{-c}}=1-\left(\frac{k}{k+1}\right)^c=1-\left(1-\frac{1}{k+1}\right)^c.$$
Recall that the Taylor expansion of $(1-x)^{p}$ for $|x|<1$ and any
scalar $p$ is given by $$(1-x)^{p}=\sum_{j=0}^\infty
(-1)^j\left(\begin{array}{c} p\\ j
		\end{array}\right)x^j=1-px+\frac{p(p-1)}{2}x^2-\frac{p(p-1)(p-2)}{6}x^3+\cdots.$$
Using this expansion for $x=\frac{1}{k+1}$ and $p=c$, we have
\[1-\frac{\e_k}{\e_{k-1}}=1 -
\left(1-c\frac{1}{k+1}+\frac{c(c-1)}{2}\frac{1}{(k+1)^2}-\frac{c(c-1)(c-2)}{6}\frac{1}{(k+1)^3}+\cdots\right)=\us{\cal
	O}(k^{-1}).\]
Therefore, from the preceding relation, we obtain $$\mbox{Term
	1}=\frac{\us{\cal O}(k^{-2})}{\eta_0^3\g_0(k+1)^{-3b-a}}=\us{\cal O}(k^{-(2-a-3b)}).$$
To \us{ensure summability of $\mbox{Term 1}$, it suffices that
$2-a-3b>1$ or equivalently $a+3b<1$. This holds by assumption and
\an{condition 3(f)} is met.}\\ 
%\[\left(1- \frac{\e_{k}}{\e_{k-1}}\right)^2\,\frac{1}{\eta_k\g_k}=\frac{O(k^{-2})}{k^{-a-b}}=O(k^{-1+(a+b-1)}).\]
%Since $a+3b<1$, and $b>0$, we have $a+b-1<0$. Hence, from the preceding relation it follows that Assumption \ref{assum:IRLSA}(d) is satisfied.
%By a similar argument, we have
% \[\left(1- \frac{\e_{k}}{\e_{k-1}}\right)^2\,\frac{1}{\eta_k^3\g_k}=\frac{O(k^{-2})}{k^{-a-3b}}=O(k^{-1+(a+3b-1)}).\]
%The assumption $a+3b<1$ implies that $a+3b-1<0$. Hence, Assumption \ref{assum:IRLSA}(e) is satisfied.\\
(g) \ In a similar fashion that we used in part (f), we can show that
$1-\frac{\eta_k}{\eta_{k-1}}=\us{\cal O}(k^{-1})$. Consider Term 3 defined as
follows:
$$\mbox{Term 3}\triangleq \frac{1}{\eta_k\g_k}\left( 1-
		\frac{\eta_k}{\eta_{k-1}} \right)^2
\frac{\us{\cal O}(k^{-2})}{\eta_0\g_0(k+1)^{-(a+b)}}=\us{\cal O}(k^{-(2-a-b)}).$$
To show that condition (g) is satisfied, it suffices to show that
$\mbox{Term 3}$ is summable. From the preceding relation, we need to
show that $2-a-b>1$ or equivalently $a+b<1$. We assumed that $a+3b <1$
and $b>0$. Thus, we have $a+b=a+3b-2b<1-2b<1$. Therefore, $\us{\cal O}(k^{-(2-a-b)})$ is summable and 
we conclude that \an{condition 3(g) is satisfied}.  \\
(h) \ \us{We have ${\g_k}\slash{\eta_k}=
{\g_0(k+1)^{-a}}\slash{(\eta_0(k+1)^{-b})}=({\g_0}\slash{\eta_0})(k+1)^{-(a-b)}.$ 
To show that ${\g_k}\slash{\eta_k}$ goes to zero when $k$ goes to
infinity, we only need to show that $a>b$. We assumed that $a+3b < 1$.
Therefore, $b<(1-a)\slash 3$. Since $a>0.5$, the preceding relation
yields $b<1\slash 6$. Thus, $b<0.5<a$, implying that \an{condition 3(h)}
	holds.} \\
(i) \ From part (f), we have $1-{\e_{k}}\slash{\e_{k-1}}=\us{\cal O}(k^{-1})$. To show the condition 3(i), we write
\[\mbox{Term
	4}\triangleq\frac{1}{\eta_k^2\g_k}\left(1-\frac{\min\{\e_k,\e_{k-1}\}}{\max\{\e_k,\e_{k-1}\}}
			\right)=\frac{1}{\eta_0^2\g_0(k+1)^{-a-2b}}\us{\cal
		O}(k^{-1})=\us{\cal O}(k^{-(1-a-2b)}).\]
Thus, it suffices to show that $a+2b<1$. This is true since $a+3b<1$ and $b>0$. Hence, $\mbox{Term 4}$ 
goes to zero implying that \an{the condition 3(i)} holds.\\
(j) \ Term 5 is defined as  $$\mbox{Term 5}\triangleq
{\frac{1}{\eta_k\g_k}\left|1-\frac{\eta_k}{\eta_{k-1}}\right|}=\frac{1}{\eta_0\g_0(k+1)^{-a-b}}\us{\cal
	O}(k^{-1})=\us{\cal O}(k^{-(1-a-b)}).$$ Since $a+3b<1$ and $b>0$, we have $a+b<1$, showing that $\mbox{Term 5}$ converges to zero.
\end{proof}

{In order to satisfy the additional condition
$\lim_{k\to\infty}\frac{\e_k}{\eta_k}=0$ used in
	Proposition~\ref{prop:sk_estimate}(b2) and
	Theorem~\ref{prop:almost-sure}(b), one would need an additional
	requirement $b-c<0$ in Lemma~\ref{lemma:stepsize_a.s.}.  As a
	concrete example satisfying the conditions in
	Lemma~\ref{lemma:stepsize_a.s.}, consider the choice
	$a=\frac{9}{16},$ $b=\frac{2}{16}$, and $c=\frac{3}{16}$.  In this
	case we also have $\lim_{k\to\infty}\frac{\e_k}{\eta_k}=0$ since
	$b-c<0$.}

%---------------------------------------------------------------------------------------------------------------
\subsection{\us{Rate of convergence to regularized smoothed trajectory}}\label{sec:mean}
%---------------------------------------------------------------------------------------------------------------
Thus far, we have discussed the convergence of the
	sequence $\{x_k\}$ generated by the RSSA scheme
	in an almost sure sense. Naturally, one may
	be curious about the rate at which these iterates converge to the solution set. 
	\an{In the existing literature, the
	development of non-asymptotic rates of convergence has been
	provided either in terms of mean-squared
			error for solution iterates or in terms of the mean gap
			function (see Table~\ref{tab:SVI_algorithms}). 
			However, we are unaware of any statements provided in
	non-Lipschitzian and merely monotone regimes in terms of solution
	iterates. In this subsection, we provide a partial answer to this
	question.}

Our metric of convergence is the distance from the solution set $X^*$, denoted by $\mbox{dist}(x_k,X^*)$, and  
{the question is at
what rate the error dist$\big(x_k,X^*\big)$ will diminish to zero.} {We
	may provide a partial answer by establishing}
the rate at which the sequence $\{x_k\}$ approaches
the regularized smoothed trajectory $\{s_k\}$. The idea is as follows: At
step $k$, instead of comparing the iterate $x_k$ with a true solution
$x^*$, we want to estimate the distance between $x_k$ and the
approximate solution $s_k$. Note that, as the algorithm proceeds, we
expect $s_k$ to be approaching the solution set $X^*$ (according to Proposition~\ref{prop:sk_estimate}). 
%The first part of this section
%provides such an analysis and we derive a generic bound for this dynamic error. 
\an{We begin the discussion by providing an assumption on the
sequences}. %This set of assumptions is essential for deriving the particular rate.
\begin{assumption}\label{assum:conv-exp}
Let the following hold:\\
(a) \  There exist $\delta \in (0,0.5)$ and $K_3\geq 0$ such that for any $k \geq K_3,$
 $$\frac{\g_k}{\eta_k\e_k^2} \leq
 \frac{\g_{k+1}}{\eta_{k+1}\e_{k+1}^2}(1+\delta \eta_{k+1}\g_{k+1});$$ \\
(b) \ There exists a constant $B_1>0$ such that for any $k \geq 0$,
$$\frac{\e_k^2}{\eta_{k-1}^2\eta_k\g_k^3}\left(1-\frac{\min\{\e_k,\e_{k-1}\}}{\max\{\e_k,\e_{k-1}\}}
		\right)^2\leq B_1;$$\\
(c) \ There exists a constant $B_2>0$ such that for any $k \geq 0$,
$$\frac{\e_k^2}{\eta_k\g_k^3}\left(1-\frac{\eta_k}{\eta_{k-1}}\right)^2\leq
{B_2}.$$
\end{assumption}

The following result provides a bound on the error that relates the iterates $\{x_k\}$ and the approximate sequence $\{s_k\}$. 
This result provides us an estimate of the performance of our algorithm with respect to the iterates of the solutions to the approximated problems VI$(X,F_k+\eta_k\bf{I})$.

\begin{proposition}[{\bf An upper bound for $\EXP{\|x_{k+1}-s_k\|^2}$}]\label{prop:UB-MSE}
Consider the RSSA scheme where $\{\g_k\}$, $\{\eta_k\}$, and $\{\e_k\}$ are {strictly} positive sequences. Let Assumptions \ref{assum:step_error_sub_1}, \ref{assum:step_error_sub_2}, \ref{assum:IRLSA}(b), \ref{assum:IRLSA}(c), and \ref{assum:conv-exp} hold. Suppose $\{\eta_k\}$ is bounded by some $\bar \eta>0$ and there exists some scalar $K_2\geq 0$ such that for any $k \geq K_2$ we have $\eta_k\g_k <1$. Then, 
\begin{align}\label{ineq:MSE-result}
\EXP{\|x_{k+1}-s_k\|^2}\leq \theta \frac{\g_k}{\eta_k\e_k^2}, \qquad \hbox{ for any }k \geq \bar K,
\end{align}
where $\bar K \triangleq \fytwo{\max\{K_1,K_2,K_3\}}$, $s_k$ is the unique solution of VI$(X, F_k+\eta_k\bf{I})$, $K_1$ and $K_3$ are given by Assumptions \ref{assum:IRLSA}(b) and {\ref{assum:conv-exp}(a)}, respectively. 
\an{The parameter $\theta$ in~\eqref{ineq:MSE-result} is such that}
\begin{align}\label{def:theta}
\theta\ge\max\left(4M^2\frac{\eta_{\bar K}\e_{\bar K}^2}{\g_{\bar K}}, \frac{2C^2 \e^2+4M^2\bar \eta^2\e^2 +16n^2C^2B_1
+4M^2B_2}{0.5-\delta}\right).
\end{align}
\end{proposition}
\begin{proof}
We begin by employing Lemma~\ref{lemma:err_ineq} and by letting
\an{$e_k\triangleq\EXP{\|x_{k}-s_{k-1}\|^2}$ for all $k$}. 
Taking expectations on both sides of~\eqref{aslemma} in Lemma
	\ref{lemma:err_ineq}, we obtain a recursion in terms of the mean
		squared error $e_k$. For any $k \geq \bar K +1$ we have
 \begin{align}\label{ineq:e_k-bound}
e_{k+1} &\leq
\left(1-\frac{1}{2}\eta_k\g_k\right)e_k+2C^2\g_k^2+4M^2\eta_k^2\g_k^2\notag \\
	& +16n^2C^2\frac{\left(1-\frac{\min\{\e_k,\e_{k-1}\}}{\max\{\e_k,\e_{k-1}\}} \right)^2}{\eta_{k-1}^2\eta_k\g_k}
+4M^2\frac{\left(1-\frac{\eta_k}{\eta_{k-1}} \right)^2}{\eta_k \g_k}.
\end{align}
To prove the result, we employ the mathematical induction on $k$. The
first step is to show that the result holds for $k=\bar K$. Using the
definition of $M$ in Remark~\ref{rem:C-M} and the Cauchy-Schwartz
inequality, we may write 
 \begin{align*}
 e_{\bar K+1} &=  \EXP{\|x_{\bar K+1}\|^2-2x_{\bar K+1}^Ts_{\bar K}+\|s_{\bar K}\|^2} 
 \leq  \EXP{\|x_{\bar K+1}\|^2+2\|x_{\bar K+1}\|\|s_{\bar K}\|+\|s_{\bar K}\|^2} \cr  
 & \leq  M^2+2M^2+M^2 
 = \left(4M^2\frac{\eta_{\bar K}\e_{\bar K}^2}{\g_{\bar K}}\right)\frac{\g_{\bar K}}{\eta_{\bar K}\e_{\bar K}^2}.\end{align*}
Let us define $\theta_{\bar K}\triangleq 4M^2{\eta_{\bar K}\e_{\bar
	K}^2}\slash{\g_{\bar K}}$. Thus, the preceding relation implies that \an{relation~\eqref{ineq:MSE-result} 
	holds} for $k=\bar K$ with $\theta= \theta_{\bar K}$.
	Now, suppose {$e_{t+1} \leq \theta {\g_t}\slash({\eta_t\e_t^2})$} for
	$\bar K < t \leq k-1$ for some finite constant \an{$\theta\ge\theta _{\bar K}$}. We will
	proceed to show that $e_{k+1} \leq \theta
{\g_k}\slash({\eta_k\e_k^2})$. Using the
	induction hypothesis, relation (\ref{ineq:e_k-bound}), and
	Assumptions~\ref{assum:conv-exp}(b) and \an{\ref{assum:conv-exp}(c)}, we obtain
\begin{align*}
e_{k+1} \leq  \left(1-\frac{1}{2}\eta_k\g_k\right)\theta \frac{\g_{k-1}}{\eta_{k-1}\e_{k-1}^2}+2C^2\g_k^2+4M^2\eta_k^2\g_k^2+16n^2C^2\frac{\g_k^2}{\e_k^2}B_1+4M^2\frac{\g_k^2}{\e_k^2}B_2.
\end{align*}
Using Assumption \ref{assum:conv-exp}(a), we obtain
\begin{align}\label{ineq:e_k_bound1-1}
e_{k+1} \leq   \left(1-\frac{1}{2}\eta_k\g_k\right)(1+\delta\eta_k\g_k)\theta \frac{\g_{k}}{\eta_{k}\e_{k}^2}+2C^2\g_k^2+4M^2\eta_k^2\g_k^2+16n^2C^2\frac{\g_k^2}{\e_k^2}B_1+4M^2\frac{\g_k^2}{\e_k^2}B_2.
\end{align}
\an{Note that, we have}
\begin{align}\label{ineq:e_k_bound1-2}
 \fyRev{\left(1-\frac{1}{2}\eta_k\g_k\right)(1+\delta\eta_k\g_k)\theta \frac{\g_{k}}{\eta_{k}\e_{k}^2}
  = \theta\frac{\g_{k}}{\eta_{k}\e_{k}^2}- \theta \left(\frac{\delta}{2}\right)\frac{\eta_k\g_k^3}{\e_k^2}+\theta \eta_k\g_k\left(-\frac{1}{2}+\delta\right)\frac{\g_{k}}{\eta_{k}\e_{k}^2}.}%+2C^2\g_k^2.
 \end{align}
Using nonpositivity of $- \theta
\left(\frac{\delta}{2}\right)\frac{\eta_k\g_k^3}{\e_k^2}$ and equation
 (\ref{ineq:e_k_bound1-2}), relation (\ref{ineq:e_k_bound1-1})
	can be expressed as follows:
\begin{align}\label{ineq:e_k_bound2}
e_{k+1} \leq
\theta\frac{\g_{k}}{\eta_{k}\e_{k}^2}+\frac{\g_{k}^2}{\e_k^2}\overbrace{\left[-\theta
	\left(\frac{1}{2}-\delta\right)+2C^2 \e^2+4M^2\bar \eta^2\e^2
		+16n^2C^2B_1+4M^2B_2\right],}^{\mbox{Term 1}}
\end{align}
where we invoke the boundednenss of \an{$\eta_k$ and $\e_k$} from above
	by $\bar \eta$ and $\e$, respectively (the latter follows from Assumption
\ref{assum:IRLSA}(c)). To complete the proof, it suffices to show that Term 1  
is nonpositive for some \an{$\theta\ge\theta_{\hat K}$.} By Assumption
\ref{assum:conv-exp}(a), we have $\left(\frac{1}{2}-\delta\right) >0$.
Therefore, if $$\theta \geq  \frac{2C^2 \e^2+4M^2\bar \eta^2\e^2 +16n^2C^2B_1
+4M^2B_2}{0.5-\delta},$$ then $\mbox{Term 1}$ is nonpositive.
Hence, $e_{k+1} \leq
\theta{\g_{k}}\slash({\eta_{k}\e_{k}^2})$ and, thus, the induction
argument is complete. In conclusion, when $\theta$ satisfies relation~(\ref{def:theta}), 
relation~(\ref{ineq:MSE-result}) holds for any $k \geq \bar K$.
\end{proof}

{The following proposition states that the~\ref{algorithm:IRLSA-impl} algorithm generates a sequence converging 
to the solution set of VI$(X,F)$ in a
mean-square sense.}
\begin{proposition}[{\bf Convergence in \fy{mean-squared} distance from solution set}]\label{prop:expectation}
Let Assumptions \ref{assum:step_error_sub_1},
	\ref{assum:step_error_sub_2}, \ref{assum:IRLSA}(a,b,c), and
	\ref{assum:conv-exp} hold. Also, assume that $\lim_{k\rightarrow
		\infty}{\g_k}\slash({\eta_k\e_k^2})=0$, and let $\{x_k\}$ be generated by the
			RSSA scheme. 
			Then, we have the following:\\
			(a)  The sequence $\{x_k\}$
			converges to the solution set $X^*$ of  VI$(X,F)$ in mean-squared sense, i.e.,
			$$\lim_{k\rightarrow
			\infty}\EXP{\hbox{dist}^2\big(x_{k},X^*\big)}=0.$$
			(b) If in addition $\lim_{k\to\infty}\frac{\e_k} {\eta_k}=0$
			and $F$ is differentiable at \an{the
			smallest norm solution $t^*\in X^*$
			with a bounded Jacobian
			in a neighborhood of $t^*$, then $\{x_k\}$ converges to  $t^*$} in mean-squared sense.
\end{proposition}
\begin{proof} 
{To show part (a), 
%\us{the required result, it suffices to   show that $$\lim_{k\rightarrow \infty}\EXP{\hbox{dist}^2\big(x_{\us{k}},\us{X^*}\big)}=0.$$}
	using the triangle inequality and by recalling that $(a+b)^2
			\leq 2a^2+2b^2$ for any $a,b \in \Real$, we estimate 
$\hbox{dist}^2\big(x_{k+1},X^*\big)$ from above,} as follows:
\begin{align*}
\hbox{dist}^2\big(x_{k+1},X^*\big) &\leq
\left(\hbox{dist}\big(x_{k+1},s_k\big) +\hbox{dist}\big(s_k,\us{X^*}\big)\right)^2 
\leq 2\hbox{dist}^2\big(x_{k+1},s_k\big)
	+2\hbox{dist}^2\big(s_k,\us{X^*}\big).
\end{align*}
Taking expectations in the preceding relation, we obtain that
\begin{align}\label{ineq:conv-mean-triangle}
\EXP{\hbox{dist}^2\big(x_{k+1},\us{X^*}\big)} \leq
2\EXP{\|x_{k+1}-s_{k}\|^2} +2\hbox{dist}^2\big(s_k,\us{X^*}\big).
\end{align}
Note that \us{in the inequality above}, the term $\hbox{dist}^2\big(s_k,\us{X^*}\big)$ is a
deterministic quantity \us{since $s_k$ is a (unique) solution to a
	deterministic problem. By Proposition~\ref{prop:UB-MSE}, there
		exists a finite constant $\theta >0$ such that
		$\EXP{\|x_{k+1}-s_k\|^2}\leq \theta
		{\g_k}\slash({\eta_k\e_k^2})$}. Therefore, from (\ref{ineq:conv-mean-triangle}) we obtain
\begin{align}\label{ineq:conv-mean-triangle2}
\EXP{\hbox{dist}^2\big(x_{k+1},\us{X^*}\big)} \leq2\theta
\frac{\g_k}{\eta_k\e_k^2} +2\hbox{dist}^2\big(s_k,\us{X^*}\big).
\end{align} 
Proposition \ref{prop:sk_estimate}(b1) indicates that the term
$\hbox{dist}^2\big(s_k,\us{X^*}\big)$ goes to zero as $k \to
	\infty$. Since $\lim_{k\rightarrow
		\infty}{\g_k}\slash({\eta_k\epsilon_k^2})=0$, from relation
		\eqref{ineq:conv-mean-triangle2}, we conclude that the term
		$\EXP{\hbox{dist}^2\big(x_{k+1},{X^*}\big)}$ goes to
		zero as $k\to\infty$. 
		
		{The result in part (b) follows similarly to the preceding analysis, wherein we replace
		$\hbox{dist}\big(x_{k+1},X^*\big)$ and $\hbox{dist}^2\big(s_{k},X^*\big)$, respectively, 
		by $\|x_{k+1}-t^*\|^2$ and $\|s_k-t^*\|^2$ with $t^*$ being the smallest norm solution, 
		and by invoking Proposition \ref{prop:sk_estimate}(b2).}
\end{proof}

As a counterpart of Lemma \ref{lemma:stepsize_a.s.}, the following
	result presents a class of the stepsize, regularization, and
		smoothing sequences that ensure  mean-square convergence.
\begin{lemma}\label{lemma:stepsize_exp}
Suppose sequences $\{\g_k\}$, $\{\eta_k\}$, and $\{\e_k\}$ are given by 
$\g_k=\g_0(k+1)^{-a}$, $\eta_k=\eta_0(k+1)^{-b}$, and $\e_k=\e_0(k+1)^{-c}$ where $a$, $b$, and $c$ satisfy the following conditions:
\begin{align*}
a,b,c>0, \quad a+b < 1 \quad a+b\leq \frac{2}{3}(1+c), \quad a > b+2c,
\end{align*}
$\g_0$, $\eta_0$, $\e_0$ are positive scalars and $\e_0 \leq \e$. Then,
	sequences $\{\g_k\}$, $\{\eta_k\}$, and $\{\e_k\}$ satisfy
	Assumption~\ref{assum:conv-exp} and $\lim_{k\rightarrow
		\infty}{\g_k}\slash({\eta_k\e_k^2})=0$. 
		{If in addition $b<c$, then $\lim_{k\to\infty}\frac{\e_k}{\eta_k}=0$.}
\end{lemma}
\begin{proof} 
The proof of this Lemma can be \us{carried out} in a
	similar vein to Lemma \ref{lemma:stepsize_a.s.}. 
	We only show that \an{the conidtion 4}(a) is satisfied. Equivalently, we need to show that there exist $0<\delta<0.5$ and $K_3\geq 0$ such that
\begin{align}\label{ineq:assump-d-equiv}
\mbox{Term 1}\triangleq \left(\frac{\g_{k-1}}{\eta_{k-1}\e_{k-1}^2}\right)\left(\frac{\eta_{k}\e_{k}^2}{\g_{k}}\right) -1\leq \delta \eta_{k}\g_{k}, \qquad \hbox{for any }k > K_3.
\end{align}
 Substituting the sequences $\{\g_k\}$, $\{\eta_k\}$, and $\{\e_k\}$ by their rules we obtain
\[\mbox{Term 1}=\fyRev{\left(\frac{\g_0{k}^{-a}}{\eta_0{k}^{-b}\e_0{k}^{-2c}}\right)\left(\frac{\eta_0(k+1)^{-b}\e_0^2(k+1)^{-2c}}{\g_0(k+1)^{-a}}\right)-1\left(1+\frac{1}{k}\right)^{a-b-2c}-1}.\]
Using the Taylor expansion for $(1+x)^p$ where $x=\frac{1}{k}$ and
$p=a-b-2c$, it can be shown that $\mbox{Term 1}={\cal O}(k^{-1})$.
Suppose $\delta$ is an arbitrary scalar in $(0,0.5)$. Multiplying and dividing by $\delta\g_k\eta_k$, we obtain
\begin{align}\label{ineq:assump-d-equiv-2}
\mbox{Term 1}=\delta\g_k\eta_k\frac{{\cal
	O}(k^{-1})}{\delta\g_k\eta_k}=\delta\g_k\eta_k\frac{{\cal
		O}(k^{-1})}{\delta\eta_0\g_0(k+1)^{-a-b}}=\delta\g_k\eta_k {\cal
			O}(k^{-(1-a-b)}).
\end{align}
Note that $a+b<1$. Therefore, ${\cal O}(k^{-(1-a-b)}) \rightarrow 0$
	when $k \rightarrow 0$. This implies that there exists some
	nonnegative number $K_3$ such that for any $k >K_3$, ${\cal O}(k^{-(1-a-b)})\leq 1$.  
	From (\ref{ineq:assump-d-equiv-2}) we obtain $\mbox{Term 1} \leq \delta\g_k\eta_k$ for any $k >K_3$. 
	Hence, we conclude that relation (\ref{ineq:assump-d-equiv}) holds implying that the \an{condition 4(a)} is satisfied. 
\end{proof}

\begin{remark}
Figure \ref{fig:abc} shows the feasible ranges for parameters $a,b,$ and
$c$ when $\g_k=\g_0(k+1)^{-a}$, $\eta_k=\eta_0(k+1)^{-b}$, and
$\e_k=\e_0(k+1)^{-c}$. Figure \ref{fig:abc}(a) represents the feasible set
of these parameters for which the almost sure convergence is guaranteed,
and Figure \ref{fig:abc}(b) shows the set for the mean-square
convergence. We observe that each set is relatively large. Note
that the two sets are distinct with a nonempty
	intersection. This corresponds well with theory in that almost-sure
convergence and convergence in mean-square are not equivalent. 
\end{remark}

We conclude this section by noting that our rate statement is not
	altogether satisfactory in that we do not relate $x_k$ to $X^*$. To
allow for precisely such a statement, we consider an averaging framework
in the next section.
\begin{figure}[htb]
\centering
\captionsetup{justification=centering}
 \subfloat[Almost sure convergence]{\label{fig:originalf}\includegraphics[scale=.5]{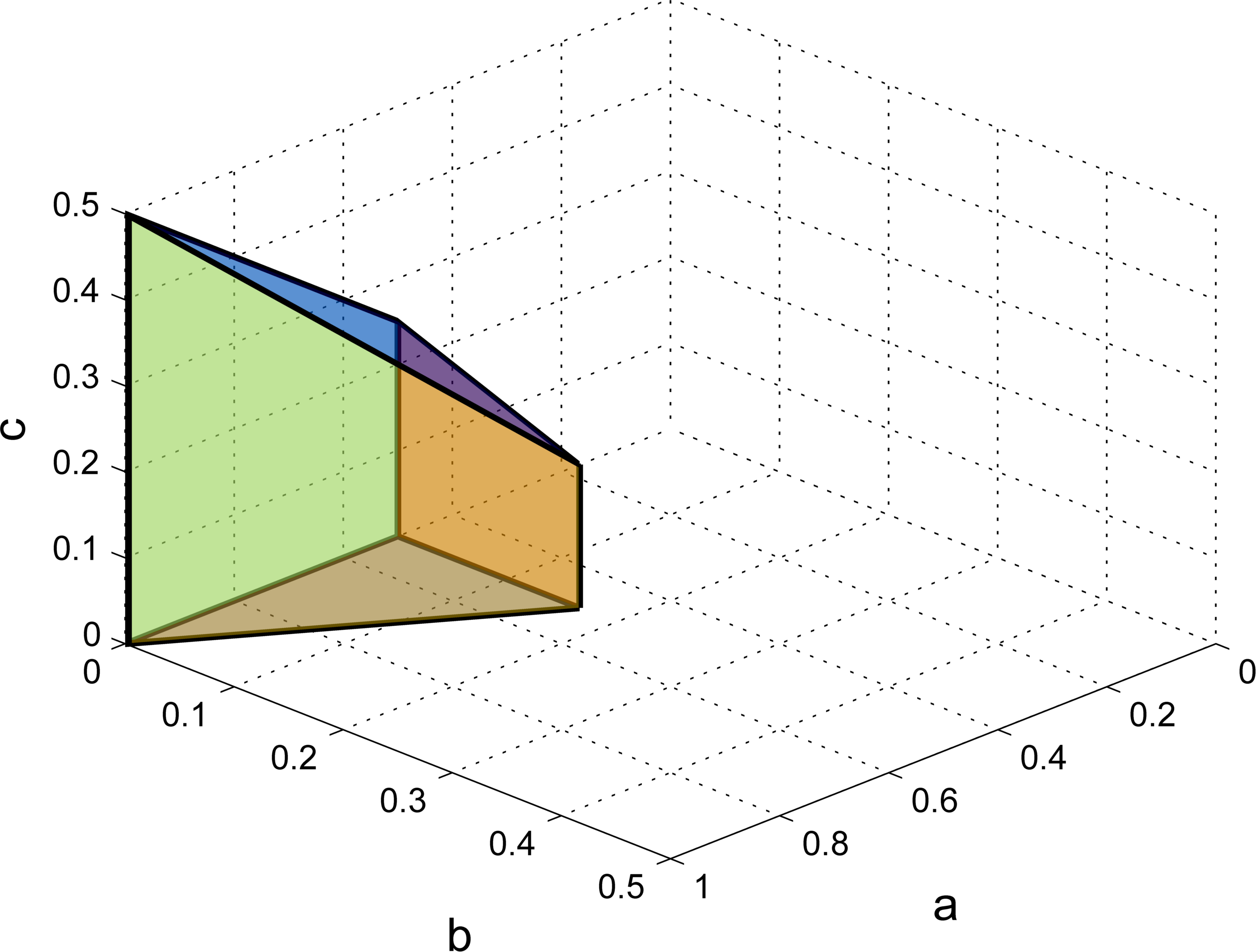}}
 \subfloat[Convergence in mean]{\label{fig:epsChange}\includegraphics[scale=.5]{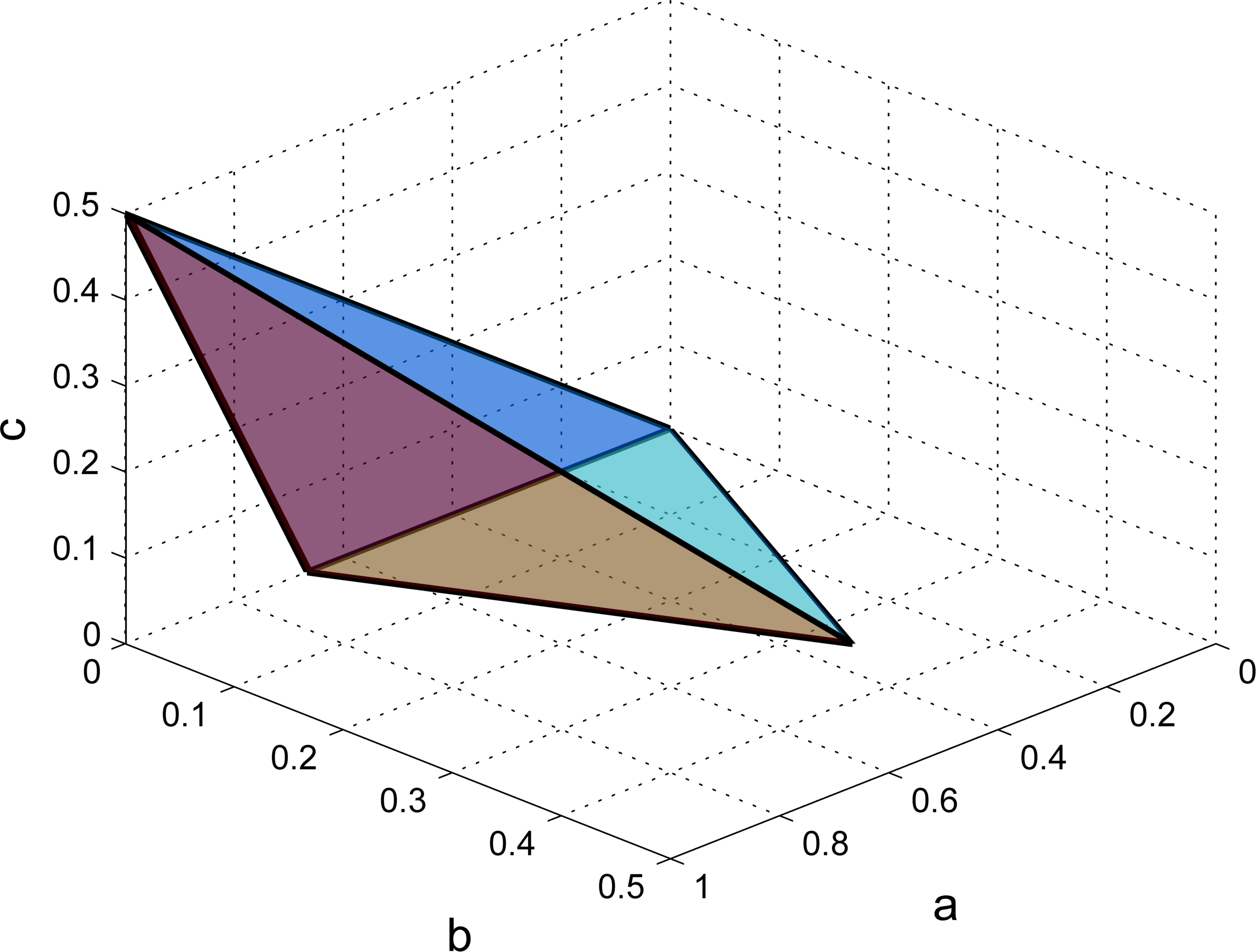}}
\caption{\small Feasible ranges for $a$, $b$, and $c$ when $\g_k=\g_0(k+1)^{-a}$, $\eta_k=\eta_0(k+1)^{-b}$, and 
$\e_k=\e_0(k+1)k^{-c}$. 
%\an{IT MAY BE GOOD TO ADD two more figures with the additional condition $b<c$  - perhaps later
%after the review}
}
\label{fig:abc}
\end{figure}

%---------------------------------------------------
\section{Rate of convergence analysis under weighted averaging}\label{sec:ave}
%---------------------------------------------------
In this part, our interest is in analyzing
the convergence and deriving rate statements for the averaged sequences
associated with the RSSA scheme. It should be emphasized that while the
underlying algorithm does not change in any way, the extracted sequence differs
in that it is a weighted average of the sequence generated by the
original scheme. The \ref{algorithm:IRLSA-averaging} scheme \uvs{generalizes} the
classical stochastic approximation method with averaging in two
directions:\\
\noindent {\em Weighted averaging:} In the {\ref{algorithm:IRLSA-averaging} algorithm}, the iterates
$\bar x_k(r)$ are defined as the weighted average of $x_0, x_1,
\ldots, x_k$ with the corresponding weights ${
\gamma_0^r}\slash({\sum_{t=0}^k \gamma_t^r}), {
	\gamma_1^r}\slash({\sum_{t=0}^k \gamma_t^r}), \dots,{
		\gamma_k^r}\slash({\sum_{t=0}^k \gamma_t^r})$. Note that when
		the stepsizes are decreasing, for $r>0$ these weights are also
		decreasing, while for $r<0$, the weights are increasing. When
		$r=0$, $\bar x_k(r)$ represents the average of $x_0, x_1,
		\ldots, x_k$ with equal weights $\frac{1}{k}$. By allowing $r$
		to be an arbitrary number, we analyze the
		convergence rate of a class of averaging schemes. In fact, we
		will see that the choice of $r$ will affect the rate of
		convergence of a suitably defined gap function. \\
\noindent {\em Regularization and smoothing:} Similar to the first part
of the paper, in the {\ref{algorithm:IRLSA-averaging} scheme}, we {employ} 
regularization and randomized smoothing. Using this generalization, we present
almost sure convergence results for the {\ref{algorithm:IRLSA-averaging} scheme (Proposition~\ref{prop:ave-all-acc})} for {the
sequence $\{\bar x_k(r)\}$} to a solution of problem
(\ref{def:SVI}). \an{Moreover, we derive the convergence rate of the method in terms of 
the gap function values (Proposition~\ref{prop:as-rate})}. Note that, here we allow for the
case that $\{\eta_k\}$ and $\{\e_k\}$ are zero sequences (referred as aSA$_r$). In
that case, the aSA$_r$ algorithm represents the
classic stochastic approximation method utilizing the averaging
technique. 

In Section~\ref{sec:41}, we provide a brief background to gap
functions and derive relevant bounds. We prove the almost sure
convergence of the sequence derived from the aRSSA$_r$ scheme in
Section~\ref{sec:as-gap}. Finally, in Section~\ref{sec:rate-gap}, we show that the
expected gap function diminishes to zero at the optimal rate of ${\cal
	O}(1/\sqrt{k})$ and extend the result to window-based averaging. 

\subsection{An introduction to gap functions}\label{sec:41}
Unlike in optimization settings where the function value provides a
natural metric to measure progress \an{of an algorithm}, no such object naturally arises in the
context of variational inequality problems. \an{Gap functions have
emerged as the analog of the objective function for quantifying the
sub-optimality of a candidate solution} $x$ for the problem VI$(X,F)$. It may
be recalled that a function $g: X \to \Real \cup
\{-\infty,+\infty\}$ is a gap function if satisfies two properties: (i)
it is restricted in sign over $X$; and (ii) $g(x) = 0$ if and only if
$x$ solves VI$(X,F)$. If $g$ is a nonnegative function, then one may
obtain a solution to VI$(X,F)$ by minimizing the gap function over $X$.
A more expansive discussion on gap functions is provided by Larsson and
Patriksson~\cite{larsson94class}. We consider a gap function that has
found significant utility in the solution of monotone variational
inequality problems.
\begin{definition}[{\bf G function}]\label{def:gap1}
Let $X \subseteq \mathbb{R}^n$ be a nonempty and closed set, {and
let the mapping $F: X\rightarrow \mathbb{R}^n$ be} defined on the set $X$. 
{Define} the following function $\mbox{G}: X \rightarrow [0,+\infty)\cup \{+\infty\}$ 
\begin{align}\label{equ:gapf}
\mbox{G}(x) \triangleq \sup_{y \in X} F(y)^T(x-y)\qquad\hbox{for all }x\in X.
\end{align}
\end{definition}
Next, we present some properties of the described function. 
We make use of these relations in the convergence analysis of the scheme~\eqref{algorithm:IRLSA-averaging}.

\begin{definition}[{\bf Weak solution}]\label{def:weak}
Consider VI$(X,F)$ where the set $X \subseteq \mathbb{R}^n$ is nonempty, closed, and convex, and 
the mapping $F: X\rightarrow \mathbb{R}^n$ is defined on the set $X$. 
A vector $x^*_w \in X$ is {said to be} a weak solution to VI$(X,F)$ if we have 
\begin{align}\label{equ:weak}
F(y)^T(y-x^*_w) \geq 0, \qquad \hbox{for all } y \in X. 
\end{align}
We let $X^*_w$ denote the set of weak solutions to VI$(X,F)$.
\end{definition}
\begin{remark}
A weak solution is considered to be a counterpart of the \an{standard} solution of
VI$(X,F)$. \an{A solution of} VI$(X,F)$ is also referred to as a {\it strong solution}.
Note that when the mapping $F$ is monotone, any strong solution of VI$(X,F)$ is also a weak
solution, i.e., $X^* \subseteq X_w^*$. Moreover, when $F$ is continuous,
	it is known that $X^*_w \subseteq X^*$ (cf.~\cite{Nem11}). 
	
	Throughout the paper, since we assume both monotonicity and continuity of the mapping $F$, there is no distinction between a weak and a strong solution.
\end{remark}

\an{Now, we provide some properties of the function $\mbox{G}(x)$ 
in the following lemma, the proof of which is in the Appendix.}
\begin{lemma}[{\bf Properties of \us{$\mbox{G}(x)$}}]\label{lemma:gap-positive}
\an{For the function \fyRev{$\hbox{G}(x)$} given by Definition~\ref{def:gap1}, we have}:
\begin{itemize}
\item [(a)] The function $\mbox{G}(x)$ %given by (\ref{equ:gapf}) 
\us{is
	a gap function, i.e., it satisfies the following:} (i) 
	$\mbox{G}(x)$ is  nonnegative for any $x \in X$; and (ii) {$x \in
		X^*$} if and only if $\mbox{G}(x) = 0$. 
\item [(b)] Assume that the mapping $F$ is bounded over $X$, i.e.,
there exists a constant $C>0$ such that $\|F(x)\|\leq C$ for any $x \in
X$. Then, the following hold: (i)  $\hbox{G}(x)$ is continuous at
any $x \in X$; and (ii) {If $X$ is bounded, i.e., there exists a
constant $M>0$ such that $\|x\|\leq M$ for any $x \in X$, then \fyRev{$\hbox{G}(x)$} is also bounded over $X$: 
$\hbox{G}(x) \leq 2CM$ for all $x\in X$.}
\end{itemize}
\end{lemma}
%-------------------------------------------------------------------------------------------------
\subsection{Convergence analysis for the averaging schemes}\label{sec:as-gap}
%-------------------------------------------------------------------------------------------------
Next, \uvs{we derive an upper bound for the expected gap function value for the
averaged sequence generated by the \ref{algorithm:IRLSA-averaging} scheme. We
begin with a basic relation for the forthcoming development}.

\begin{lemma}\label{lemma:gap-bounds} Consider problem (\ref{def:SVI}) and
let the sequence $\{\bar x_k(r)\}$ be generated by the \ref{algorithm:IRLSA-averaging} algorithm, where $\g_k>0$, $\e_k\geq0$ and
$\eta_k\geq 0$ for any $k \geq 0$, and  $r \in \Real$. Suppose that
{Assumptions~\ref{assum:step_error_sub_1} and~\ref{assum:step_error_sub_2} hold}.
%, and let the stepsize sequence $\{\g_k\}$ be non-increasing. 
Then, for any $k \geq 0$ and $y\in X$ the following relation holds:
\begin{align}\label{ineq:rate_proof_000}
& \g_k^rF(y)^T(x_k-y)
 \leq \frac{1}{2}\g_k^{r-1}\left(\|x_{k}-y\|^2+\|u_{k}-y\|^2\right)
-\frac{1}{2}\g_k^{r-1}\left(\|x_{k+1}-y\|^2+\|u_{k+1}-y\|^2\right)
 \cr &\qquad
 +\g_k^r\left(2\e_kC+\eta_kM^2+w_k^T(u_k-x_k)+\frac{\g_k}{2}\|w_k\|^2 +\g_k\|\Phi(x_k+z_k,\xi_k)\|^2+\g_k\eta_k^2M^2\right).
\end{align}
\end{lemma}
\begin{proof} {For any $y \in X$, the
non-expansivity property of the projection operator implies that}
\begin{align*} 
\|x_{k+1}-y\|^2&=\|\Pi_{X}(x_k-\g_k(F(x_k+z_k)+\eta_kx_k+w_k))-\Pi_{X}(y)\|^2 \cr 
&\leq \|x_k-\g_k(F(x_k+z_k)+\eta_kx_k+w_k)-y\|^2. 
\end{align*}
From the preceding relation, {by noting that $F(x_k+z_k)+w_k=\Phi(x_k+z_k,\xi_k)$}, we obtain 
\begin{align}\label{ineq:proofLemmaEightSecondIneq}
\|x_{k+1}-y\|^2&\leq \|x_{k}-y\|^2-2\g_k(F(x_k+z_k)+\eta_kx_k+w_k)^T(x_k-y) +\g_k^2\|F(x_k+z_k)+\eta_kx_k+w_k\|^2\cr 
&=\|x_{k}-y\|^2-2\g_kF(x_k+z_k)^T(x_k-y)-2\g_k\eta_kx_k^T(x_k-y)-2\g_kw_k^T(x_k-y)\cr &+\g_k^2\|\Phi(x_k+z_k,\xi_k)+\eta_kx_k\|^2\cr
&\leq \|x_{k}-y\|^2-2\g_kF(x_k+z_k)^T((x_k+z_k)-y) +2\g_kF(x_k+z_k)^Tz_k+2\g_k\eta_kx_k^Ty\cr
&-2\g_kw_k^T(x_k-y)+2\g_k^2\|\Phi(x_k+z_k,\xi_k)\|^2+2\g_k^2\eta_k^2\|x_k\|^2,
\end{align}
where in the last inequality, we added and subtracted the term 
$2\g_kF(x_k+z_k)^Tz_k$, {dropped the term $2\g_k\eta_kx_k^Tx_k$, and used \fyRev{$\|a+b\|^2\le 2\|a\|^2 + 2\|b\|^2$}
to estimate the term $\|\Phi(x_k+z_k,\xi_k)+\eta_kx_k\|^2$}. By using the Cauchy-Schwartz inequality,
	invoking Remark \ref{rem:C-M}, and \us{by recalling} $\|z_k\|\leq \e_k$, we obtain
\begin{align}\label{eq:oneo}
\|x_{k+1}-y\|^2
\leq & \|x_{k}-y\|^2-2\g_kF(x_k+z_k)^T((x_k+z_k)-y)+2\g_k\| F(x_k+z_k)\|\|z_k\|\cr 
&+2\g_k\eta_k\|x_k\|\|y\|-2\g_kw_k^T(x_k-y)+2\g_k^2\|\Phi(x_k+z_k,\xi_k)\|^2 +2\g_k^2\eta_k^2M^2\cr
\leq &\|x_{k}-y\|^2-2\g_k(F(x_k+z_k)-F(y))^T((x_k+z_k)-y)-2\g_kF(y)^T((x_k+z_k)-y)\cr
 &+2\g_k\e_kC+2\g_k\eta_kM^2+2\g_kw_k^T(y-x_k)+2\g_k^2\|\Phi(x_k+z_k,\xi_k)\|^2 +2\g_k^2\eta_k^2M^2\cr
 \leq & \|x_{k}-y\|^2-2\g_kF(y)^T(x_k-y) +{2\g_k\e_kC} + 2\g_k\e_kC+2\g_k\eta_kM^2\cr
 &+2\g_kw_k^T(y-x_k)+2\g_k^2\|\Phi(x_k+z_k,\xi_k)\|^2+2\g_k^2\eta_k^2M^2,
\end{align}
where in the second inequality, we add and subtract the term
$2\g_kF(y)^T((x_k+z_k)-y)$, {while in the last inequality we invoke the
Cauchy-Schwartz inequality to obtain $-2\g_kF(y)^Tz_k\le 2\g_k\e_kC$. 
In the last inequality, we also invoke} the monotonicity property of mapping $F$ on $X^\e$,
		 which implies that  the term
			 $-2\g_k(F(x_k+z_k)-F(y))^T((x_k+z_k)-y)$ in the second
			 inequality is nonpositive. 
			 
{We next define an auxiliary sequence} $u_{k+1}$ as 
\begin{align}\label{def:u_t}u_{k+1}=\Pi_X[u_k+\g_kw_k], \quad
	\hbox{for any $k\geq 0$},\end{align} 
where $u_0=x_0$. {By writing \fyRev{$w_k^T(y-x_k)=w_k^T(u_k-x_k)+w_k^T(y-u_k)$}, the  inequality~\eqref{eq:oneo} yields 
for all $y\in X$,}
\begin{align}\label{ineq:rate_proof_00}
\|x_{k+1}-y\|^2
&\leq \|x_{k}-y\|^2-2\g_kF(y)^T(x_k-y)
 +4\g_k\e_kC+2\g_k\eta_kM^2\cr
 &+2\g_kw_k^T(u_k-x_k) +2\g_kw_k^T(y-u_k)+2\g_k^2\|\Phi(x_k+z_k,\xi_k)\|^2+2\g_k^2\eta_k^2M^2.
\end{align}
Next, we estimate the term $2\g_kw_k^T(y-u_k)$ {by using~\eqref{def:u_t} to obtain for all $y\in X$,} 
 \begin{align*}
\|u_{k+1}-y\|^2&=\|\Pi_X[u_k+\g_kw_k]-\us{\Pi_X(y)}\|^2\leq \|u_k+\g_kw_k-y\|^2 \cr 
&= \|u_k-y\|^2+2\g_kw_k^T(u_k-y)+\g_k^2\|w_k\|^2.
\end{align*}
Therefore, we have
 $2\g_kw_k^T(y-u_k)  \leq \|u_k-y\|^2-\|u_{k+1}-y\|^2 +\g_k^2\|w_k\|^2.$
The preceding relation and (\ref{ineq:rate_proof_00}) imply that 
\begin{align*}
\|x_{k+1}-y\|^2
&\leq \|x_{k}-y\|^2-2\g_kF(y)^T(x_k-y)
 +4\g_k\e_kC+2\g_k\eta_kM^2+2\g_kw_k^T(u_k-x_k) \cr
 &+\|u_k-y\|^2-\|u_{k+1}-y\|^2+\g_k^2\|w_k\|^2+2\g_k^2\|\Phi(x_k+z_k,\xi_k)\|^2+2\g_k^2\eta_k^2M^2.
\end{align*}
Rearranging the terms and multiplying both sides of the preceding
inequality by $\g_k^{r-1}\slash 2$ for some constant $r \in \mathbb{R}$,
		   the required result follows for any $k\geq 0$
\begin{align*}
& \g_k^rF(y)^T(x_k-y)
\leq
\frac{1}{2}\g_k^{r-1}\left(\|x_{k}-y\|^2+\|u_{k}-y\|^2\right)
-\frac{1}{2}\g_k^{r-1}\left(\|x_{k+1}-y\|^2+\|u_{k+1}-y\|^2\right)\cr
 &\quad+\g_k^r\left(2\e_kC+\eta_kM^2+w_k^T(u_k-x_k) +\frac{1}{2}\g_k\|w_k\|^2
	+\g_k\|\Phi(x_k+z_k,\xi_k)\|^2+\g_k\eta_k^2M^2\right).
\end{align*}
\end{proof}

{Using Lemma~\ref{lemma:gap-bounds}, we next provide a generic bound for the average sequence with any $r\in\Real$.}
\begin{lemma}[Error bounds for gap function]\label{prop:gap-bounds} 
Consider problem (\ref{def:SVI}) and 
let the sequence $\{\bar x_k(r)\}$ be produced by the {\ref{algorithm:IRLSA-averaging} algorithm}, where $\g_k>0$, $\e_k\geq0$ and $\eta_k\geq 0$ for any $k \geq 0$, and  $r \in \Real$. Suppose that Assumptions~\ref{assum:step_error_sub_1} 
and~\ref{assum:step_error_sub_2} hold, {and assume that}  
the stepsize sequence $\{\g_k\}$ is non-increasing. 
Then, for any $N\geq 1$, %\ref{assum:step_error_sub_1}, 
   \begin{align}\label{ineq:gap-general-bound2}\EXP{ \hbox{G}(\bar
		   x_N(r))} & \leq
   \frac{4M^2(\g_{0}^{r-1}+\bko_r\g_{N-1}^{r-1})}{\sum_{k=0}^{N-1}\g_k^r}
		    + \frac{1}{\sum_{k=0}^{N-1}\g_k^r}\left(\sum_{k=0}^{N-1}\g_k^r(2\e_kC+\eta_kM^2+\frac{3}{2}\g_kC^2+\g_k\eta_k^2M^2)\right),
    \end{align}
 where $\bko_r=0$ when $r\geq 1$ and $\bko_r=1$ when $r<1$.
\end{lemma}
\begin{proof}
We consider the two cases depending on the value of $r$, {namely, $r\ge 1$ and $r<1$.}\\
\noindent \underline{Case of ${r \geq 1}$}: Let us assume that $r$ is an arbitrary
 fixed number such that $r \geq1$, \us{implying that $r-1 \geq 0$.
 Since $\{\g_k\}$ is assumed to be a non-increasing sequence, it
 follows}  that $\g_{k+1}^{r-1} \leq \g_k^{r-1}$. Consequently,
	 {by Lemma~\ref{lemma:gap-bounds} (cf.\ relation~\eqref{ineq:rate_proof_000}) we have for all $k\ge0$ and $y\in X$,}
 \begin{align*}
& \g_k^rF(y)^T(x_k-y)
\leq \frac{1}{2}\g_k^{r-1}\left(\|x_{k}-y\|^2+\|u_{k}-y\|^2\right)
- \frac{1}{2}\g_{k+1}^{r-1}\left(\|x_{k+1}-y\|^2+\|u_{k+1}-y\|^2\right)
 \cr &\qquad+\g_k^r\left(2\e_kC+\eta_kM^2+w_k^T(u_k-x_k) +\frac{1}{2}\g_k\|w_k\|^2
  +\g_k\|\Phi(x_k+z_k,\xi_k)\|^2 +\g_k\eta_k^2M^2\right).
\end{align*}
Summing the preceding inequality from $k=0$ to $N-1$, where $N \geq
	1$ is a fixed number, yields
 \begin{align*}
& \sum_{k=0}^{N-1}\g_k^rF(y)^T(x_k-y)  
\leq  \frac{1}{2}\g_0^{r-1}\left(\|x_{0}-y\|^2+\|u_{0}-y\|^2\right)
- \frac{1}{2} \g_{N}^{r-1}\left(\|x_{N}-y\|^2+\|u_{N}-y\|^2\right)\cr
&\quad +\sum_{k=0}^{N-1}\g_k^r\left(2\e_kC+\eta_kM^2+w_k^T(u_k-x_k)
		 +\frac{1}{2} \g_k\|w_k\|^2  +\g_k\|\Phi(x_k+z_k,\xi_k)\|^2 +\g_k\eta_k^2M^2\right)\cr
 &\leq
 \fyRev{4M^2\g_0^{r-1}}
 +\sum_{k=0}^{N-1}\g_k^r\left(2\e_kC+\eta_kM^2+w_k^T(u_k-x_k)
		 +\frac{1}{2} \g_k\|w_k\|^2  +\g_k\|\Phi(x_k+z_k,\xi_k)\|^2 +\g_k\eta_k^2M^2\right),
\end{align*}
where the second inequality is a consequence of noting that
$\|x_{0}-y\|^2 \leq 4M^2$ and $\|u_{0}-y\|^2 \leq 4M^2$, and the
	non-negativity of the sum $\|x_{N}-y\|^2+\|u_{N}-y\|^2$.  Since by the definition of  $\bar x_N(r)$ we have
	$$\bar x_N(r) =
			\sum_{k=0}^{N-1}\frac{\g_k^r}{\sum_{k=0}^{N-1}\g_k^r}x_k,$$ we obtain for all $y\in X$ and $N\ge1$,
\begin{align*}
 \left(\sum_{k=0}^{N-1}\g_k^r\right)F(y)^T(\bar x_N(r)-y) &
 \leq
 \fyRev{4M^2\g_0^{r-1}}+\sum_{k=0}^{N-1}\g_k^r\left(2\e_kC+\eta_kM^2+w_k^T(u_k-x_k)
		 \right. \cr & \left. \quad +\frac{1}{2}\g_k\|w_k\|^2+\g_k\|\Phi(x_k+z_k,\xi_k)\|^2+\g_k\eta_k^2M^2\right).
\end{align*}
Taking the supremum over the set $X$ with respect to $y$ and invoking
the definition of the gap function (Definition \ref{def:gap1}), we have
the following inequality:
 \begin{align*}
 \left(\sum_{k=0}^{N-1}\g_k^r\right)\hbox{G}(\bar x_N(r)) 
 &\leq  \fyRev{4M^2\g_0^{r-1}}+\sum_{k=0}^{N-1}\g_k^r\left(2\e_kC+\eta_kM^2 +\g_k\eta_k^2M^2\right)\cr 
 &+\sum_{k=0}^{N-1}\g_k^rw_k^T(u_k-x_k)
 +\sum_{k=0}^{N-1}\g_k^{r+1}\left(\frac{1}{2}\|w_k\|^2+\|\Phi(x_k+z_k,\xi_k)\|^2\right).
\end{align*}
By taking expectations on both sides of the preceding inequality, we obtain 
\begin{align}\label{ineq:13-0} 
 \left(\sum_{k=0}^{N-1}\g_k^r\right)\EXP{\hbox{G}(\bar x_N(r))}
 &\leq  \fyRev{4M^2\g_0^{r-1}}+\sum_{k=0}^{N-1}\g_k^r\left(2\e_kC+\eta_kM^2
		 +\g_k\eta_k^2M^2\right)+\underbrace{\sum_{k=0}^{N-1}\g_k^r\EXP{w_k^T(u_k-x_k)}}_{\hbox{Term
 1}}\notag \\
 &+\frac{1}{2}\underbrace{\sum_{k=0}^{N-1}\g_k^{r+1}\EXP{\|w_k\|^2}}_{\hbox{Term
 2}}+\sum_{k=0}^{N-1}\g_k^{r+1}\EXP{\|\Phi(x_k+z_k,\xi_k)\|^2}.
\end{align}
Next, we estimate Terms $1$ and $2$. The 
{\ref{algorithm:IRLSA-averaging} algorithm} and the definition of $u_k$ in
(\ref{def:u_t}) imply that $x_{k}$ and $u_k$ are both
$\sF_{k}$-measurable. Thus, the term $u_k-x_{k}$ is
$\sF_{k}$-measurable. Moreover, the definition of $w_k$ imply that $w_k$
is $\sF_{k+1}$-measurable.  Therefore, for any $k\geq 0$:
\begin{align*}
 \EXP{w_k^T(u_k-x_k)\mid \sF_{k}\cup\{z_k\}} =(u_k-x_k)^T\EXP{w_k  \mid \sF_{k}\cup\{z_k\}}=0,
\end{align*}
where in the last equality we have used Lemma~\ref{lemma:bound-on-errors}. 
Taking expectations in the preceding equation, we obtain
\begin{align}\label{ineq:13}
\EXP{w_k^T(u_k-x_k)}=0, \quad \hbox{for any } k \geq 0.
\end{align}	 
{Furthermore, by Lemma~\ref{lemma:bound-on-errors}, we also have $\EXP{\|w_k\|^2}\le C^2$ for all $k\ge0$.
By Assumption~\ref{assum:step_error_sub_1}(c)  it follows that $\EXP{\|\Phi(x_k+z_k,\xi_k)\|^2}\le C^2$.
Substituting the preceding two upper estimates and \eqref{ineq:13} in the inequality~\eqref{ineq:13-0}, we 
obtain the desired inequality}.\\
\noindent \underline{Case of ${r <1}$}: 
Let us assume that $r$ is an arbitrary
 fixed number such that $r<1$. Adding and subtracting the term
 \fyRev{$0.5\g_{k-1}^{r-1} (\|x_{k}-y\|^2+\|u_{k}-y\|^2)$} from the
 right-hand side of relation (\ref{ineq:rate_proof_000}), we obtain the
 following inequality:
 \begin{align*}
&\g_k^rF(y)^T(x_k-y)
 \leq \fyRev{\frac{1}{2}\g_{k-1}^{r-1}
 \left(\|x_{k}-y\|^2+\|u_{k}-y\|^2\right)}
 -\frac{1}{2}\g_{k}^{r-1}\left(\|x_{k+1}-y\|^2+\|u_{k+1}-y\|^2\right)\cr 
 &\quad+\fyRev{\underbrace{\frac{1}{2}\left(\g_k^{r-1}  -{\g_{k-1}^{r-1}}\right) \left(\|x_{k}-y\|^2+\|u_{k}-y\|^2\right)}_{\small{\hbox{Term }} 3}}\cr 
 &\quad +\g_k^r\left(2\e_kC+\eta_kM^2+w_k^T(u_k-x_k) +\frac{1}{2}\g_k\|w_k\|^2 +\g_k\|\Phi(x_k+z_k,\xi_k)\|^2
+\g_k\eta_k^2M^2\right).\end{align*}
 Since $1-r>0$ and $\{\g_k\}$ is non-increasing, the term
 $\g_{k}^{r-1}-\g_{k-1}^{r-1}$ is nonnegative. \us{Recall that we have}
 \fyRev{$\|x_{k}-y\|^2\leq 2\|x_{k}\|^2+2\|y\|^2\leq 4M^2$ and $\|u_{k}-y\|^2
 \leq 4M^2$}, \us{allowing us to claim that} 
\fyRev{$\hbox{Term }3 \leq 4M^2\left(\g_k^{r-1}  -{\g_{k-1}^{r-1}}\right)$.}
{By using these estimates and, then, taking the summations over the resulting inequality} from $k=1$ to $N-1$ for
a fixed value $N \geq 1$, 
{and dropping the non-positive terms $-0.5\g_{N-1}^{r-1}(\|x_{N}-y\|^2+\|u_{N}-y\|^2)$
and \fyRev{$-4M^2\g_{0}^{r-1}$},
we obtain the following relation for all $y\in X$ and $N\ge1$,}
 \begin{align}\label{eq:ooeo}
 & \sum_{k=1}^{N-1}\g_k^rF(y)^T(x_k-y) 
 \leq \fyRev{\frac{1}{2} \g_{0}^{r-1} \left(\|x_{1}-y\|^2+\|u_{1}-y\|^2\right)} +4M^2\g_{N-1}^{r-1}\cr 
 &\quad +\sum_{k=1}^{N-1}\g_k^r\left(2\e_kC+\eta_kM^2+w_k^T(u_k-x_k) + \frac{1}{2}\g_k\|w_k\|^2
		 +\g_k\|\Phi(x_k+z_k,\xi_k)\|^2 + \g_k\eta_k^2M^2\right).\qquad
\end{align}
Consider now inequality~\eqref{ineq:rate_proof_000} when $k=0$. 
{By adding
	the resulting inequality to relation~\eqref{eq:ooeo}, we obtain for all $y\in X$ and $N\ge 1$,}
  \begin{align*}
 & \sum_{k=0}^{N-1}\g_k^rF(y)^T(x_k-y) 
 \leq \fyRev{\frac{1}{2} \g_{0}^{r-1} \left(\|x_{0}-y\|^2+\|u_{0}-y\|^2\right)} +4M^2\g_{N-1}^{r-1}\cr 
 &\quad +\sum_{k=0}^{N-1}\g_k^r\left(2\e_kC+\eta_kM^2+w_k^T(u_k-x_k) + \frac{1}{2}\g_k\|w_k\|^2
		 +\g_k\|\Phi(x_k+z_k,\xi_k)\|^2 + \g_k\eta_k^2M^2\right).\qquad
\end{align*}
 The remainder of the proof can be carried out  in a similar
	 fashion to that of the preceding case ($r\ge1$).
Combining the results of both cases, we obtain the \an{stated}
	result.
\end{proof}
\fyRev{We make use of the following inequalities in our analysis, with proofs provided in the Appendix.
\begin{lemma}\label{lemma:ineqHarmonic}
For any scalar $\alpha$ and integers $\ell$ and $N$ where $0\leq \ell \leq N-1$, we have:
\begin{itemize}
\item [(a)] $\ln\left( \frac{N+1}{\ell+1}\right) \leq \sum_{k=\ell}^{N-1}\frac{1}{k+1} \leq \frac{1}{\ell+1}+\ln\left( \frac{N}{\ell+1}\right)$.
\item [(b)] $\frac{N^{\alpha+1}-(\ell+1)^{\alpha+1}}{\alpha+1}\leq \sum_{k=\ell}^{N-1}(k+1)^\alpha \leq (\ell+1)^\alpha+\frac{(N+1)^{\alpha+1}-(\ell+1)^{\alpha+1}}{\alpha+1}$ 
for any $\alpha \neq -1$.
\end{itemize}
\end{lemma}
}
We now proceed to show that the expected gap function diminishes to zero
as $k\to\infty$ under suitable assumptions on the various parameter
sequences. We also
		show that for  a specific class of stepsize sequences and in the
		absence of smoothing and regularization, the expected gap
		function converges to zero at the optimal rate.

\begin{lemma}[{\bf Convergence of the expected gap function values}]\label{lemma:gap-conv}
Consider problem (\ref{def:SVI}) and let sequence $\{ \bar x_k(r)\}$ be generated by 
the \ref{algorithm:IRLSA-averaging} algorithm. Suppose that Assumptions~\ref{assum:step_error_sub_1} and \ref{assum:step_error_sub_2} hold. {Also,  assume that 
the sequences $\{\g_k\}$, $\{\eta_k\}$, and $\{\e_k\}$ are} given by $\g_k=\g_0(k+1)^{-a}$, $\eta_k=\eta_0(k+1)^{-b}$, and 
$\e_k=\e_0(k+1)^{-c}$ with $\g_0>0$, $\eta_0 \geq 0$, $\e_0 \geq 0$. 
Then, for any $b,c>0$ and any $a$ and $r$ such that 
\begin{align*}
(a,r) \in \fyRev{\mathcal{S}\triangleq }\{(u,v)\mid 0<uv\leq 1 \hbox{ and }v\geq 1\} \cup  \{(u,v)\mid 0<u <1 \hbox{ and }v < 1\},
\end{align*}
the values \an{$\EXP{ \hbox{G}(\bar x_k(r))}$ converge} to zero as $k$ goes to infinity. 
\end{lemma}
\begin{proof} %Let $b$ and $c$ be positive numbers. 
{We consider the two cases corresponding to the two sets defining the range of $(a,r)$}.\\
(1) \ {Assume that ${(a,r) \in \{(u,v)\mid 0<uv\leq 1 \hbox{
	and }v\geq 1\}}$}. In this case we have $r \geq 1$. Since
	$\g_k={\g_0}{(k+1)^{-a}}$ is a non-increasing sequence, the
	conditions of \fyRev{Lemma \ref{prop:gap-bounds}} hold. We show that when $0<ar \leq 1$, 
	the \an{values $\EXP{ \hbox{G}(\bar x_k(r))}$ converge} to zero. 
	Note that $\eta_k \leq \eta_0$ implies that $\eta_k^2 \leq \eta_0^2$. Let us define \fyRev{$M^*\triangleq 2\max\{\eta_0^2M^2,1.5C^2\}$}. From relation (\ref{ineq:gap-general-bound2}) for $r \geq 1$ we have 
   \begin{align}\label{ineq:gap-conv1}\EXP{ \hbox{G}(\bar x_N(r))} \leq  \frac{1}{\sum_{k=0}^{N-1}\g_k^r}\left( 4\g_{0}^{r-1}M^2+\sum_{k=0}^{N-1}\g_k^r(2\e_kC+\eta_kM^2+M^*\g_k)\right).
    \end{align}
Let us define the following terms:
\begin{align}\label{def:terms}& \us{h_N} \triangleq
\left(\sum_{k=0}^{N-1}(k+1)^{-ar}\right)^{-1}, \qquad \us{\ell_N
	\triangleq } \fyRev{\e_0}\frac{\sum_{k=0}^{N-1}(k+1)^{-(ar+c)}}{\sum_{k=0}^{N-1}(k+1)^{-ar}}, \cr  &
\us{m_N \triangleq } \eta_0\frac{\sum_{k=0}^{N-1}(k+1)^{-(ar+b)}}{\sum_{k=0}^{N-1}(k+1)^{-ar}},
	\qquad \us{p_N} \triangleq \frac{\sum_{k=0}^{N-1}(k+1)^{-(ar+a)}}{\sum_{k=0}^{N-1}(k+1)^{-ar}}.
\end{align}  
Therefore, relation (\ref{ineq:gap-conv1}) implies that 
   \begin{align}\label{ineq:gap-conv2}\us{\EXP{ \hbox{G}(\bar x_N(r))} \leq
   4\us{\g_0^{-1}}M^2\us{h_N} +2C \ell_N +M^2m_N + \g_0M^*p_N.}
   \end{align}     
To show that \an{$\EXP{ \hbox{G}(\bar x_k(r))}$}  goes to zero, it is enough
to prove that \uvs{the terms} $h_N,\ell_N, m_N,$ and $p_N$ approach 
zero as $N\to\infty$. Since we assumed that $0<ar \leq 1 $,  $h_N$ goes to
zero \us{as $N$ tends to $+\infty$}. 

In the following, we show that $\us{\lim_{N\rightarrow
	\infty}\ell_N=0}$. If $\e_0=0$, then $\ell_N=0$ for all $N$. Otherwise, consider the following cases:\\
(i) {\it Case $ar \neq 1$ and $ar+c\neq1$}: \fyRev{From Lemma
\ref{lemma:ineqHarmonic} we obtain the following: }
\fyRev{\begin{align*}
	 0 \leq \ell_N \leq \e_0\frac{1+\frac{(N+1)^{1-(ar+c)}
		-1}{1-(ar+c)}}{\frac{N^{1-(ar)} -1}{1-(ar)}}\triangleq \hbox{Term }1.
	%\implies    \lim_{N\rightarrow \infty}\ell_N = \lim_{N\rightarrow \infty}  N^{-c}=0.
\end{align*} 
Note that $c>0$ and also $ar<1$ since $(a,r) \in \mathcal{S}$ and we assumed $ar\neq 1$. If $ar < ar+c<1$, then we have Term $1=\mathcal{O}(N^{-c})$. If $ar<1<ar+c$, then Term $1=\mathcal{O}(N^{ar-1})$. In both cases,  $\lim_{N\rightarrow \infty}\ell_N=0$.}\\
(ii) {\it Case $ar\neq1$ and $ar+c=1$}: Since $c>0$, we have
$ar<1$. \fyRev{Lemma
\ref{lemma:ineqHarmonic} yields the following:}
\fyRev{\begin{align*}
 0 \leq \ell_N  \leq
	\e_0\frac{1+\ln(N)}{\frac{N^{1-(ar)} -1}{1-(ar)}} 
	\implies    \lim_{N\rightarrow \infty}{\ell_N} =\lim_{N\rightarrow \infty}\frac{\ln(N)}{N^{1-ar}} =
		\lim_{N\rightarrow \infty}\frac{\ln(N)}{N^{c}} =0.
\end{align*}}\\
(iii) {\it Case $ar=1$ and $ar+c\neq1$}: \fyRev{Since $c>0$, we have
$ar+c>1$, and Lemma
\ref{lemma:ineqHarmonic} implies}
\fyRev{\begin{align*}
& 0 \leq \ell_N \leq \e_0
	\frac{1+\frac{(N+1)^{1-(ar+c)} -1}{1-(ar+c)}}{\ln(N+1)}=\e_0
	\frac{1+\frac{(N+1)^{-c} -1}{-c}}{\ln(N+1)} \implies  \lim_{N\rightarrow \infty}\ell_N = \lim_{N\rightarrow \infty}  \frac{1+c-(N+1)^{-c}}{\ln(N+1)}=0.
\end{align*}
}

In conclusion, when $r\geq 1$, when $0<ar \leq 1$, we have
$\lim_{N\rightarrow \infty}\ell_N=0$. \us{A similar limit can be derived 
for $m_N$ and $p_N$}. Therefore, using relation (\ref{ineq:gap-conv2}) we conclude that $\lim_{N\rightarrow \infty}\EXP{ \hbox{G}(\bar x_N(r))}=0$.\\
\noindent (2) \ {Assume now that ${(a,r) \in\{(u,v)\mid 0<u <1 \hbox{ and }v < 1\}}$.} In this case $r<1$. From relation (\ref{ineq:gap-general-bound2}) and the definition of $M^*$ in the first part of this proof, we have 
 \begin{align}\label{ineq:gap-conv4}\EXP{ \hbox{G}(\bar x_N(r))} \leq  \frac{1}{\sum_{k=0}^{N-1}\g_k^r}\left(\frac{4M^2}{\g_{0}^{1-r}}+\frac{ 4M^2}{\g_{N-1}^{1-r}}+\sum_{k=0}^{N-1}\g_k^r(2\e_kC+\eta_kM^2+\g_kM^*)\right).
    \end{align}

 \fyRev{Consider the definitions in (\ref{def:terms}) and the following
	\begin{align}\label{def:terms0} 
v_N & \triangleq \frac{N^{a(1-r)}}{\sum_{k=0}^{N-1}(k+1)^{-ar}}.
\end{align}  }
%In addition, we define $$\mbox{Term 5}=\frac{N^{a(1-r)}}{\sum_{k=0}^{N-1}(k+1)^{-ar}}.$$ 
 Relation (\ref{ineq:gap-conv4}) implies that 
 \fyRev{  \begin{align}\label{ineq:gap-conv6}\EXP{ \hbox{G}(\bar x_N(r))} &\leq
   \frac{4M^2}{\g_0}h_N +2C\ell_N +M^2m_N + \g_0M^*p_N+
   \frac{4M^2}{\g_0}v_N.
   \end{align}
Since in this case $ar \neq 1$, using Lemma \ref{lemma:ineqHarmonic} and the
 definition of ${v_N}$ and that $a<1$, we have
     \begin{align}\label{ineq:rate-term5} & 0 \leq  v_N \leq 
			  \frac{N^{a(1-r)}}{\frac{N^{1-(ar)} -1}{1-(ar)}} \implies  \lim_{N\rightarrow \infty} v_N =
		\lim_{N\rightarrow \infty}  N^{a-1}=0.    \end{align}  
%		(Note that since . So for the integral in the denominator of the
%		 preceding statement we do not have to consider the logarithm
%		 case.)  
		Since $r<1$ and $a>0$, we have $1\leq N^{a(1-r)}$
		implying that $h_N \leq v_N$ for all $N$. Therefore,
		$h_N$ tends to zero as $N \to \infty$. To show that $p_N$
		tends to zero as $N \to \infty$, we consider the following cases:\\
 (i) {\it Case $a(1+r)<1$}: Using Lemma \ref{lemma:ineqHarmonic}, \an{we obtain} 
   \begin{align*}%\label{ineq:rate-term4}
 &  0 \leq p_N \leq
	 \fyRev{\frac{1+\frac{(N+1)^{1-a(1+r)} -1}{1-a(1+r)}}{\frac{N^{1-(ar)}
		 -1}{1-ar}}} \implies   \lim_{N\rightarrow \infty}p_N =\lim_{N\rightarrow \infty}  \frac{(N+1)^{1-a(1+r)}}{N^{1-ar}} =\lim_{N\rightarrow \infty}  N^{-a}=0.\end{align*}\\
(ii) {\it Case $a(1+r)>1$}: 
Since $ar<1$, by Lemma \ref{lemma:ineqHarmonic}  we have
   \begin{align*}%\label{ineq:rate-term4}
 &  0 \leq p_N \leq
	 \fyRev{\frac{1+\frac{(N+1)^{1-a(1+r)} -1}{1-a(1+r)}}{\frac{N^{1-(ar)}
		 -1}{1-ar}}} \implies   \lim_{N\rightarrow \infty}p_N =\lim_{N\rightarrow \infty}  \frac{a(1+r)-(N+1)^{1-a(1+r)}}{N^{1-ar}-1} =\lim_{N\rightarrow \infty}  N^{ar-1}=0.\end{align*}	\\	 
(iii) {\it Case $a(1+r)=1$}: Since $ar<1$, using Lemma \ref{lemma:ineqHarmonic}  \an{we see that}
    \begin{align*}
&	0 \leq p_N \leq  \frac{1+\ln(N)}{\frac{N^{1-ar}
		-1}{1-(ar)}}\implies  \lim_{N\rightarrow \infty}p_N =\lim_{N\rightarrow
		\infty}\frac{\ln(N)}{N^{1-ar}}= 0.
	\end{align*}
In a similar fashion to the preceding analysis, one can show that
 $\lim_{N\rightarrow \infty} \ell_N =\lim_{N\rightarrow \infty} m_N=0$. Therefore, using relation (\ref{ineq:gap-conv4}) we conclude that
	 $\lim_{N\rightarrow \infty}\EXP{ \hbox{G}(\bar x_N(r))}=0$.}
\end{proof}

In the following, we analyze the  convergence of the averaged
	sequence $\bar x_k(r)$ to the solution set of problem
		(\ref{def:SVI}). First, 
		we present conditions under  which a subsequence of the averaged
		sequence converges to the solution set almost surely. 

\begin{proposition}[{\bf Almost sure convergence of subsequences of
	$\mathbf{\bar x_k(r)}$}]\label{prop:ave-one-acc}
Consider problem (\ref{def:SVI}) 
and suppose the conditions of Lemma \ref{lemma:gap-conv} are satisfied. 
Then, \an{we have} almost surely
\begin{align}
\liminf_{k \rightarrow \infty} \mathrm{G}(\bar
		x_k(r))=0 \qquad \mbox{ and } \qquad 
		\liminf_{k \rightarrow \infty} \mathrm{dist}\left(\bar
		x_k(r),X^*\right)=0. 
\end{align}
\end{proposition}
\begin{proof} Since the conditions of Lemma \ref{lemma:gap-conv} hold, we have 
$\lim_{k \rightarrow \infty}\EXP{\hbox{G}(\bar x_k(r))}=0.$ 
Invoking Lemma \ref{lemma:gap-positive}(a) yields $\hbox{G}(\bar x_k(r)) \geq 0$ 
for any $k\geq 1$. Using Fatou's lemma, we conclude that
\an{$$\liminf_{k \rightarrow \infty} \hbox{G}(\bar x_k(r))=0\qquad a.s.$$}

\an{To prove the relation for $\mathrm{dist}\left(\bar
		x_k(r),X^*\right)$, }
we note that every accumulation point of the \uvs{sequences} produced by this scheme lies in $X$
by the definition of the algorithm and by the \uvs{closedness} of $X$. It follows that at every
accumulation  point, the gap function is nonnegative. 
We now proceed by contradiction and assume the result is false.  Consequently,  we have
that 
\begin{align*} 
\liminf_{k \to \infty} \mbox{dist}(\bar x_k(r), X^*) > 0 \mbox{ with
a positive probability.} 
\end{align*}
Consequently, along any sequence produced by the algorithm,  {with
	positive probability, there exists no 
subsequence that converges} to the solution set. In other words, with
positive probability, we have that the gap function tends to a
positive number (since the limit point lies in $X$) along every such
subsequence associated with this sequence, i.e., 
\begin{align*} 
\liminf_{k \to \infty} \fyRev{\hbox{G}(\bar x_k(r))} > 0 \mbox{ with
a positive probability.} 
\end{align*}
But this contradicts the fact  that $\liminf_{k \rightarrow \infty} \hbox{G}(\bar x_k(r))=0$ a.s.\
and, hence, the result follows.
\end{proof} 

In Proposition~\ref{prop:ave-one-acc}, we proved the  convergence  of the
	averaged sequence \us{in a subsequential sense}. However, in the absence
		of regularization and smoothing, there is no guarantee that the
		\us{entire} sequence $\bar x_k(r)$ is convergent. Motivated by this shortcoming, in
		sequel, we present a class of
		stepsize, regularization and smoothing sequences such that
	 the entire averaging sequence is convergent \us{in an almost-sure
		 sense}  to the least norm solution of the problem.
		 Subsequently, 
		 we also provide a rate analysis for the expected gap function when almost sure convergence is attained. 
		 \an{In our analysis, we make use of the following basic result for averaged sequences.} 
		 %The proof is provided in the Appendix.} 

\begin{lemma}[{\bf Theorem 6, pg.\ 75 of~\cite{Knopp51}}]\label{lemma:cesaro} 
Let $\{u_t\}\subset \mathbb{R}^n$ be a convergent sequence of vectors
with the limit point {$\hat u\in\mathbb{R}^n$. Suppose that $\{\alpha_k\}$ is a
sequence of positive numbers where $\sum_{k=0}^\infty \alpha_k=\infty$. 
Consider the average sequence $\{v_k\}$ given by} 
\[v_k\triangleq \frac{\sum_{t=0}^{k-1} \alpha_t u_t}{\sum_{t=0}^{k-1} \alpha_t}\qquad\hbox{for all }k\ge1.\]
Then, we have $\lim_{k \rightarrow \infty} v_k=\hat u.$
\end{lemma}

\begin{remark}\label{rem:cesaro}
When $\{x_k\}$ is a convergent sequence, by Lemma~\ref{lemma:cesaro}, 
the condition $\sum_{t=0}^\infty \g_t^{r}=\infty$
needs to be met so that the averaging sequence $\bar x_k(r)$ converges
to the same limit point. When the stepsize $\g_k$ is of the form
$\frac{\g_0}{(k+1)^a}$, this condition is equivalent to $ar\leq1$. For
example, when $0.5<a<1$, the parameter \an{$r$ has to lie in $(-\infty,2)$}. 
%while $r	\geq 2$ leads to a violation of this requirement. }
\end{remark}

\begin{proposition}[{\bf Almost sure convergence of the sequence $\mathbf{\{\bar x_k(r)\}}$}]\label{prop:ave-all-acc}
Consider problem (\ref{def:SVI}) and let sequence $\{ \bar x_k(r)\}$ be
generated by the {\ref{algorithm:IRLSA-averaging} algorithm}. Suppose that
Assumptions \ref{assum:step_error_sub_1} and
\ref{assum:step_error_sub_2} hold. {Also, assume that} 
sequences $\{\g_k\}$, $\{\eta_k\}$, and $\{\e_k\}$ are given by
$\g_k=\g_0(k+1)^{-a}$, $\eta_k=\eta_0(k+1)^{-b}$, and
$\e_k=\e_0(k+1)^{-c}$ with positive constants
$\g_0,\eta_0,\e_0$. Moreover, assume that mapping $F$ is differentiable
at $t^*$ and its Jacobian is bounded in a neighborhood of $t^*$.
Suppose that $(a,b,c,r)$ are chosen such that the following hold:
\begin{align}\label{cond:abcr} a,b,c>0, \quad a>0.5,
   \quad  a+3b<1, \quad b+2c<a, \quad b<c \quad  \hbox{and} \quad r \leq
   \frac{1}{a}.\end{align}
 Then, {almost surely, $\lim_{k \to \infty}\{\bar x_k(r)\}=t^*$ where $t^*$ is the least norm solution of VI$(X,F)$.} 
\end{proposition}
\begin{proof} Since \us{the} conditions of Lemma
\ref{lemma:stepsize_a.s.} are satisfied, we may invoke Theorem~\ref{prop:almost-sure}b.
This ensures that $\{x_k\}$ tends to $t^*$ in an a.s.
sense. Since we assumed $ar\leq1$, Lemma
	\ref{lemma:cesaro} (see  Remark \ref{rem:cesaro})
	implies that $\lim_{k \to \infty} \bar x_{k}(r)=\lim_{k \to
		\infty}x_k=t^*$ almost surely.
%		
%We note that by our assumptions, Prop. \ref{prop:sk_estimate} b(2) can
%be claimed to hold.  
%In the \us{final} step of the proof, we show that $\bar x_t$ converges
%to $t^*$ almost surely. , \us{Prop.
%	\ref{prop:almost-sure}(a) holds}, i.e. $\lim_{k\to \infty}
%	\|x_k-s_k\|=0$ almost surely. \us{Having proven that $\{s_k\}$ is
%		convergent},  it
%%	follows that $x_k$ is convergent and $\lim_{k \to \infty}x_k=t^*$
%	almost surely. Since we assumed $ar\leq1$, the result of Lemma
%	\ref{lemma:cesaro}  and the discussion of Remark \ref{rem:cesaro}
%	imply that $\lim_{k \to \infty} \bar x_{k}(r)=\lim_{k \to
%		\infty}x_k=t^*$ almost surely.
%
\end{proof}

%--------------------------------------------------------------------------------------------
\subsection{Rate analysis for the gap function}\label{sec:rate-gap}
In this subsection, we analyze the convergence rate of the expected
	gap function. 

\begin{proposition}[{\bf Convergence rate of expected gap function values}]\label{prop:as-rate}
Suppose the conditions of Proposition~\ref{prop:ave-all-acc} are satisfied. 
Then, for any given $0<\delta<\frac{1}{6}$, there exist some
$a,b,c $, and $r$ satisfying (\ref{cond:abcr}) for which the term $\EXP{\hbox{G}(\bar x_k(r))}$
converges to zero with the order ${\cal O}(k^{-(\frac{1}{6}-\delta)})$. 
More precisely, 
let $0<\delta<\frac{1}{6}$ be a given number and choose $\delta'$ such that \fyRev{$0<\delta'<\delta$}. Suppose $a=0.5+3(\delta-\delta')$ and $b=\frac{1}{6}-\delta$, $c=\frac{1}{6}$, and 
$r<\frac{0.5-3(\delta-\delta')}{0.5+3(\delta-\delta')}$. Then,  
 $\EXP{\hbox{G}(\bar x_k(r))}$ converges to zero with the order ${\cal O}(k^{-(\frac{1}{6}-\delta)})$. 
\end{proposition}

\begin{proof}
Note that since \fyRev{$ar<0.5-3(\delta-\delta')$} and
$0<\delta'<\delta<\frac{1}{6}$, we have $0<ar<0.5$ and $0<r<1$.
Therefore, the inequality \eqref{ineq:gap-general-bound2} holds for
$r<1$. Let us define \fyRev{$M^*\triangleq 2\max\{\eta_0^2M^2,1.5C^2\}$} \fyRev{ and consider the definitions given by
	\eqref{def:terms} and \eqref{def:terms0}}.  
  It follows that 
   
    \fyRev{  \begin{align}\label{ineq:gap-conv6-nd}\EXP{ \hbox{G}(\bar x_N(r))} &\leq
   \frac{4M^2}{\g_0}h_N +2C\ell_N +M^2m_N + \g_0M^*p_N+
   \frac{4M^2}{\g_0}v_N.
   \end{align}
Note that by choosing $a=0.5+3(\delta-\delta')$ and $b=\frac{1}{6}-\delta$, $c=\frac{1}{6}$, and 
$r<\frac{0.5-3(\delta-\delta')}{0.5+3(\delta-\delta')}$, all of the values $ar$, $ar+c$, and $ar+b$, and $ar+a$ are smaller than one. Thus, from Lemma
\ref{lemma:ineqHarmonic}, we have
    \begin{align*}
  & h_N \leq  \frac{1}{\frac{N^{1-ar} -1}{1-ar}}  \ \implies  h_N = {\cal O}(N^{ar-1})={\cal O}(N^{-(0.5+3(\delta-\delta'))}), \cr 
 &    \ell_N \leq  \frac{1+\frac{(N+1)^{1-(ar+c)}
			-1}{1-(ar+c)}}{\frac{N^{1-ar} -1}{1-ar}}
		 \ \implies \ell_N = {\cal O}(N^{-c})={\cal O}(N^{-\frac{1}{6}}), \cr 
    & m_N \leq  \frac{1+\frac{(N+1)^{1-(ar+b)}
			-1}{1-(ar+b)}}{\frac{N^{1-ar} -1}{1-(ar)}}  \
		\implies  m_N = {\cal O}(N^{-b})={\cal O}(N^{-\left(\frac{1}{6}-\delta\right)}),\cr 
    & p_N \leq  \frac{1+\frac{(N+1)^{1-(ar+a)}
			-1}{1-(ar+a)}}{\frac{N^{1-ar} -1}{1-ar}}
	\ \implies p_N = {\cal O}(N^{-a})={\cal O}(N^{-(0.5+3(\delta-\delta'))}), \cr 
      &   v_N  \leq  \frac{N^{a(1-r)}}{\frac{N^{1-ar} -1}{1-ar}}  \
			\implies v_N= {\cal O}(N^{-(1-a)})={\cal O}(N^{-(0.5-3(\delta-\delta'))}).
   \end{align*}
Note that $\delta < \frac{1}{6}$ and $\delta'>0$, it follows that $\frac{1}{6}-\delta < 0.5-3(\delta-\delta')$. Therefore, from \eqref{ineq:gap-conv6-nd}, we obtain $\EXP{ \hbox{G}(\bar x_N(r))}={\cal O}(N^{-\left(\frac{1}{6}-\delta\right)})$. }
\end{proof}

\an{In what follows}, we set the regularization
	and smoothing parameters to zero i.e., $\eta_k=\e_k=0$ for all $k
		\geq 0$. \an{In this case,} the \ref{algorithm:IRLSA-averaging}
			algorithm without regularization and smoothing reduces to  the aSA$_r$ algorithm given by:
 \begin{align}\label{algorithm:IRLSA-averaging2}
 \tag{aSA$_r$}\begin{aligned}
x_{k+1}&=\Pi_{X}\left(x_k-\g_k\Phi(x_k,\xi_k)\right), \cr 
 \bar x_{k+1}(r) &\triangleq \frac{\sum_{t=0}^k \gamma_t^r x_t}{\sum_{t=0}^k \gamma_t^r}. \end{aligned}
\end{align}
First, we show that the expected gap
function \an{values along} the averaged sequence generated by the~\ref{algorithm:IRLSA-averaging2} algorithm, 
converges to zero at the optimal
rate of ${\cal O}({1}\slash{\sqrt{k}})$ for $r<1$. 

\begin{proposition}[{\bf Optimal rate of convergence for \fyRev{\ref{algorithm:IRLSA-averaging2}}}]\label{prop:optimal-rate}
Consider problem (\ref{def:SVI}) and let sequence $\{ \bar x_k(r)\}$ be
generated by the {\ref{algorithm:IRLSA-averaging2}} algorithm \us{ and
	suppose that} Assumptions \ref{assum:step_error_sub_1} and
	\ref{assum:step_error_sub_2}(a) \fyRev{and $\g_k=\frac{\g_0}{\sqrt{k+1}}$ with
	 $\g_0>0$. Then, the following results hold:
\begin{itemize}
\item [(a)] If $r=1$, then $\EXP{ \hbox{G}(\bar x_N(r))}$ converges to zero as $N \to \infty$ 
at the rate  ${\cal O}\left(\frac{\ln(N)}{\sqrt{N}}\right)$.
 \item [(b)]  For any arbitrary $r<1$, $\EXP{ \hbox{G}(\bar x_N(r))}$ converges to zero as $N \to \infty$ 
at the rate  ${\cal O}\left(\frac{1}{\sqrt{N}}\right)$.
\end{itemize}}
\end{proposition}
\fyRev{\begin{proof} 
From Lemma \ref{prop:gap-bounds} and assuming $\e_0=\eta_0=0$, we have 
 \begin{align}\label{ineq:gap-ineq-noreg-nosmooth}\EXP{ \hbox{G}(\bar
		   x_N(r))} & \leq
 \frac{1}{\sum_{k=0}^{N-1}\g_k^r}\left(4M^2(\g_{0}^{r-1}+\bko_r\g_{N-1}^{r-1})+1.5C^2\sum_{k=0}^{N-1}\g_k^{r+1}\right).
    \end{align}

 \noindent (a) Replacing $r=1$ and $\g_k=\frac{\g_0}{\sqrt{k+1}}$ in \eqref{ineq:gap-ineq-noreg-nosmooth}, and recalling Lemma \ref{lemma:ineqHarmonic}, we obtain
  \begin{align*}%\label{ineq:gap-ineq-noreg-nosmooth2}
  \EXP{ \hbox{G}(\bar
		   x_N(r))} & \leq
 \frac{4M^2+1.5C^2\g_0^2\sum_{k=0}^{N-1}(k+1)^{-1}}{\g_0\sum_{k=0}^{N-1}(k+1)^{-0.5}} \leq \frac{4M^2\g_0^{-1}+1.5C^2\g_0\left(1+\ln(N)\right)}{2(\sqrt{N}-1)}={\cal O}\left(\frac{\ln(N)}{\sqrt{N}}\right).
    \end{align*}

 \noindent (b)   Replacing $\g_k=\frac{\g_0}{\sqrt{k+1}}$ in \eqref{ineq:gap-ineq-noreg-nosmooth}, using $r<1$, 
 and invoking Lemma \ref{lemma:ineqHarmonic}, we obtain
  \begin{align*}%\label{ineq:gap-ineq-noreg-nosmooth}
  \EXP{ \hbox{G}(\bar
		   x_N(r))} & \leq
 \frac{4M^2\g_{0}^{r-1}\left(1+N^\frac{1-r}{2}\right)+1.5C^2\g_0^{r+1}\sum_{k=0}^{N-1}(k+1)^{-\frac{r+1}{2}}}{\g_0^r\sum_{k=0}^{N-1}(k+1)^{-\frac{r}{2}}} \\
& \leq  \frac{8M^2\g_{0}^{-1}N^\frac{1-r}{2}+1.5C^2\g_0\sum_{k=0}^{N-1}(k+1)^{-\frac{r+1}{2}}}{\sum_{k=0}^{N-1}(k+1)^{-\frac{r}{2}}} \\
& \leq \frac{8M^2\g_{0}^{-1}N^\frac{1-r}{2}+1.5C^2\g_0\left(1+\frac{(N+1)^\frac{1-r}{2}-1}{\frac{1-r}{2}}\right)}{(N^{1-\frac{r}{2}}-1)/(1-\frac{r}{2})}= \mathcal{O}\left(N^{\frac{1-r}{2}-(1-\frac{r}{2})}\right)={\cal O}\left(\frac{1}{\sqrt{N}}\right).
\end{align*}
\end{proof}}
Comparing this result with the \us{more standard averaging scheme
	that uses $r=1$} (cf.~\cite{nemirovski_robust_2009}) supports the
		idea of using $r<1$ for the \uvs{averaging} sequence $\bar x_k(r)$.
		Specially, when $r<0$ and the sequence $\g_k$ is decreasing, the
		weights in the averaging sequence \us{grow implying that more
			recently generated iterates are attributed more weight}. A
			more general form of the \ref{algorithm:IRLSA-averaging2} algorithm is when the average
			sequence is calculated using a window-based formula given by
			Algorithm~\ref{alg:window}.
 The following result is
 derived using \fyRev{Lemma~\ref{prop:gap-bounds}}. 
\begin{corollary}
Consider problem (\ref{def:SVI}) and let the sequence $\{\bar x_k(r)\}$ be generated by the \ref{alg:window} algorithm, where $\g_k>0$ and  $r \in \Real$. Suppose Assumptions \ref{assum:step_error_sub_1} and \ref{assum:step_error_sub_2} hold, 
and {let the stepsize sequence $\{\g_k\}$ be} non-increasing. Then,   
   \begin{align}\label{ineq:gap-general-bound1-window}\EXP{
	   \hbox{G}(\bar x_N^\ell(r))} \leq
	   \frac{1}{\sum_{k=\ell}^{N-1}\g_k^r}\left(
			   4M^2(\g_{\ell}^{r-1}+\g_{N-1}^{r-1}\bko_r)+C^2\sum_{k=\ell}^{N-1}\g_k^{r+1}\right),
    \end{align}
where $0\leq \ell \leq N-1$, $N \geq 1$, $\bko_r=0$ when $r\geq 1$ and
$\bko_r=1$ when $r<1$.
\end{corollary}
\fyRev{
\begin{proof}
The proof can be done in a similar fashion to the proofs of \an{Lemmas}~\ref{lemma:gap-bounds} and~\ref{prop:gap-bounds}. More precisely, in the proof of each of these Lemmas, we put $\e_k=\eta_k=0$ in all the steps. There is one place in the proof of Lemma~\ref{lemma:gap-bounds} that we can make a sharper bound, which is in the relation~\eqref{ineq:proofLemmaEightSecondIneq}. In the second inequality of~\eqref{ineq:proofLemmaEightSecondIneq}, since $\eta_k=\e_k=0$, we have $$\g_k^2\|\Phi(x_k+z_k,\xi_k)+\eta_kx_k\|^2=\g_k^2\|\Phi(x_k,\xi_k)\|^2,$$ implying that we do not have to use the relation $\|a+b\|^2\le 2\|a\|^2 + 2\|b\|^2$. As a consequence, instead of having $2\g_k^2\|\Phi(x_k,\xi_k)\|^2$, we will have $\g_k^2\|\Phi(x_k,\xi_k)\|^2$ in the last inequality of~\eqref{ineq:proofLemmaEightSecondIneq}. As a result, in the final bound, the multiplier of $C^2$ changes from $1.5$ to
		 $1$. 
\end{proof}}

 \begin{remark}
 The above result \us{generalizes the} bound
 in~\cite{nemirovski_robust_2009} in two directions. First, instead of
 assuming $r=1$, we \us{allow for $r$ to be a real number, leading to
	 the addition of the term $\g_{N-1}^{r-1}\bko_r$.} Second, we derive
	 this bound for the gap function of \us{monotone variational
		 inequality problems}, while the bound
		 in~\cite{nemirovski_robust_2009} addresses convex stochastic
		 optimization problems. \us{Our generalization leads to a
			 slightly different bound; specifically, in that $C^2$ in 
				 (\ref{ineq:gap-general-bound1-window}) is replaced
					 by $0.5C^2$ in the optimization setting.}  
 \end{remark}
 
Next, we develop a window-based diminishing stepsize rule and provide an
associated rate result.

\fyRev{\begin{proposition}[Rate of convergence for window-based averaging schemes]\label{prop:window}
Consider problem (\ref{def:SVI}) and let the sequence $\{\bar
x_N^\ell(r)\}$ be generated by the \ref{alg:window} algorithm, where $r \leq 1$ and $\ell=\lceil{\lambda N}\rceil$ for a fixed
{$\lambda \in (0,1)$}
with {$N > \frac{2}{1-\lambda}$}. Suppose that Assumptions \ref{assum:step_error_sub_1} and \ref{assum:step_error_sub_2} hold and the stepsize sequence $\{\g_N\}$ is given by 
\begin{align}\label{def:stepsize-window}\g_N=\frac{2M\sqrt{1+\bko_r}}{C\sqrt{N+1}}, \quad \hbox{for all }N\geq 0.
\end{align}
Then, the following statements hold:
\begin{itemize}
\item [(a)] For $r = 1$, the optimal convergence rate is attained, i.e., $\EXP{ \hbox{G}(\bar x_N^\ell(r))} =\mathcal{O}\left(\frac{1}{\sqrt{N}}\right)$. 
\item [(b)] For any $r<1$, the optimal convergence rate is attained, i.e., $\EXP{ \hbox{G}(\bar x_N^\ell(r))} =\mathcal{O}\left(\frac{1}{\sqrt{N}}\right)$.
\item[(c)] Suppose the theoretical upper bound (TUB) of the expected gap function in \eqref{ineq:gap-general-bound1-window} is defined as follows:
\[TUB(r,\lambda,N)\triangleq \frac{1}{\sum_{k=\ell}^{N-1}\g_k^r}\left(
			   4M^2(\g_{\ell}^{r-1}+\g_{N-1}^{r-1}\bko_r)+C^2\sum_{k=\ell}^{N-1}\g_k^{r+1}\right), \]
where $\g_k$ is given by \eqref{def:stepsize-window}. Then, for any arbitrary $r<1$, there exists $\bar \lambda \in (0,1)$ and $\bar N \geq 1$ such that 
\[TUB(r,\lambda,N) \leq TUB(1,\lambda,N), \quad \hbox{for any } \lambda <\bar \lambda \hbox{ and }N > \bar N . \]
\end{itemize}
\end{proposition}
 \begin{proof} 
Note that $\g_0=\frac{2M\sqrt{1+\bko_r}}{C}$ implying  $\g_N=\frac{\g_0}{\sqrt{N+1}}$. \\
%{We consider two cases: $r=1$ and $r<1$.}
%(i) The case of $r=1$: 
\noindent
(a) \an{When $r = 1$,}
from (\ref{ineq:gap-general-bound1-window}) and Lemma \ref{lemma:ineqHarmonic} we obtain 
 \begin{align*}
\EXP{ \hbox{G}(\bar x_N^\ell(1))} &\leq
\frac{4M^2+\g_0^2C^2\sum_{k=\ell}^{N-1}\frac{1}{k+1}}{\g_0\sum_{k=\ell}^{N-1}\frac{1}{\sqrt{k+1}}}= 2MC \ \frac{1+\sum_{k=\ell}^{N-1}\frac{1}{k+1}}{\sum_{k=\ell}^{N-1}\frac{1}{\sqrt{k+1}}} \\ &\leq 2MC \ \frac{1+(\ell+1)^{-1}+\ln\left( \frac{N}{\ell+1}\right)}{2(\sqrt{N}-\sqrt{\ell+1})} \leq MC \ \frac{1+1+\ln\left( \frac{1}{\lambda}\right)}{\sqrt{N}\left(1-\sqrt{\frac{\lambda N+2}{N}}\right)} = {\cal O}\left(\frac{1}{\sqrt{N}}\right),
\end{align*}
where we use the relation $\lambda N \leq \ell < \lambda N +1$ in the last inequality.\\
(b) \an{When $r<1$, from~(\ref{ineq:gap-general-bound1-window})} and Lemma \ref{lemma:ineqHarmonic} and invoking $\g_0=\frac{2\sqrt{2}M}{C}$, we obtain 
 \begin{align*}
&\EXP{ \hbox{G}(\bar x_N^\ell(r))}  \leq TUB(r,\lambda, N) =   \dfrac{{4M^2\g_0^{r-1}\left((\ell+1)^{\frac{1-r}{2}}+N^{\frac{1-r}{2}}\right)}+C^2\g_0^{r+1}\sum_{k=\ell}^{N-1}(k+1)^{-\frac{r+1}{2}}}{\g_0^{r}\sum_{k=\ell}^{N-1}(k+1)^{-\frac{r}{2}}}\\
& \leq  \dfrac{2\sqrt{2}MCN^{\frac{1-r}{2}}+2\sqrt{2}MC\left((\ell+1)^{-\frac{r+1}{2}}+\dfrac{(N+1)^\frac{1-r}{2}-(\ell+1)^\frac{1-r}{2}}{\frac{1-r}{2}}\right)}{\dfrac{N^{1-\frac{r}{2}}-(\ell+1)^{1-\frac{r}{2}}}{1-\frac{r}{2}}},
\end{align*}
where in the last inequality, we replaced $\g_0$ by its value and used $(\ell+1)^\frac{1-r}{2} \leq N^\frac{1-r}{2}$. We have
 \begin{align}\label{ineq:tub1}
&\EXP{ \hbox{G}(\bar x_N^\ell(r))}\leq TUB(r,\lambda, N) \notag\\
& \leq 2\sqrt{2}MC\left(1-\frac{r}{2}\right) \ \frac{N^{\frac{1-r}{2}}+N^{\frac{1-r}{2}}\left((\ell+1)^{-1}(\frac{\ell+1}{N})^{\frac{1-r}{2}}+\left((1+\frac{1}{N})^\frac{1-r}{2}-(\frac{\ell+1}{N})^\frac{1-r}{2}\right)\frac{2}{1-r}\right)}{N^{1-\frac{r}{2}}\left(1-\left(\frac{\ell+1}{N}\right)^{1-\frac{r}{2}}\right)} \notag\\
& \leq \frac{2\sqrt{2}MC\left(1-\frac{r}{2}\right)}{\sqrt{N}}\ \frac{1+\left(1\times 1+\left((1+\frac{1}{N})^\frac{1-r}{2}-\lambda^\frac{1-r}{2}\right)\frac{2}{1-r}\right)}{1-\left(\frac{\lambda N +2}{N}\right)^{1-\frac{r}{2}}} \notag\\
& \leq \frac{2\sqrt{2}MC\left(1-\frac{r}{2}\right)}{\sqrt{N}}\ \frac{1+\left(1+2^\frac{1-r}{2}\frac{2}{1-r}\right)}{1-\left(\lambda +\frac{2}{N}\right)^{1-\frac{r}{2}}}
\leq \frac{4\sqrt{2}MC\left(1-\frac{r}{2}\right)}{\sqrt{N}}\ \frac{1+\frac{\left(\sqrt{2}\right)^{1-r}}{1-r}}{1-\sqrt{\lambda +\frac{2}{N}}}.
\end{align}
This implies that  \an{$\EXP{ \hbox{G}(\bar x_N^\ell(r))} =\mathcal{O}\left(\frac{1}{\sqrt{N}}\right)$ for any arbitrary $r<1$}.\\
(c)  
\an{To prove this part,} we derive a lower bound for $TUB(1,\lambda,N)$ and compare it with the upper bound we obtained for the case $r<1$ given by \eqref{ineq:tub1}. From (\ref{ineq:gap-general-bound1-window}) and 
Lemma~\ref{lemma:ineqHarmonic}, we have
 \begin{align}\label{ineq:tub2}
TUB(1,\lambda,N) &=
\frac{4M^2+\g_0^2C^2\sum_{k=\ell}^{N-1}\frac{1}{k+1}}{\g_0\sum_{k=\ell}^{N-1}\frac{1}{\sqrt{k+1}}}= 2MC \ \frac{1+\sum_{k=\ell}^{N-1}\frac{1}{k+1}}{\sum_{k=\ell}^{N-1}\frac{1}{\sqrt{k+1}}}  \notag
\\ &\geq 2MC \ \frac{\ln\left( \frac{N+1}{\ell+1}\right)}{(\ell+1)^{-0.5}+2(\sqrt{N+1}-\sqrt{\ell+1})} \notag \\
 &\geq 2MC \ \frac{\ln\left( \frac{1}{\lambda}\right)}{1+2\sqrt{N+1}} \geq   \frac{2MC\ln\left( \frac{1}{\lambda}\right)}{3\sqrt{2N}}.
\end{align}
  From \eqref{ineq:tub1} and \eqref{ineq:tub2}, for any $r<1$ we have \begin{align*}%\label{ineq:tub3}
  \frac{TUB(r,\lambda,N)}{TUB(1,\lambda,N)}&\uvs{\leq} \left(\frac{4\sqrt{2}MC\left(1-\frac{r}{2}\right)}{\sqrt{N}}\ \frac{1+\frac{\left(\sqrt{2}\right)^{1-r}}{1-r}}{1-\sqrt{\lambda +\frac{2}{N}}} \right) / \left( \frac{2MC\ln\left( \frac{1}{\lambda}\right)}{3\sqrt{2N}}\right)\cr &= \frac{12\left(1-\frac{r}{2}\right)\left(1+\frac{\left(\sqrt{2}\right)^{1-r}}{1-r}\right)}{\left(1-\sqrt{\lambda +\frac{2}{N}}\right)\ln\left(\frac{1}{\lambda}\right)} \triangleq \frac{h(r)}{\left(1-\sqrt{\lambda +\frac{2}{N}}\right)\ln\left(\frac{1}{\lambda}\right)} .
  \end{align*}
  Taking the limit from the preceding relation when $\lambda$ goes to zero, it yields that the term 
  $ \frac{TUB(r,\lambda,N)}{TUB(1,\lambda,N)}$ converges to zero. 
  This implies that there exists $\bar \lambda \in (0,1)$ such that for any $\lambda < \bar \lambda$, 
  we have $TUB(r,\lambda,N) < TUB(1,\lambda,N)$. Note that to have $1-\sqrt{\lambda +\frac{2}{N}}>0$ for any $\lambda < \bar \lambda$, it suffices to assume $N > \bar N \triangleq \frac{2}{1-\bar \lambda}$. Hence, the desired result holds.
 \end{proof}
\begin{remark}[{\bf Comparison of the classic window-based averaging with the proposed averaging scheme}] 
The contribution of Proposition \ref{prop:window} is three-fold:
\begin{enumerate}
\item[(i)] First, it
addresses stochastic variational inequalities, a generalization of stochastic convex optimization problems addressed 
\an{in~\cite{nemirovski_robust_2009,Nedic14}.} 
\item[(ii)] Second, it extends the theory of classical
averaging SA schemes (cf.~\cite{nemirovski_robust_2009}) to the weighted averaging regime, 
\an{and considers more general weighted averaging scheme than that of~\cite{Nedic14} for convex optimization (corresponding to $r=-1$).} In particular, 
\an{it shows} that the
optimal rate of convergence is attained not only \an{for} $r=1$, but also
\an{for} an arbitrary parameter $r < 1$.
\item[(iii)] Third, in terms of the finite-time behavior
of the algorithm, for \an{any $r<1$}, we \uvs{may} always choose
the window parameter $\lambda$ such that the constant factor of the error bound
is smaller than that of the \uvs{classical} window-based averaging scheme. This implies
that we should expect a smaller gap value \an{for $r<1$}, when the
value of $\lambda$ is selected small enough. In fact, this will be tested in
our numerical experiments in the next section where we present our simulation
results for a stochastic Nash-Cournot game. 
\end{enumerate} 
\uvs{Finally, we observe that an extreme case arises by choosing
	$\lambda = e^{-2h(r)}$. Through this choice, we obtain a crude bound
		on $\mbox{TUB}(r,\lambda,N)/\mbox{TUB}(1,\lambda,N)$ of the
		following nature:
$$\frac{\mbox{TUB}(r,\lambda,N)}{\mbox{TUB}(1,\lambda,N)} \leq \frac{1}{2 (1-\sqrt{e^{-h(r)} + 2/N})} \approx {\frac{1}{2}},$$
where $h(r) >> 1$ and $N >> 1$. In effect, when the window size is
chosen in accordance with the choice of $r$, then the upper bound is
improved by approximately 50\% over the case when $r = 1$. We observe
that in this case, the window size tends to $N$ and we recover the full
averaging SA schemes.} 
\end{remark}}
 
%%%%%%%%%%%%%%%%%%%%%%%%%%
\section{Numerical Results}\label{sec:num}
%%%%%%%%%%%%%%%%%%%%%%%%%%
\us{In this section, we compare the performance of our schemes through a
set of computational experiments conducted on a stochastic Nash-Cournot game.  In Section \ref{ex:st-nc}, we  introduce the stochastic Nash-Cournot game
and derive the (sufficient) equilibrium conditions, which are compactly
stated as as a stochastic variational inequality.} In Section
\ref{sec:numerics-RSSA}, the simulation results for  the
\ref{algorithm:IRLSA-impl} scheme are presented  and support the
asymptotic a.s. and mean-square convergence results from Theorem
\ref{prop:almost-sure} and Proposition~\ref{prop:expectation}, respectively.
Next, in Section \ref{sec:numerics-ASA}, we provide the simulation
results the \ref{algorithm:IRLSA-averaging} scheme across different
values of parameter $r$. Throughout this section, we use the gap
function's evaluation as the metric for our comparisons. To calculate
the gap value, we use the commercial solver {\sc Knitro}~\cite{knitro}. 

\subsection{A networked stochastic Nash-Cournot game}\label{ex:st-nc}
A classical example of a Nash game is a networked Nash-Cournot
game~\cite{metzler03nash-cournot,kannan10online}.
In this problem, there are $\cal{I}$ firms that compete over a network
of $\cal{J}$ nodes in selling a  product. Each firm $i$ wants to
maximize profit by choosing nodal production at every node $j$, denoted
by $g_{ij}$, and the level of sales at node $j$, denoted by $s_{ij}$.
Let $\bar s_j = \sum_{i=1}^{\an{\cal I}} s_{ij}$ denote the aggregate sales at node
$j$. \us{By the Cournot structure}, we assume that the price at node $j$, denoted by  $P_j(\bar s_j,\xi)$, is a nonlinear stochastic function of the form $a_j-b_j\bar s_j^\sigma$, where $a_j$ is a uniform random variable drawn from $[lb_{a_j},ub_{a_j}]$, and $b_j$  and $\sigma \geq 1$
are constants. Furthermore, we assume the firm $i$'s cost of production at node
$j$ is denoted by the $C_{ij}(g_{ij})\triangleq c_{j}g_{ij} +d_{j}$, where $c_j$ and $d_j$ are constants. Other than the nonnegativity  constraints for $s_{ij}$ and $g_{ij}$, there are two types of constraints. Firm $i$'s production at node $j$ is capacitated
by $\mathrm{cap}_{ij}$. Also, the aggregated level of sales of each firm is equal to the aggregated level of production. Therefore, firm $i$'s optimization problem is given by the
following (Note that we assume transportation costs are zero): 
\begin{align*}
\displaystyle \min_{x_i \in X_i} & \qquad {\EXP{f_i(x,\xi)}},
\end{align*}
where {$x=(x_1;\ldots; x_{{\cal I}})$} with $x_i=(g_i; s_i)$, {$g_i=(g_{i1};\ldots;g_{i{\cal J}})$, 
$s_i=(s_{i1};\ldots;s_{i{\cal J}})$}, 
\begin{align*}
f_i(x,\xi) & \triangleq \sum_{j =1}^{\cal J} \left( C_{ij}
(g_{ij}) - P_j(\bar s_j,\xi)
s_{ij}\right), %x_i = (g_i, s_i), 
\\ \mbox{ and } X_i & \triangleq \left\{
(g_i,s_i) \mid \sum_{j=1}^{\cal J}
g_{ij} = \sum_{j=1}^{\cal J} s_{ij}, \quad g_{ij},\,s_{ij} \geq 0,\quad g_{ij} \leq \mathrm{cap}_{ij}, \hbox{ for all }\ j = 1,
\hdots, {\cal J}\right\}. %\hfill \qed
\end{align*}
Applying the interchange between the expectation and the
derivative operator, the resulting equilibrium conditions of the preceding stochastic Nash-Cournot
game can be compactly captured by the stochastic variational inequality VI$(X,F)$
where 
$ X \triangleq \prod_{i=1}^{\cal I} X_i$ and {$F(x) = (F_1(x);\ldots; F_{\cal I}(x))$}
with $F_i(x)=\EXP{\nabla_{x_i} f_i(x,\xi)}$. Note that it can be shown that when $1< \sigma \leq 3$ and ${\cal I} \leq \frac{3\sigma-1}{\sigma -1}$, or $\sigma=1$, the mapping $F$ is strictly monotone. We consider a Cournot competition with $5$ firms and $4$ nodes, i.e., ${\cal I}=5$ and ${\cal J}=4$. We assume $\sigma=1$, $[lb_{a_j},ub_{a_j}]=[49.5,50.5]$, $cap_{ij}=300$, $b_j=0.05$, $c_{j}=1.5$ for all $i$ and $j$. Throughout this section, we assume the mean and the standard deviation of the gap function is calculated using a sample of size $50$. Also, we assume the starting point of algorithms is the origin, unless stated otherwise. 
% is given by \eqref{def-F}. 

%\subsection{Notations}\label{sec:num-notation}
Throughout this section we use the following notation: $N$
denotes the simulation length in the scheme, $x_0$ denotes the
starting point of the algorithm. Furthermore, the gap function is given
by Defintion~\ref{def:gap1}. We examine both the RSSA scheme, its averaged variant
given by aRSSA$_r$ for different values of $r$ as well as the
window-based variant denoted by aRSSA$_{\ell,r}$. In the
aRSSA$_{\ell,r}$ scheme, $\ell$ is assumed to be equal to $\lceil \lambda
N\rceil$ where $0<\lambda<1$ is a constant. 
%$\cal{I}$ denotes the numberof firms and $\cal{J}$ denotes the number of nodes. 

\subsection{Convergence of the RSSA scheme}\label{sec:numerics-RSSA}
In this section, we present the simulation results for the RSSA
scheme and  report the performance of the algorithm using the sample mean
and sample standard deviation of the gap function. Table \ref{tab:RSSA}
shows the results for $4000$ iterations. 
\an{For the stepsize $\g_k$, regularization parameter $\eta_k$, and the smoothing parameter
$\e_k$, we use $\g_k={\g_0}(k+0.1N)^{-a}$,
	$\eta_k={\eta_0}{(k+1)^{-b}}$, and $\e_k={\e_0}{(k+1)^{-c}}$
	where $\g_0=1$, $\eta_0=10^{-4}$ and $\e_0=10^{-2}$.} 
Note that the
	term $0.1N$ is added in the stepsize to stabilize the performance of the SA scheme.
	Furthermore, \us{we chose $\eta_0$ and $\e_0$ to be smaller when $b$
		and $c$ are small}, respectively. In
	the first $9$ settings, our goal is to study the sensitivity of the
	RSSA algorithm with respect to the parameters \an{$a$, $b$, $c$, where we use $S(p)$ to denote a particular setting of
these parameters.} In
	these settings, the values of the parameters given in the table
	satisfy conditions of both Lemma \ref{lemma:stepsize_a.s.} and Lemma
	\ref{lemma:stepsize_exp}. In the first three settings, we increase
	$a$ and keep $b$ and $c$ unchanged. In the second group, $b$ is
	increasing, while in the third group $c$ is increasing. We observe
	that increasing $a$ slows down the convergence of the gap function,
	but increasing $b$ or $c$ speeds up the convergence of the gap
	function slightly. This makes sense because the optimal rate of
	convergence is attained at $a=0.5$. On the other hand, by making $b$
	or $c$, larger the regularization and smoothing sequences decay to
	zero faster implying that the perturbations  introduced in the SA
	algorithm due to regularization \us{and} smoothing techniques are fading
	out. We also observe that  the average value of the gap function is
	more sensitive to the change in the parameter $a$ while being more
	robust to the changes in $b$ or $c$.
\begin{table}[htb] 
\tiny
\centering 
\begin{tabular}{|c|c|c c c|c c|} 
\hline 
- & -& \multicolumn{3}{c|}{Parameters} &  \multicolumn{2}{c|}{Gap function} 
\\ 
\hline
\an{S(p)}&- & $a$ & $b$ & $c$ & mean & std \\
\hline
 1&  &$0.501$& $0.099$& $0.200$&  $6.12$e$-3$ & $2.68$e$-3$
  \\ 
2& a &$0.600$& $0.099$ & $0.200$ &  $1.03$e$-1$ & $9.48$e$-3$
  \\ 
  
3&  &$0.700$& $0.099$ &$0.200$&  $6.90$e$+1$ & $1.74$e$-1$
  \\ 
  \hline  
4&  &$0.501$& $0.100$& $0.167$&  $6.37$e$-3$ & $3.56$e$-3$ 
  \\ 
  
5& b&$ 0.501$& $0.130$ &$0.167$& $4.89$e$-3$ & $2.68$e$-3$  
  \\

  6&  &$0.501$& $0.166$ & $0.167$ & $4.33$e$-3$ & $2.11$e$-3$ 
  \\ \hline
7&  &$0.501$& $0.166$ & $0.130$ & $5.02$e$-3$ & $2.65$e$-3$  
  \\ 
8& c &$0.501$& $0.166$ & $0.100$ & $5.05$e$-3$ & $2.42$e$-3$ 
  \\ 
9&  &$0.501$& $0.166$ & $0.167$ & $4.33$e$-3$ & $2.11$e$-3$ 
  \\ 
\hline

10& - &$0.401$& $0.200$& $0.100$& $7.78$e$-3$ & $5.04$e$-3$
  \\ 
  \hline
11&- &$0.600$& $0.133$ & $0.099$&  $9.42$e$-2$ & $7.56$e$-3$
  \\ 
\hline
\end{tabular} 
\caption{RSSA algorithm with different settings of parameters $a$, $b$, and $c$.} 
\label{tab:RSSA} 
\end{table}
 In setting $S(10)$, the parameters ensure convergence in the
 mean-squared sense provided by Lemma~\ref{lemma:stepsize_exp} but do
 not suffice in ensuring almost sure convergence provided in Lemma
 \ref{lemma:stepsize_a.s.}. In the setting $S(11)$, the converse holds.
 Figure \ref{fig:GAP3} illustrates the sample mean of gap function
 \us{over the set of simulations} for settings $S(10)$ and $S(11)$. Both
 plots show the sample mean for the gap function. The round dots in
 these plots represent the  observed gap values for each of the 50
 sample paths at every $100$ iterations.
 \begin{figure}[htb]
 \centering
  \subfloat[Setting $S(10)$ -- $(a,b,c)=(0.401,0.200,0.100)$]{\label{fig:mean}\includegraphics[scale=.50, angle=0]{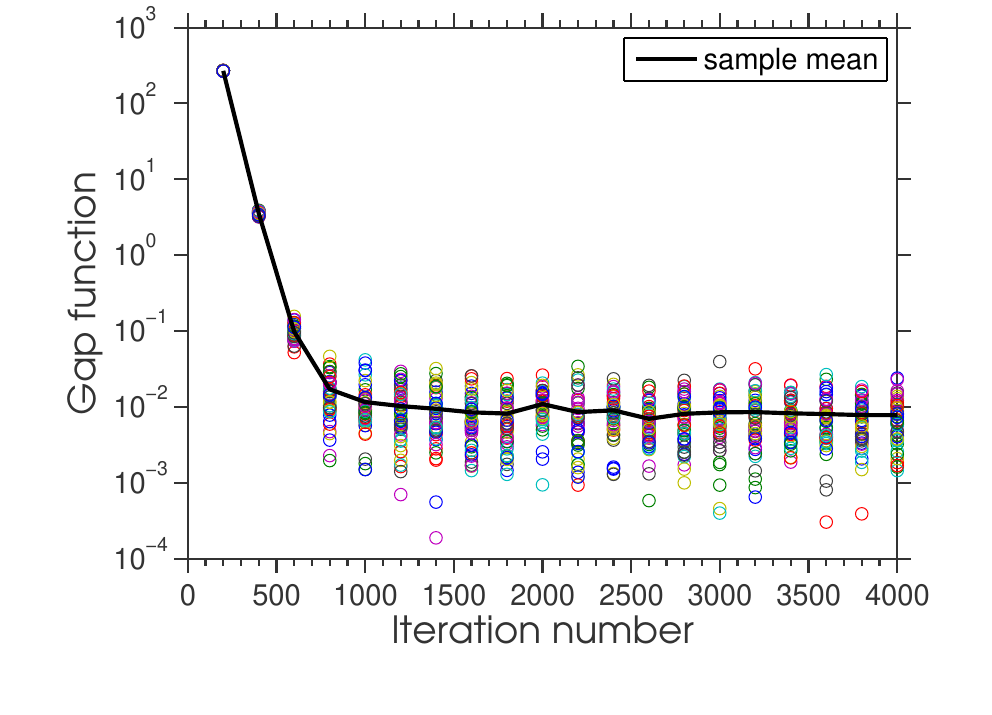}}
 \subfloat[Setting $S(11)$ -- $(a,b,c)=(0.600,0.133,0.099)$]{\label{fig:as}\includegraphics[scale=.50,  angle=0]
 {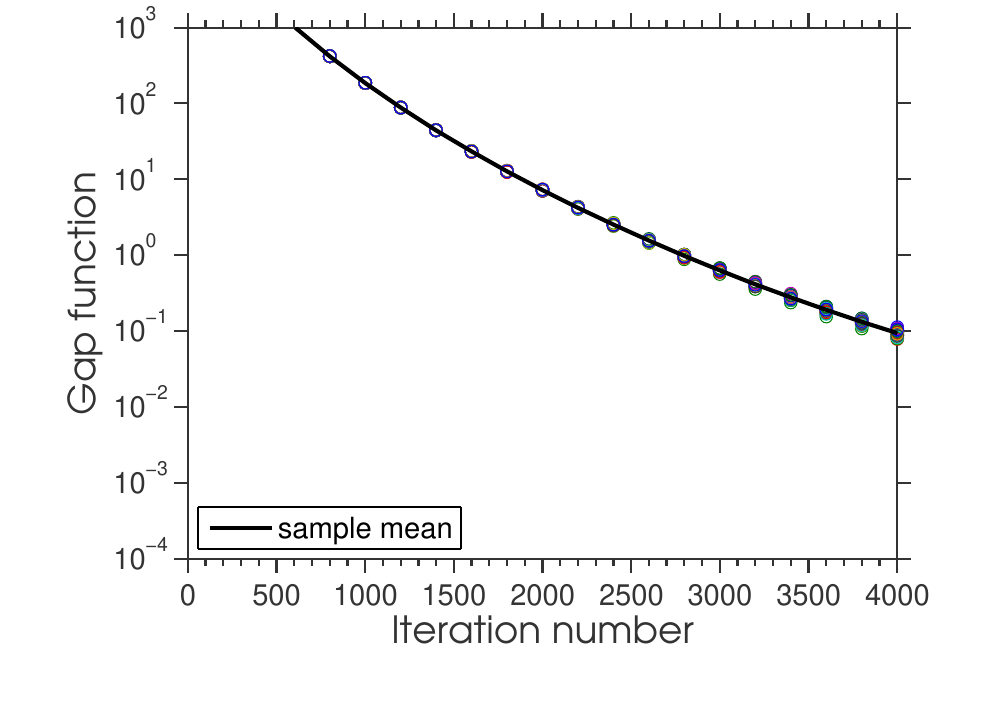}}

\caption{Convergence of RSSA algorithm with different settings of parameters $a$, $b$, $c$.}
\label{fig:GAP3}
\end{figure}

\begin{figure}[htb]
  \centering
  \includegraphics[scale=.40, angle=0]{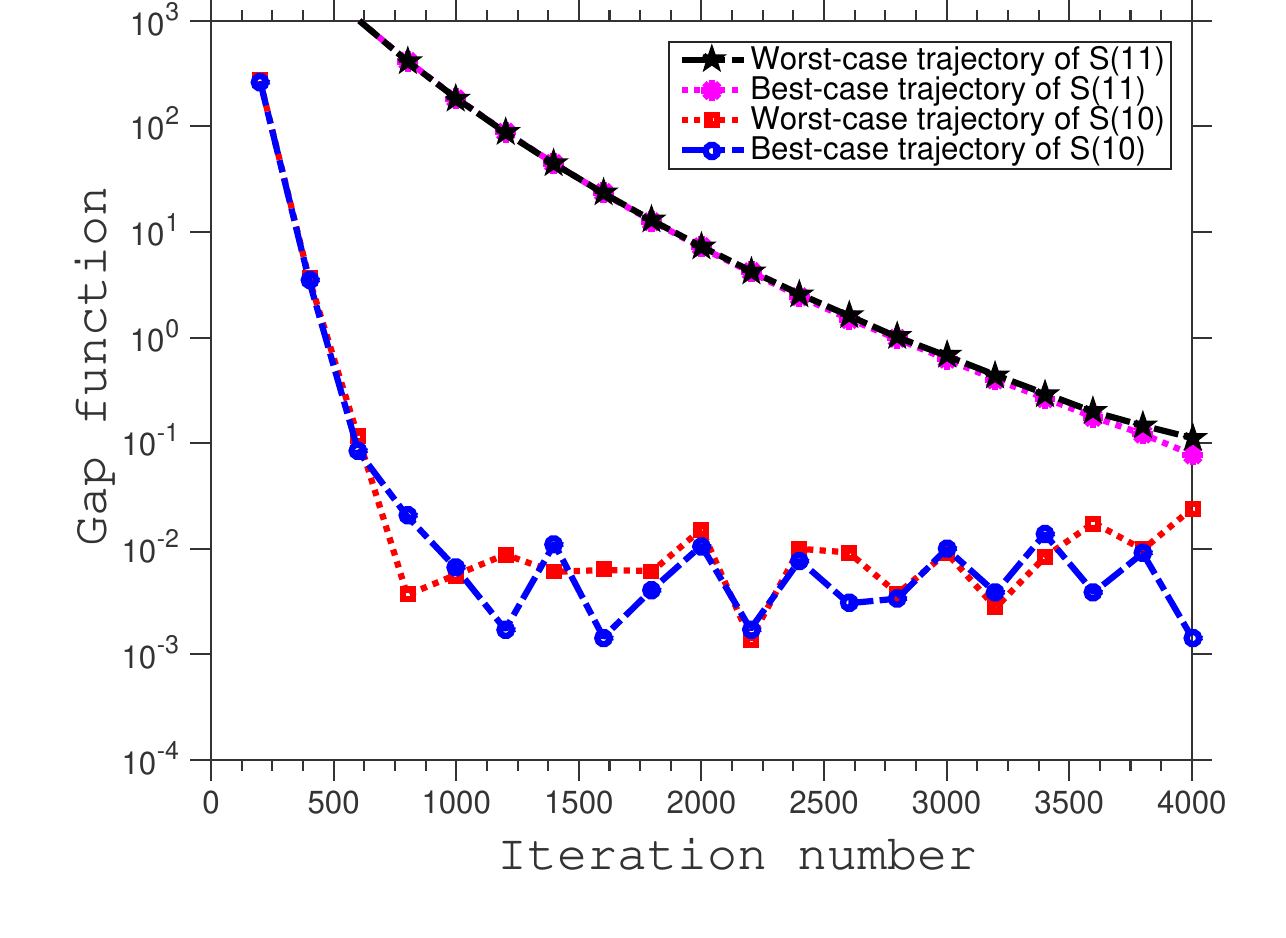}
  \vspace{-.3in}
  \caption{Convergence in mean vs. a.s. convergence: worst-case and best-case trajectories}
  \label{fig:GAP3-2}
\end{figure}

We observe from Figure~\ref{fig:GAP3}(a), that
although the mean gap function is approaching zero, the variance
across sample paths is relatively large. %This suggests that there may be sample paths that may not converge to the solution set. 
This observation is aligned with the knowledge that the choice of $(a,b,c)$
do not guarantee almost sure convergence for $S(10)$.  However, in
Figure~\ref{fig:GAP3}(b) the conditions of almost sure convergence are
met, we observe that the variance \uvs{in the gap function at the
	terminal iterate} is far
	smaller and all of the $50$ trajectories remain close to the sample
	mean. \uvs{From a practical standpoint, this suggests that almost
		every sample path will show similar performance.} \fyRev{Figure~\ref{fig:GAP3-2} illustrates the worst-case and the best-case sample paths among all of the $50$ trajectories for each of the settings $S(10)$ and $S(11)$. We observe that the variance of the gap function's value for the setting $S(10)$ is significantly larger than that of the setting $S(11)$.} 

%--------------------------------------------------------------------------------------------------------------------------------------------
\subsection{Convergence of the \fyRev{aSA$_r$ and aSA$_{\ell,r}$} schemes}\label{sec:numerics-ASA}
\uvs{Next, we compare} the performance of the averaging schemes
across different values of $r$. Motivated by \an{Proposition~\ref{prop:window}},
the stepsize used in our analysis is assumed to be of the form
$\g_k=\frac{2M}{C\sqrt{k+1}}$ where $M$ is the bound on the Euclidean
norm of $x \in X$ and $C$ represents the bound on the norm on mapping
$F$ over the set $X$. Note that here we use identical stepsizes for $r =
1$ and $r = -1$, to allow for using the \uvs{same set of} iterates generated by
the SA algorithm for both schemes. In
Table~\ref{tab:ave1}, we report the sample mean of the gap function
over $50$ samples. The rows \an{in Table~\ref{tab:ave1}} correspond to the value of the parameter
$\lambda$ which changes from $0$ to $1$ in steps of $0.1$. \an{Note that $\lambda=0$
implies that $\ell=0$, which corresponds to} the \fyRev{aSA$_r$} scheme. Moreover,
$\lambda=1$ corresponds to the SA scheme without averaging since
$\ell=N$ in this case. Cases when $\lambda$ is between $0$ and
$1$ correspond to the \fyRev{aSA$_{r,\ell}$} scheme. The columns in the table are sorted
based on the iteration number $N$ from $1000$ to $4000$. Each column
includes the results for the case that $r=-1$ and the standard choice
$r=+1$.

\begin{table}[htb] 
\tiny
\centering 
\begin{tabular}{|c|c|c c|c c|c c|c c|} 
\hline 
\multicolumn{2}{|c|}{Scheme} & \multicolumn{2}{c|}{N=1000} &  \multicolumn{2}{c|}{N=2000} &\multicolumn{2}{c|}{N=3000} &  \multicolumn{2}{c|}{N=4000} 
\\ 
\hline
-& $\lambda$&  $r=-1$ & $r=+1$ &$r=-1$ & $r=+1$ &$r=-1$ & $r=+1$ & $r=-1$ & $r=+1$ \\
\hline

aSA$_r$& $0$& $3.85$e$+1$ & $ 1.19$e$+3 $ & $ 5.44$e$+0 $ & $ 5.93$e$+2 $ & $ 1.64$e$+0 $ & $ 3.93$e$+2 $ & $ 6.94$e$-1 $ & $ 2.94$e$+2 $ \\
\hline
 & $0.1$&$1.97$e$+1 $ & $ 7.30$e$+1 $ & $ 1.49$e$+0 $ & $ 8.12$e$+0 $ & $ 2.47$e$-1 $ & $ 1.68$e$+0 $ & $ 5.96$e$-2 $ & $ 4.67$e$-1 $ \\

 & $0.2$&$1.05$e$+1 $ & $ 2.31$e$+1 $ & $ 5.01$e$-1 $ & $ 1.40$e$+0 $ & $ 5.59$e$-2 $ & $ 1.80$e$-1 $ & $ 9.66$e$-3 $ & $ 3.38$e$-2 $ \\

 & $0.3$&$6.03$e$+0 $ & $ 9.92$e$+0 $ & $ 2.01$e$-1 $ & $ 3.88$e$-1 $ & $ 1.68$e$-2 $ & $ 3.55$e$-2 $ & $ 2.51$e$-3 $ & $ 5.23$e$-3 $ \\

 & $0.4$&$3.68$e$+0 $ & $ 5.04$e$+0 $ & $ 9.16$e$-2 $ & $ 1.39$e$-1 $ & $ 6.40$e$-3 $ & $ 1.01$e$-2 $ & $ 1.14$e$-3 $ & $ 1.57$e$-3 $ \\

aSA$_{\ell,r}$ & $0.5$&$2.35$e$+0 $ & $ 2.85$e$+0 $ & $ 4.62$e$-2 $ & $ 5.97$e$-2 $ & $ 3.14$e$-3 $ & $ 3.99$e$-3 $ & $ 8.46$e$-4 $ & $ 9.25$e$-4 $ \\

 & $0.6$&$1.56$e$+0 $ & $ 1.75$e$+0 $ & $ 2.55$e$-2 $ & $ 2.94$e$-2 $ & $ 2.04$e$-3 $ & $ 2.25$e$-3 $ & $ 8.57$e$-4 $ & $ 8.71$e$-4 $ \\

 & $0.7$&$1.08$e$+0 $ & $ 1.14$e$+0 $ & $ 1.55$e$-2 $ & $ 1.66$e$-2 $ & $ 1.70$e$-3 $ & $ 1.75$e$-3 $ & $ 1.00$e$-3 $ & $ 9.96$e$-4 $ \\

 & $0.8$&$7.61$e$-1 $ & $ 7.79$e$-1 $ & $ 1.04$e$-2 $ & $ 1.06$e$-2 $ & $ 1.75$e$-3 $ & $ 1.77$e$-3 $ & $ 1.26$e$-3 $ & $ 1.26$e$-3 $ \\

 & $0.9$&$5.52$e$-1 $ & $ 5.55$e$-1 $ & $ 7.64$e$-3 $ & $ 7.69$e$-3 $ & $ 1.84$e$-3 $ & $ 1.84$e$-3 $ & $ 1.60$e$-3 $ & $ 1.60$e$-3 $ \\
\hline
SA & $1$&$4.12$e$-1 $ & $ 4.12$e$-1 $ & $ 6.04$e$-3 $ & $ 6.04$e$-3 $ & $ 2.07$e$-3 $ & $ 2.07$e$-3 $ & $ 2.30$e$-3 $ & $ 2.30$e$-3 $ \\
\hline
\end{tabular} 
\caption{Gap function's comparison between  $r=-1$ and $r=+1$.} 
\label{tab:ave1} 
\end{table}
Naturally, when $\lambda = 1$, there is no averaging and in the last row
of the table, the gap value is identical for both $r=-1$ and $r=+1$. Importantly, we see
that for most values of $N$ and $\lambda$, both the \fyRev{aSA$_r$ scheme and the aSA$_{\ell,r}$}  scheme have
lower gaps when $r=-1$ \uvs{compared with $r=1$}. \uvs{In fact}, when
$\lambda$ is small \uvs{and the averaging window is large}, this difference becomes
\uvs{even more pronounced}. For example, \uvs{in} the case that $\lambda=0$ and $N=1000$, the
gap value for aSA$_{\ell,r}$ scheme with $r=-1$ is about $39$, while this gap 
is nearly two orders of magnitude larger at $1190$ when $r=+1$. We show this difference in Figure \ref{fig:GAP1}
and Figure \ref{fig:GAP2}. It can be seen that when comparing both
averaging schemes, both \fyRev{aSA$_r$ and aSA$_{\ell,r}$} have a lower gap
for $r = -1$ vs $r = 1$ for any of the examined values of $N$. 
%Recall
	%that the original motivation of  averaging schemes lay in developing
		%a higher level of robustness to the underlying randomness; 
		The
		smaller the value of $\lambda$, the \uvs{larger the window over
			which averaging is carried out}, implying the
	more robustness of the SA scheme. Therefore, there is a trade-off
	between increasing $\lambda$ and the robustness of SA scheme. When
	$\lambda$ is large, although the difference between the performance
	of $r=-1$ and $r=+1$ becomes small, \an{the case $r=-1$ almost always has a smaller gap value
	than the case $r=1$, as shown in Table~\ref{tab:ave1}. }

\begin{figure}[htb]
 \centering
 \subfloat[$\lambda=0$]{\label{fig:lambda0}\includegraphics[scale=.50,  angle=0]
 {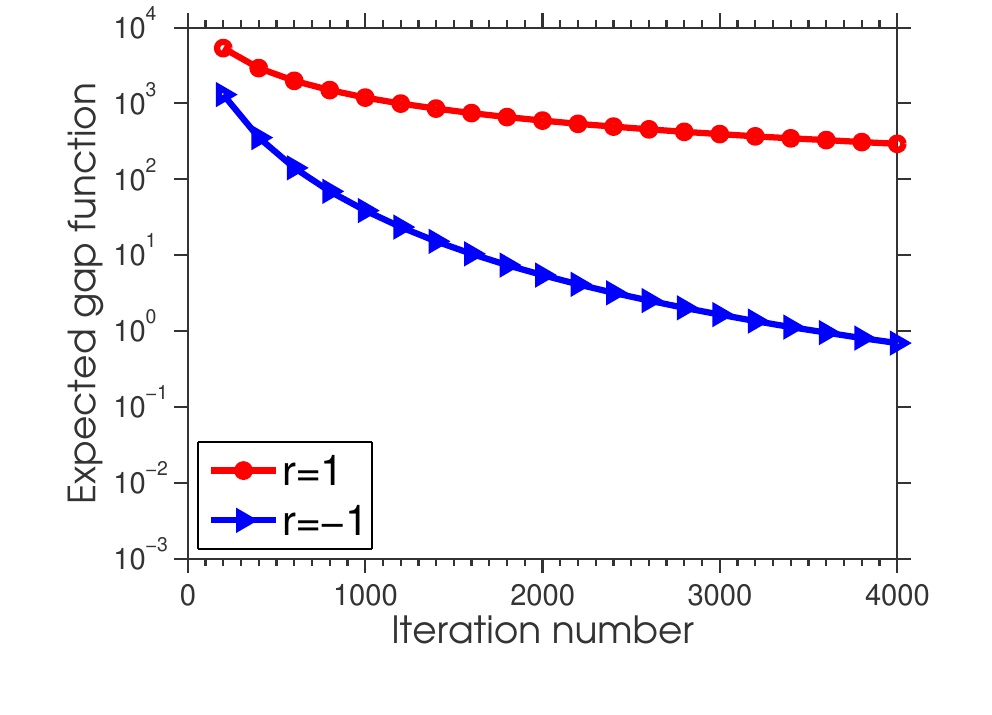}}
 \subfloat[$\lambda=0.1$]{\label{fig:lambda1}\includegraphics[scale=.50, angle=0]{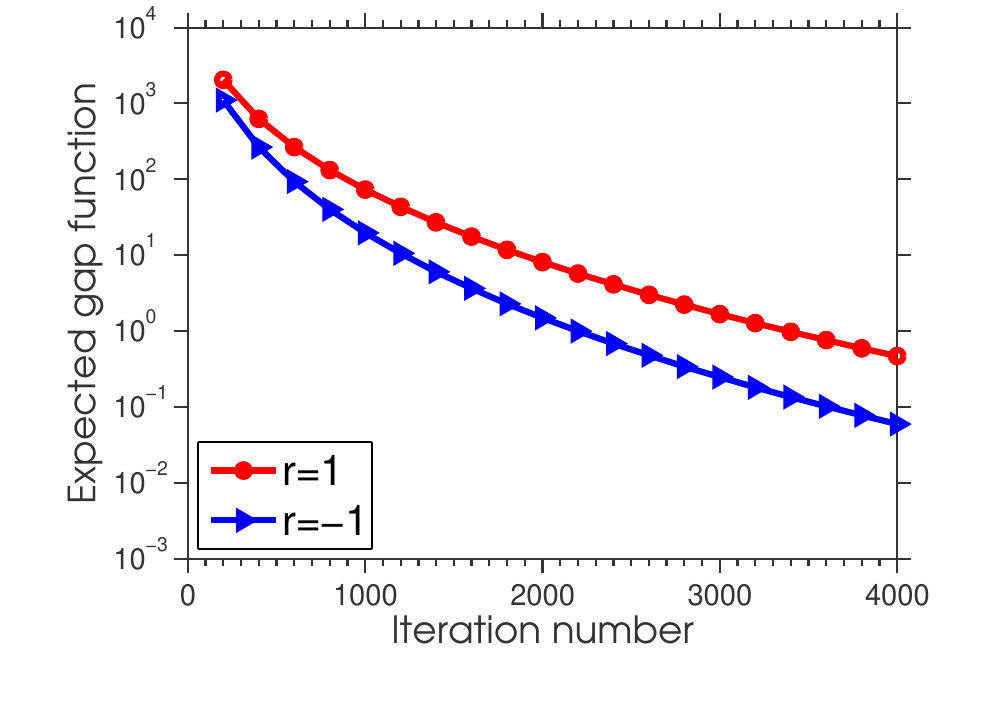}}
\caption{Gap function: \fyRev{aSA$_r$ scheme (l)  and
	aSA$_{\ell,r}$ }schemes with $\lambda=0.1$ (r)}
\vspace{-0.2in}
\label{fig:GAP1}
\end{figure}
\begin{figure}[htb]%\vspace{-0.2in}
 \centering
 \subfloat[$\lambda=0.2$]{\label{fig:lambda2}\includegraphics[scale=.50,  angle=0]
 {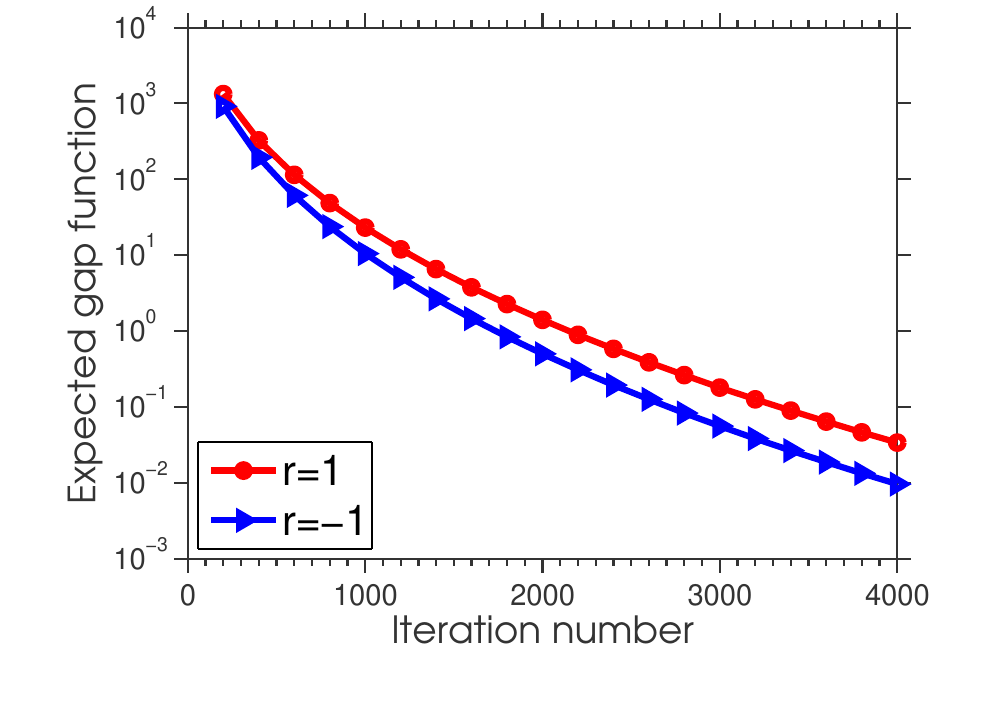}}
 \subfloat[$\lambda=0.3$]{\label{fig:lambda3}\includegraphics[scale=.50, angle=0]{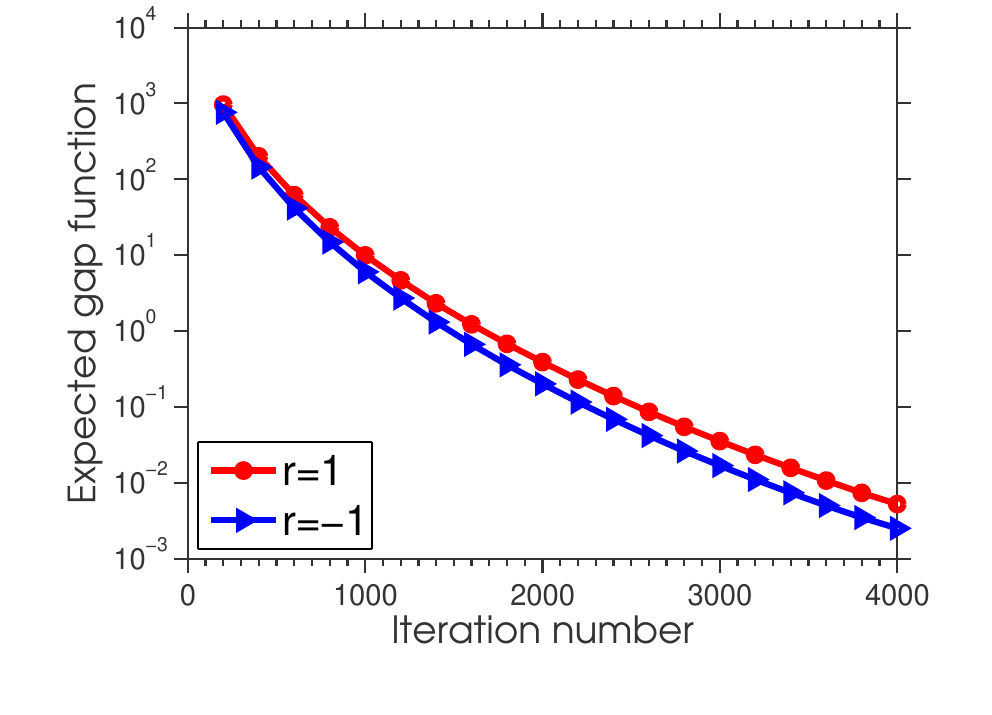}}

\caption{Gap function: \fyRev{aSA$_{r,\ell}$} scheme with $\lambda = 0.2$ (l)
	and $\lambda=0.3$.}
\label{fig:GAP2}
\vspace{-0.2in}
\end{figure}
 One question that may arise \an{is the optimal choice} of $\lambda$
 \uvs{in the design of} the aSA$_{\ell,r}$ scheme. We observe that the
 answer to this question depends on $N$. When $N$ is small, in this case
 $1000$, the SA scheme (\uvs{essentially no averaging}) \uvs{displays
	 the minimal } error implying that $\lambda=1$ performs the best.
	 However, for larger values of $N$, \uvs{the minimal} error occurs
	 \uvs{at a } smaller $\lambda$ and the larger the value of $N$, the smaller the value of $\lambda$. For example, at $N=4000$, $\lambda=0.5$ has the smallest error.
\subsection{Sensitivity analysis}\label{sec:numerics-ASA2}
In this section we investigate the performance of the averaging schemes
when some parameters of the Cournot game change. First, we increase the
number of firms from $5$ to $15$ and maintain other parameters fixed.
Table~\ref{tab:ave2} shows the simulation results for the new problem
with $15$ firms. We observe that the results are similar to the case
where ${\cal I}=5$. Importantly, \us{for} almost any $N$ and any $\lambda <1$,
	  the averaging schemes perform better with $r=-1$ than with $r=+1$.
	  Specifically, when $\lambda$ is small this difference is
	  significant.  
\begin{table}[htb] \tiny  
\centering 
\begin{tabular}{|c|c|c c|c c|c c|c c|} 
\hline 
\multicolumn{2}{|c|}{Scheme} & \multicolumn{2}{c|}{N=1000} &  \multicolumn{2}{c|}{N=2000} &\multicolumn{2}{c|}{N=3000} &  \multicolumn{2}{c|}{N=4000} 
\\ 
\hline
-& $\lambda$&  $r=-1$ & $r=+1$ &$r=-1$ & $r=+1$ &$r=-1$ & $r=+1$ & $r=-1$ & $r=+1$ \\
\hline
aSA$_{r}$ &$	0	$ & $	5.78$e$+1	$ & $	1.47$e$+3	$ & $	8.43$e$+0	$ & $	7.33$e$+2	$ & $	2.56$e$+0	$ & $	4.86$e$+2	$ & $	1.09$e$+0	$ & $	3.63$e$+2	$ \\ \hline
&$	0.1	$ & $	3.19$e$+1	$ & $	1.11$e$+2	$ & $	2.67$e$+0	$ & $	1.36$e$+1	$ & $	4.75$e$-1	$ & $	3.02$e$+0	$ & $	1.22$e$-1	$ & $	8.93$e$-1	$ \\
&$	0.2	$ & $	1.81$e$+1	$ & $	3.80$e$+1	$ & $	9.91$e$-1	$ & $	2.66$e$+0	$ & $	1.23$e$-1	$ & $	3.79$e$-1	$ & $	2.29$e$-2	$ & $	7.73$e$-2	$ \\
&$	0.3	$ & $	1.09$e$+1	$ & $	1.74$e$+1	$ & $	4.29$e$-1	$ & $	8.05$e$-1	$ & $	4.04$e$-2	$ & $	8.32$e$-2	$ & $	6.18$e$-3	$ & $	1.31$e$-2	$ \\
&$	0.4	$ & $	6.90$e$+0	$ & $	9.29$e$+0	$ & $	2.07$e$-1	$ & $	3.09$e$-1	$ & $	1.59$e$-2	$ & $	2.50$e$-2	$ & $	2.38$e$-3	$ & $	3.60$e$-3	$ \\
aSA$_{\ell,r}$&$	0.5	$ & $	4.56$e$+0	$ & $	5.47$e$+0	$ & $	1.09$e$-1	$ & $	1.39$e$-1	$ & $	7.51$e$-3	$ & $	9.69$e$-3	$ & $	1.36$e$-3	$ & $	1.64$e$-3	$ \\
&$	0.6	$ & $	3.12$e$+0	$ & $	3.46$e$+0	$ & $	6.12$e$-2	$ & $	7.06$e$-2	$ & $	4.29$e$-3	$ & $	4.87$e$-3	$ & $	1.03$e$-3	$ & $	1.09$e$-3	$ \\
&$	0.7	$ & $	2.20$e$+0	$ & $	2.32$e$+0	$ & $	3.67$e$-2	$ & $	3.95$e$-2	$ & $	3.01$e$-3	$ & $	3.18$e$-3	$ & $	1.02$e$-3	$ & $	1.03$e$-3	$ \\
&$	0.8	$ & $	1.59$e$+0	$ & $	1.63$e$+0	$ & $	2.35$e$-2	$ & $	2.42$e$-2	$ & $	2.42$e$-3	$ & $	2.47$e$-3	$ & $	1.21$e$-3	$ & $	1.20$e$-3	$ \\
&$	0.9	$ & $	1.18$e$+0	$ & $	1.18$e$+0	$ & $	1.58$e$-2	$ & $	1.59$e$-2	$ & $	2.32$e$-3	$ & $	2.32$e$-3	$ & $	1.55$e$-3	$ & $	1.55$e$-3	$ \\ \hline
SA&$	1	$ & $	8.89$e$-1	$ & $	8.89$e$-1	$ & $	1.22$e$-2	$ & $	1.22$e$-2	$ & $	2.71$e$-3	$ & $	2.71$e$-3	$ & $	2.04$e$-3	$ & $	2.04$e$-3	$ \\
\hline
\end{tabular} 
\caption{Comparison of \an{gap} function for ${\cal I} = 15$.}   
\label{tab:ave2}  
\vspace{-0.1in}
\end{table}
Next, we  assume that $x_0$ for every sample path is a point where
$s_{ij}=g_{ij}=150$ for any $i,j$, rather than the origin.
Table \ref{tab:ave4} provides the simulation results for this case. The
performance of all the schemes is similar to the original setting. We
observe that the averaging schemes have a smaller expected gap function
when $r=-1$ than when $r=+1$. 
\begin{table}[htb] \tiny \centering 
\begin{tabular}{|c|c|c c|c c|c c|c c|} 
\hline 
\multicolumn{2}{|c|}{Scheme} & \multicolumn{2}{c|}{N=1000} &  \multicolumn{2}{c|}{N=2000} &\multicolumn{2}{c|}{N=3000} &  \multicolumn{2}{c|}{N=4000} 
\\ 
\hline
-& $\lambda$&  $r=-1$ & $r=+1$ &$r=-1$ & $r=+1$ &$r=-1$ & $r=+1$ & $r=-1$ & $r=+1$ \\
\hline

aSA$_{r}$&$	0	$ & $	3.38$e$+1	$ & $	1.24$e$+3	$ & $	4.63$e$+0	$ & $	5.72$e$+2	$ & $	1.34$e$+0	$ & $	3.56$e$+2	$ & $	5.42$e$-1	$ & $	2.52$e$+2	$ \\ \hline
&$	0.1	$ & $	1.76$e$+1	$ & $	6.50$e$+1	$ & $	1.32$e$+0	$ & $	7.22$e$+0	$ & $	2.20$e$-1	$ & $	1.49$e$+0	$ & $	5.35$e$-2	$ & $	4.17$e$-1	$ \\
&$	0.2	$ & $	9.36$e$+0	$ & $	2.05$e$+1	$ & $	4.43$e$-1	$ & $	1.25$e$+0	$ & $	4.95$e$-2	$ & $	1.60$e$-1	$ & $	8.78$e$-3	$ & $	3.00$e$-2	$ \\
&$	0.3	$ & $	5.38$e$+0	$ & $	8.85$e$+0	$ & $	1.76$e$-1	$ & $	3.42$e$-1	$ & $	1.49$e$-2	$ & $	3.11$e$-2	$ & $	2.39$e$-3	$ & $	4.74$e$-3	$ \\
&$	0.4	$ & $	3.28$e$+0	$ & $	4.50$e$+0	$ & $	7.92$e$-2	$ & $	1.21$e$-1	$ & $	5.73$e$-3	$ & $	8.85$e$-3	$ & $	1.16$e$-3	$ & $	1.54$e$-3	$ \\
aSA$_{\ell,r}$&$	0.5	$ & $	2.10$e$+0	$ & $	2.54$e$+0	$ & $	3.94$e$-2	$ & $	5.12$e$-2	$ & $	2.91$e$-3	$ & $	3.62$e$-3	$ & $	9.02$e$-4	$ & $	9.84$e$-4	$ \\
&$	0.6	$ & $	1.39$e$+0	$ & $	1.56$e$+0	$ & $	2.15$e$-2	$ & $	2.49$e$-2	$ & $	2.01$e$-3	$ & $	2.19$e$-3	$ & $	8.94$e$-4	$ & $	9.09$e$-4	$ \\
&$	0.7	$ & $	9.56$e$-1	$ & $	1.01$e$+0	$ & $	1.29$e$-2	$ & $	1.38$e$-2	$ & $	1.76$e$-3	$ & $	1.81$e$-3	$ & $	1.03$e$-3	$ & $	1.02$e$-3	$ \\
&$	0.8	$ & $	6.74$e$-1	$ & $	6.90$e$-1	$ & $	8.62$e$-3	$ & $	8.84$e$-3	$ & $	1.78$e$-3	$ & $	1.80$e$-3	$ & $	1.29$e$-3	$ & $	1.28$e$-3	$ \\
&$	0.9	$ & $	4.89$e$-1	$ & $	4.92$e$-1	$ & $	6.47$e$-3	$ & $	6.50$e$-3	$ & $	1.80$e$-3	$ & $	1.81$e$-3	$ & $	1.62$e$-3	$ & $	1.62$e$-3	$ \\ \hline
SA&$	1	$ & $	3.60$e$-1	$ & $	3.60$e$-1	$ & $	5.79$e$-3	$ & $	5.79$e$-3	$ & $	2.05$e$-3	$ & $	2.05$e$-3	$ & $	2.32$e$-3	$ & $	2.32$e$-3	$ \\
\hline
\end{tabular} 
\caption{Comparison of \an{gap} function for different starting points.} 
\label{tab:ave4}
\vspace{-0.1in}
\end{table}
\us{Lastly, we are interested in observing the performance of the averaging
schemes for other choices of $r$. In \an{Proposition~\ref{prop:optimal-rate}(b)}, we
showed that the optimal rate of convergence is attained when $r<1$. Our
goal is to compare the case $r=+1$ with two other cases where $r=-0.5$
and $r=+0.5$. In this study, we used the original settings of
parameters. Table \ref{tab:ave5} presents the results of this
simulation. Interestingly, comparing these results with those in Table
\ref{tab:ave1}, we see that both cases $r=-0.5$ and $r=+0.5$ have a
superior performance to $r=+1$. It is worth noting that when $\lambda
\leq 0.6$, $r = -0.5$ tends to perform better than $r = 0.5$. A natural question that emerges is the
best choice of $r$. While one may conjecture that that when $r<1$, the
performance of the averaging scheme improves as $r$ tends to $-\infty$,
			\uvs{this may not be true.} Consider a setting when $r$ goes to
				$-\infty$. Consequently, $\bar
x_{N}$ tends to $x_{N-1}$ implying the SA scheme represents the case
with $r=-\infty$. However, for example in Table \ref{tab:ave1}, when
$r=-1$ and $\lambda=0.5$ the \fyRev{aSA$_{\ell,r}$} scheme performs better than the SA
scheme. Therefore, decreasing $r$ \uvs{may} not necessarily speed up the
convergence of the gap {function}.  Finding the best choice for $r$ requires
more analysis and remains the subject of future research.}
\begin{table}[htb] 
\centering  \tiny 
\begin{tabular}{|c|c|c c|c c|c c|c c|} 
\hline 
\multicolumn{2}{|c|}{Scheme} & \multicolumn{2}{c|}{N=1000} &  \multicolumn{2}{c|}{N=2000} &\multicolumn{2}{c|}{N=3000} &  \multicolumn{2}{c|}{N=4000} 
\\ 
\hline
-& $\lambda$&  $r=-0.5$ & $r=+0.5$ &$r=-0.5$ & $r=+0.5$ &$r=-0.5$ & $r=+0.5$ & $r=-0.5$ & $r=+0.5$ \\
\hline

aSA$_{r}$&$ 0 $ & $ 7.38$e$+1 $ & $ 4.13$e$+2 $ & $ 1.40$e$+1 $ & $ 1.48$e$+2 $ & $ 5.13$e$+0 $ & $ 8.07$e$+1 $ & $ 2.50$e$+0 $ & $ 5.24$e$+1 $ \\ \hline
&$ 0.1 $ & $ 2.72$e$+1 $ & $ 5.28$e$+1 $ & $ 2.30$e$+0 $ & $ 5.41$e$+0 $ & $ 4.07$e$-1 $ & $ 1.07$e$+0 $ & $ 1.02$e$-1 $ & $ 2.89$e$-1 $ \\
&$ 0.2 $ & $ 1.28$e$+1 $ & $ 1.90$e$+1 $ & $ 6.53$e$-1 $ & $ 1.10$e$+0 $ & $ 7.58$e$-2 $ & $ 1.37$e$-1 $ & $ 1.34$e$-2 $ & $ 2.51$e$-2 $ \\
&$ 0.3 $ & $ 6.84$e$+0 $ & $ 8.78$e$+0 $ & $ 2.38$e$-1 $ & $ 3.31$e$-1 $ & $ 2.04$e$-2 $ & $ 2.96$e$-2 $ & $ 3.01$e$-3 $ & $ 4.35$e$-3 $ \\
&$ 0.4 $ & $ 3.98$e$+0 $ & $ 4.66$e$+0 $ & $ 1.02$e$-1 $ & $ 1.26$e$-1 $ & $ 7.16$e$-3 $ & $ 8.99$e$-3 $ & $ 1.23$e$-3 $ & $ 1.44$e$-3 $ \\
aSA$_{\ell,r}$&$ 0.5 $ & $ 2.47$e$+0 $ & $ 2.72$e$+0 $ & $ 4.93$e$-2 $ & $ 5.60$e$-2 $ & $ 3.33$e$-3 $ & $ 3.75$e$-3 $ & $ 8.61$e$-4 $ & $ 9.00$e$-4 $ \\
&$ 0.6 $ & $ 1.61$e$+0 $ & $ 1.70$e$+0 $ & $ 2.64$e$-2 $ & $ 2.83$e$-2 $ & $ 2.09$e$-3 $ & $ 2.19$e$-3 $ & $ 8.58$e$-4 $ & $ 8.65$e$-4 $ \\
&$ 0.7 $ & $ 1.09$e$+0 $ & $ 1.12$e$+0 $ & $ 1.57$e$-2 $ & $ 1.63$e$-2 $ & $ 1.72$e$-3 $ & $ 1.74$e$-3 $ & $ 1.00$e$-3 $ & $ 9.96$e$-4 $ \\
&$ 0.8 $ & $ 7.66$e$-1 $ & $ 7.74$e$-1 $ & $ 1.04$e$-2 $ & $ 1.06$e$-2 $ & $ 1.76$e$-3 $ & $ 1.77$e$-3 $ & $ 1.26$e$-3 $ & $ 1.26$e$-3 $ \\
&$ 0.9 $ & $ 5.53$e$-1 $ & $ 5.54$e$-1 $ & $ 7.65$e$-3 $ & $ 7.68$e$-3 $ & $ 1.84$e$-3 $ & $ 1.84$e$-3 $ & $ 1.60$e$-3 $ & $ 1.60$e$-3 $ \\ \hline
SA&$ 1 $ & $ 4.12$e$-1 $ & $ 4.12$e$-1 $ & $ 6.04$e$-3 $ & $ 6.04$e$-3 $ & $ 2.07$e$-3 $ & $ 2.07$e$-3 $ & $ 2.30$e$-3 $ & $ 2.30$e$-3 $ \\

\hline
\end{tabular} 
\caption{Comparison of gap function:  $r=-0.5$ and $r=+0.5$.} 
\label{tab:ave5}  
\end{table}

\section{Concluding remarks}\label{sec:concl}
{We consider a stochastic variational inequality problem with
monotone and possibly non-Lipschitzian maps over a closed, convex, and
compact set. Much of the past research aimed at deriving almost sure
convergence of the iterates has required strong monotonicity and Lipschitz continuity of
the map. In the first part of the paper, by conducting a simultaneous
smoothing and regularization of the map, we develop a regularized smoothing
	stochastic approximation (RSSA) scheme. By updating the smoothing
	parameter, regularization parameter, and the steplength sequence
	after every iteration at suitably defined rates, the generated
	sequence can be shown to converge almost surely to the solution of
	the original problem. Unfortunately, such a scheme does not
	immediately admit a non-asymptotic rate statement, motivating the
	development of an averaging-based scheme.  In the second part of the
	paper, we \uvs{generalize standard averaging-based SA schemes to a
		setting} where the weights of the averaged sequence are
	parameterized in terms of a constant $r$ which is known to be $1$
	in the classic averaging methods. \uvs{We show that when $r < 1$,
		the mean of the gap function diminishes to zero at a rate of
			${\cal O}(1/K^{(1/6)-\delta})$ while the optimal rate of ${\cal O}(1/\sqrt{K})$ is recovered when smoothing and regularization is suppressed. Note that the latter rate is superior to the sub-optimal rate of ${\cal O}(\ln(K)/\sqrt{K})$ seen in standard averaging
		schemes with $r=1$.} \uvs{Furthermore}, a window-based averaging method using
	$r<1$ is \us{also} shown
	to recover the optimal convergence rate. \uvs{Numerical experiments on a
	classical Nash-Cournot game provide several insights. First, we note
	that the RSSA schemes with a.s. convergence guarantees produce
	sequences that display far less variance in terms of the gap
	function in contrast with their counterparts that arise from
	guarantees of convergence in the mean. Second, significant benefits of using
	$r < 1$ are observed in comparison with $r = 1$ particularly when
	$\lambda$ is smaller. Yet, much remains to be
	understood regarding the optimal (or good) choices of $r$, given how crucial a
	role it plays in the empirical performance.}
%	displays superiority in  the performance of
%	the averaging schemes when $r<1$. Moreover, employing regularization
%	and smoothing in the averaging SA scheme, the empirical behavior of
%	the averaged sequence appears to support the claim of  almost sure convergence
%%	of the entire sequence.  

\section*{Appendix}

{\noindent \bf Proof of Lemma~\ref{lemma:gap-positive}.}
\begin{proof}
(a)\ {We start by showing part (i)}. Let $x \in X$ be \an{arbitrary. Then, we} have  
\[\hbox{G}(x)= \sup_{y \in X} F(y)^T(x-y) \geq F(z)^T(x-z), \quad \hbox{for any }z \in X.\]
For $z=x$, the preceding inequality implies that $\hbox{G}(x) \geq F(x)^T(x-x)=0.$
Therefore, the gap function (\ref{equ:gapf}) is nonnegative for any $x \in X$, {thus showing part (i). 
To prove part (ii), assume that $x^* \in X^*_w$}. Relation (\ref{equ:weak}) implies that 
\[F(y)^T(x^*_w-y) \leq 0, \qquad \hbox{for any } y \in X. \]
Invoking the definition of $\hbox{G}$ in (\ref{equ:gapf}), from preceding inequality we have 
\[\hbox{G}(x^*_w) \leq 0, \qquad \hbox{for any } y \in X. \]
However, since $x^*_w \in X$, Lemma \ref{lemma:gap-positive}(a)
	indicates that $\hbox{G}(x^*_w) \geq 0$. Therefore, we conclude that
	for any $x^* \in X^*_w$, we have $\hbox{G}(x^*)=0$. Now assume that
	$\hbox{G}(x)=0$ for some $x \in X$. Therefore, $\sup_{y \in X}
	F(y)^T(x-y)=0$ implying that $ F(y)^T(x-y)\leq 0$ for any $y \in X$.
	Equivalently, we have $ F(y)^T(y-x)\geq 0$ for any $y \in X$. This
	implies that $x \in X^*_w$.

\noindent(b)(i) \ 
Let $\{u_k\} \subset \Real^n$ be an arbitrary sequence in $X$ such that $\lim_{k \rightarrow \infty} u_k =u_0$. 
Since $X$ is a closed set, we have $u_0 \in X$. We want to show that $\lim_{k \rightarrow \infty}\hbox{G}(u_k)=\hbox{G}(u_0)$. We show this relation in two steps. First, using relation (\ref{equ:gapf}), for any $k \geq 0$, we have
\begin{align}\label{ineq:gap-continuity}
\hbox{G}(u_k)&= \sup_{y \in X} F(y)^T(u_k-y) = \sup_{y \in X}
F(y)^T(u_k-u_0+u_0-y)\notag \\ 
&=\sup_{y \in X} \left(F(y)^T(u_0-y)+F(y)^T(u_k-u_0)\right) \notag \\ & \leq \sup_{y \in X} F(y)^T(u_0-y) +\sup_{y \in X} F(y)^T(u_k-u_0),
\end{align}
where in the second relation we add and subtract $u_0$, and in the last
relation we used the well-known inequality $\sup_A \left(f+g\right) \leq
\sup_A f + \sup_A g$ for any two real valued functions $f$ and $g$
defined on the set $A$. Using the Cauchy-Schwarz inequality and relation (\ref{ineq:gap-continuity}), we obtain for any $k \geq 0$,
\begin{align*}
\hbox{G}(u_k)
&\leq \sup_{y \in X} F(y)^T(u_0-y) +\sup_{y \in X}\left( \|F(y)\|\|u_k-u_0\|\right)\cr 
&= \sup_{y \in X} F(y)^T(u_0-y) +\|u_k-u_0\| \sup_{y \in X} \|F(y)\| \leq \sup_{y \in X} F(y)^T(u_0-y) +C\|u_k-u_0\|, 
\end{align*}
where in the last inequality we used the boundedness assumption of the
mapping $F$ over the set $X$. Taking \us{the limit superiors} on  both sides of the
	preceding inequality, we obtain
\begin{align}\label{ineq:gap-continuity2}
\an{\limsup_{k \rightarrow \infty}}\, \hbox{G}(u_k)
&\leq \an{\limsup_{k \rightarrow \infty}} \left (\sup_{y \in X} F(y)^T(u_0-y) +C\|u_k-u_0\|\right)\cr 
&= \sup_{y \in X} F(y)^T(u_0-y)+ C\lim_{k \rightarrow \infty}  \|u_k-u_0\|= \hbox{G}(u_0),
\end{align}
where the last relation is obtained by \us{recalling} that $u_0$ is the
limit point of the sequence $\{u_k\}$. In the second step of the proof
for continuity of $\hbox{G}(x)$, using relation (\ref{equ:gapf}), for
any $y \in X$ and any $k \geq 0$, we have $\hbox{G}(u_k) \geq
F(y)^T(u_k-y)$. Let $v \in \Real^n$ be an arbitrary fixed vector in $X$.
Therefore, the preceding inequality holds for $y=v$, i.e., 
\begin{align*}
\hbox{G}(u_k) \geq F(v)^T(u_k-v).
\end{align*}
Taking \an{the limit inferiors in} both sides of the preceding inequality when $k$ goes to infinity, we have 
\begin{align*}
\an{\liminf_{k \rightarrow \infty}}\, \hbox{G}(u_k) 
&\geq \an{\liminf_{k \rightarrow \infty}} F(v)^T(u_k-v) = F(v)^T\left(\lim_{k \rightarrow \infty}(u_k)-v\right)
= F(v)^T(u_0-v).
\end{align*}
Since the preceding relation holds for any arbitrary $v \in X$, taking supremum from the right-hand side and using the relation (\ref{equ:gapf}) we obtain 
\begin{align}\label{ineq:gap-continuity3}
\an{\liminf_{k \rightarrow \infty}}\, \hbox{G}(u_k) &\geq \sup_{v\in X} F(v)^T(u_0-v)=\hbox{G}(u_0).
\end{align}
From (\ref{ineq:gap-continuity2}) and (\ref{ineq:gap-continuity3}), we
conclude that the gap function $\hbox{G}(x)$ is continuous at any $x \in
X$.

\noindent (b)(ii) \ For any $x,y \in X$ we have 
\begin{align*}
F(y)^T(x-y) &\leq  \|F(y)\|\|x-y\| \leq \|F(y)\|(\|x\| + \|y\|) \leq 2CM,
\end{align*}
where the first, second, and third inequalities follow from  the Cauchy-Schwarz
inequality,  the triangle inequality, and the boundedness assumption on
	the mapping $F$ and the set $X$. Taking the supremum {over $y\in X$ in the
	preceding relation and by using~\eqref{equ:gapf}}, we obtain the desired result.
\end{proof}

\fyRev{
{\noindent \bf Proof of Lemma \ref{lemma:ineqHarmonic}.} 
\begin{proof}
(a) Consider the function $h(x)=\frac{1}{x}$ for \an{$x>0$}. 
Since $h(x)$ is a continuous decreasing function, we have 
\[\sum_{k=\ell}^{N-1}\frac{1}{k+1}=\sum_{k=\ell+1}^N\frac{1}{k} =\frac{1}{\ell+1}+ \sum_{k=\ell+2}^N\frac{1}{k} \leq \frac{1}{\ell+1} +\int_{\ell+1}^N\frac{1}{x}\ dx = \frac{1}{\ell+1} +\ln\left(\frac{N}{\ell+1}\right).\]
Also, we \uvs{may} write 
\[\sum_{k=\ell}^{N-1}\frac{1}{k+1}=\sum_{k=\ell+1}^N\frac{1}{k}  \geq \int_{\ell+1}^{N+1}\frac{1}{x}\ dx =\ln\left(\frac{N+1}{\ell+1}\right).\]
Therefore, the desired result in part (a) holds.\\
(b) Let us define $g(x)= x^\alpha$ for $x>0$. Consider the case where $\alpha >0$. This implies that $g(x)$ is an increasing function and we can write
\[\sum_{k=\ell}^{N-1}(k+1)^\alpha=\sum_{k=\ell+1}^N k^\alpha \leq \int_{\ell+1}^{N+1}x^\alpha \leq (\ell+1)^\alpha+\frac{(N+1)^{\alpha+1}-(\ell+1)^{\alpha+1}}{\alpha+1},\]
and 
\[\sum_{k=\ell}^{N-1}(k+1)^\alpha=\sum_{k=\ell+1}^N k^\alpha \geq \int_{\ell}^{N}x^\alpha \ dx\geq \int_{\ell+1}^{N}x^\alpha \ dx =\frac{N^{\alpha+1}-(\ell+1)^{\alpha+1}}{\alpha+1}. \]
Therefore the result of part (b) holds for $\alpha >0$.
Now, we consider the case where $\alpha <0$ and $\alpha\neq -1$. This implies that $g(x)$ defined in part (b) is a decreasing function and we can write 
\begin{align*}&\sum_{k=\ell}^{N-1}(k+1)^\alpha=(\ell+1)^\alpha+\sum_{k=\ell+2}^N k^\alpha \leq (\ell+1)^\alpha+\int_{\ell+1}^{N}x^\alpha \ dx\leq (\ell+1)^\alpha+\int_{\ell+1}^{N+1}x^\alpha \ dx\\ &=(\ell+1)^\alpha+\frac{(N+1)^{\alpha+1}-(\ell+1)^{\alpha+1}}{\alpha+1},\end{align*}
and 
\[\sum_{k=\ell}^{N-1}(k+1)^\alpha=\sum_{k=\ell+1}^N k^\alpha \geq \int_{\ell+1}^{N+1}x^\alpha \ dx\geq \int_{\ell+1}^{N}x^\alpha \ dx =\frac{N^{\alpha+1}-(\ell+1)^{\alpha+1}}{\alpha+1}. \]
Therefore the result of part (b) holds for $\alpha <0$ and $\alpha\neq -1$.
\end{proof}}

\vspace{-0.1in}
\bibliographystyle{amsplain}
\bibliography{wsc11-v02,demobib}
%%%%%%%%%%%%%%%%%
\end{document}